\documentclass[10pt, a4paper,reqno]{amsart}
\usepackage{amsthm,euscript,colonequals}
\usepackage{stmaryrd}%need for \mapsfrom

\theoremstyle{}
{\theoremstyle{definition}
\newtheorem{dfn}{Definition}[section]}
\newtheorem{prop}[dfn]{Proposition}
\newtheorem{nota}[dfn]{Notation}
\newtheorem{thm}[dfn]{Theorem}
{\theoremstyle{definition}
\newtheorem{rem}[dfn]{Remark}}
\newtheorem{lem}[dfn]{Lemma}
\newtheorem{cor}[dfn]{Corollary}

{\theoremstyle{definition}
\newtheorem{exa}[dfn]{Example}}

\usepackage[top=30truemm,bottom=30truemm,left=35truemm,right=35truemm]{geometry}

\usepackage{amsmath,amssymb,amscd,bm,graphicx,tikz-cd,upgreek,dsfont,amsfonts,tikz,longtable}
\usetikzlibrary{cd,calc}
\usepackage[colorlinks]{hyperref}

\tikzset{
        DB/.style={circle,draw=black,circle,fill=white,inner sep=0pt, minimum size=5pt},
        DW/.style={circle,draw=black,fill=black,inner sep=0pt, minimum size=5pt},
        cvertex/.style={circle,draw=black,fill=white,inner sep=1pt,outer sep=3pt},
        vertex/.style={circle,fill=black,inner sep=1pt,outer sep=3pt},
        star/.style={circle,fill=yellow,inner sep=0.75pt,outer sep=0.75pt},
        tvertex/.style={inner sep=1pt,font=\scriptsize},
	pvertex/.style={circle,inner sep=1pt,outer sep=2pt,font=\scriptsize},
        gap/.style={inner sep=0.5pt,fill=white}
}

\usepackage{enumitem}
\setlist[enumerate]{format=\normalfont}

\usepackage{array,multirow,booktabs,longtable}
\numberwithin{equation}{section}

\newcommand{\ii}{\kern 1pt {\rm i}\kern 1pt }

\newcommand{\Spec}{\operatorname{Spec}}
\newcommand{\Supp}{\operatorname{Supp}}
\newcommand{\Hom}{\operatorname{Hom}}
\newcommand{\Mor}{\operatorname{Mor}}
\newcommand{\Ext}{\operatorname{Ext}}
\newcommand{\Tor}{\operatorname{Tor}}
\newcommand{\RHom}{\operatorname{{\mathbf R}Hom}}
\newcommand{\End}{\operatorname{End}}
\newcommand{\Pic}{\operatorname{Pic}}

\newcommand{\Auteq}{\operatorname{Auteq}}

\newcommand{\coh}{\operatorname{coh}}

\newcommand{\fmod}{\operatorname{mod}}
\newcommand{\refl}{\operatorname{ref}}
\newcommand{\add}{\operatorname{add}}

\newcommand{\proj}{\operatorname{proj}}

\newcommand{\Kb}{\operatorname{K^b}}
\newcommand{\Km}{\operatorname{K}^{-}\!}
\newcommand{\Perf}{\operatorname{Perf}}

\newcommand{\CM}{\operatorname{CM}}

\newcommand\Db{\mathop{\rm{D}^b}}
\def\EG{\mathop{\sf EG}\nolimits}

\newcommand{\Image}{\operatorname{Im}}

\newcommand{\Ker}{\operatorname{Ker}}
\newcommand{\Cok}{\operatorname{Cok}}

\newcommand{\LatticePoint}[1][]{%
\begin{tikzpicture}
\filldraw (0,0) circle (2pt);
\end{tikzpicture}
}

\newcommand{\rk}{\operatorname{\sf rk}}

\newcommand{\Stab}{\operatorname{Stab}}
\newcommand{\cStab}[1]{\mathrm{Stab}_{#1}^{\kern -0.5pt \circ}\kern -0.2pt}
\newcommand{\nStab}[1]{\mathrm{Stab}_{#1}\kern -0.1pt}
\newcommand{\TitsR}{{\sf Tits}_{\kern 1pt \bR}}
\newcommand{\CTitsR}{{\sf CTits}_{\kern 1pt \bR}}
\newcommand{\CTitsC}{{\sf CTits}_{\kern 1pt \bC}}

\newcommand{\cAut}[1]{\mathrm{Aut}_{#1}^{\kern -0.5pt \circ}\kern -0.2pt}

\newcommand{\Per}{{}^{0}\!\operatorname{Per}}

\def\MutTo{\mathop{\sf MutTo}\nolimits}

\def\Targ{\mathop{\sf Term}\nolimits}

\def\Br{\mathop{\sf PBr}\nolimits}
\def\Tr{\mathop{\sf Tr}\nolimits}

\newcommand{\Mut}{\operatorname{Mut}}

\newcommand{\Flop}{\operatorname{\sf Flop}}
\newcommand{\Flops}{\operatorname{Flop}}

\newcommand{\Level}{\operatorname{\sf Level}}
\newcommand{\Alcove}{\operatorname{\sf Alcove}}

\newcommand{\cA}{\mathcal{A}}
\newcommand{\cB}{\mathcal{B}}

\newcommand{\cG}{\mathcal{G}}
\newcommand{\cH}{\mathcal{H}}

\newcommand{\cO}{\mathcal{O}}
\newcommand{\cP}{\mathcal{P}}

\newcommand{\cY}{\mathcal{Y}}
\newcommand{\cZ}{\mathcal{Z}}

\newcommand{\bC}{\mathbb{C}}
\newcommand{\bE}{\mathbb{E}}

\newcommand{\bH}{\mathbb{H}}

\newcommand{\bP}{\mathbb{P}}
\newcommand{\bQ}{\mathbb{Q}}
\newcommand{\bR}{\mathbb{R}}
\newcommand{\bU}{\mathbb{U}}

\newcommand{\bZ}{\mathbb{Z}}

\newcommand{\m}{\mathfrak{m}}

\newcommand{\scrA}{\EuScript{A}}
\newcommand{\scrB}{\EuScript{B}}
\newcommand{\scrC}{\EuScript{C}}
\newcommand{\scrD}{\EuScript{D}}
\newcommand{\scrE}{\EuScript{E}}
\newcommand{\scrF}{\EuScript{F}}
\newcommand{\scrG}{\EuScript{G}}
\newcommand{\scrH}{\EuScript{H}}
\newcommand{\scrI}{\EuScript{I}}
\newcommand{\scrJ}{\EuScript{J}}
\newcommand{\scrK}{\EuScript{K}}
\newcommand{\scrL}{\EuScript{L}}
\newcommand{\scrM}{\EuScript{M}}
\newcommand{\scrN}{\EuScript{N}}
\newcommand{\scrO}{\EuScript{O}}
\newcommand{\scrP}{\EuScript{P}}
\newcommand{\scrQ}{\EuScript{Q}}

\newcommand{\scrS}{\EuScript{S}}
\newcommand{\scrT}{\EuScript{T}}
\newcommand{\scrU}{\EuScript{U}}
\newcommand{\scrV}{\EuScript{V}}
\newcommand{\scrW}{\EuScript{W}}

\newcommand{\dsG}{\mathds{G}}
\newcommand{\dsN}{\mathds{N}}

\newcommand{\simto}{\xrightarrow{\sim}}

\renewcommand{\l}{\langle}
\renewcommand{\r}{\rangle}

\newcommand{\Delt}{\Updelta_{\kern 0.05em 0}}
\newcommand{\DeltAff}{\Updelta_{\kern 0.05em 0}^{\aff}}
\newcommand{\WDelt}{W_{\kern -0.1em \Updelta}}
\newcommand{\WDeltaff}{W_{\kern -0.1em \Updelta_{\aff}}}
\newcommand{\Wkern}[1]{W_{\kern -0.1em #1}\kern 0.05em}
\newcommand{\wo}[1]{w_{\kern -0.075em #1}}
\newcommand{\wop}[1]{w^{\phantom J}_{\kern -0.1em #1}}
\newcommand{\GammaJ}{\Upgamma_{\kern -0.05em J}}
\newcommand{\GammaS}{\Upgamma_{\kern -0.05em \scrJ}}
\newcommand{\Cone}[1]{{\sf TCone}{ (#1)}}
\newcommand{\xlup}[1]{{}^{#1}\kern -0.15em x}
\newcommand{\xlupmax}[1]{{}^{#1}\kern -0.25em x}
\newcommand{\iDelta}{\iota_{\kern -0.075em \Updelta}}
\newcommand{\Phisub}[1]{\Phi_{\kern -0.1em #1}}
\newcommand{\aff}{\operatorname{\mathsf{aff}}\nolimits}

\usepackage[10pt,nona4]{optional}

\newcommand\Curve{\mathrm{C}}
\def\Cl{\mathop{\rm Cl}\nolimits}
\def\Id{\mathop{\rm{Id}}\nolimits}

\newcommand{\step}[1]{\medskip\emph{Step #1}:}

\newcolumntype{C}[1]{>{\centering\let\newline\\\arraybackslash\hspace{0pt}}m{#1}}

\def\EeightFourScalex{0.5}
\def\EeightFourScaley{0.5}

\newcommand{\EeightFour}[1][]{%
\begin{tikzpicture}[xscale=\EeightFourScalex,yscale=\EeightFourScaley]
\coordinate (A) at (0,0);
\node at (A) {\LatticePoint};
\coordinate (B) at (0,2);
\coordinate (C) at (0,4);
\coordinate (Ap) at (4,0);
\coordinate (Bp) at (4,2);
\coordinate (Cp) at (4,4);
\coordinate (b) at (2,4);
\coordinate (bp) at (2,0);
\draw (A)--(C);
\draw (Ap)--(Cp);
\draw (A)--(Ap);
\draw (B)--(Bp);
\draw (C)--(Cp);
\draw (b)--(bp);
\draw (A)--(b);
\draw (A)--(Cp);
\draw (C)--(bp);
\draw (C)--(Ap);
\draw (Ap)--(b);
\draw (Cp)--(bp);
\end{tikzpicture}
}

\begin{document}

\title[]{Stability Conditions for $3$-fold Flops}

\author{Yuki Hirano}
\address{Y.~Hirano, Department of Mathematics, Kyoto University, Kitashirakawa-Oiwake-cho, Sakyo-ku, Kyoto, 606-8502, Japan.}\email{y.hirano@math.kyoto-u.ac.jp}
\author{Michael Wemyss}
\address{M.~Wemyss, The Mathematics and Statistics Building, University of Glasgow, University Place, Glasgow, G12 8QQ, UK.}
\email{michael.wemyss@glasgow.ac.uk}

\begin{abstract} 
Let $f\colon X\to\Spec R$ be a $3$-fold flopping contraction, where $X$ has at worst Gorenstein terminal singularities and $R$ is complete local. We describe the space of Bridgeland stability conditions on the null subcategory $\scrC$ of $\Db(\coh X)$, which consists of those complexes that derive pushforward to zero, and also on the affine subcategory $\scrD$, which consists of complexes supported on the exceptional locus. We show that a connected component $\cStab{}{\scrC}$ of $\Stab{\scrC}$ is the universal cover of the complexified complement of the real hyperplane arrangement associated to $X$ via the Homological MMP, and more generally that $\cStab{n}\scrD$ is a regular covering space of the infinite hyperplane arrangement constructed in \cite{IW9}.  Neither arrangement is Coxeter in general. As a consequence, we give the first description of the Stringy K\"ahler Moduli Space (SKMS) for all smooth irreducible $3$-fold flops.  The answer is surprising: we prove that the SKMS is always a sphere, minus either $3,4,6,8,12$ or $14$ points, depending on the length of the curve. 
%Let $f\colon X\to\mathrm{Spec}\, R$ be a 3-fold flopping contraction, where $X$ has at worst Gorenstein terminal singularities and $R$ is complete local. We describe the space of Bridgeland stability conditions on the null subcategory $\mathscr{C}$ of the bounded derived category of $X$, which consists of those complexes that derive pushforward to zero, and also on the affine subcategory $\mathscr{D}$, which consists of complexes supported on the exceptional locus. We show that a connected component of stability conditions on $\mathscr{C}$ is the universal cover of the complexified complement of the real hyperplane arrangement associated to $X$ via the Homological MMP, and more generally that a connected component of normalised stability conditions on $\mathscr{D}$ is a regular covering space of the infinite hyperplane arrangement constructed in Iyama-Wemyss [IW9].  Neither arrangement is Coxeter in general. As a consequence, we give the first description of the Stringy K\"ahler Moduli Space (SKMS) for all smooth irreducible 3-fold flops.  The answer is surprising: we prove that the SKMS is always a sphere, minus either 3, 4, 6, 8, 12 or 14 points, depending on the length of the curve. 
\end{abstract}
\thanks{YH was supported by  JSPS KAKENHI 19K14502 and The Max Planck Institute for Mathematics, and MW by EPSRC grants~EP/R009325/1 and EP/R034826/1.}
%\subjclass[2010]{Primary~; Secondary~}
%\keywords{stability condition, flopping contraction, noncommutative resolutions}
\maketitle{}

\section{Introduction}
Our setting is $3$-fold flopping contractions, namely $f\colon X\to \Spec R$, where $(R,\m)$ is a three-dimensional complete local Gorenstein $\mathbb{C}$-algebra with at worst terminal singularities.  We allow $X$ to be singular, with $X$ having at worst terminal singularities.  Consider the fibre $\Curve\colonequals f^{-1}(\m)$, which with its reduced scheme structure is well-known to decompose into a union of $n$ irreducible curves, each isomorphic to $\bP^1$. 

Given this setup, consider the following two subcategories of $\Db(\coh X)$
\begin{align*}
\scrC&\colonequals \{ \scrF\in\Db(\coh X)\mid \mathbf{R
}f_*\scrF=0\} \\
\scrD&\colonequals \{ \scrF\in\Db(\coh X)\mid \Supp\scrF\subseteq \Curve\}.
\end{align*}
It is a fundamental question to describe the spaces of stability conditions on $\scrC$ and $\scrD$, and to use this to help describe the autoequivalence group of $\Db(\coh X)$.  Both $\scrC$ and $\scrD$ have finite length hearts, and it is well-known from surfaces \cite{B3} that stability conditions on $\scrC$ should exhibit `finite-type' ADE behaviour, whilst $\scrD$ should be the `affine' version.  

One of the problems is that traditional finite and affine Coxeter groups do not suffice in this setting. On one hand, it is possible that $\scrC$ is controlled by a Coxeter arrangement that does not have an associated affine Coxeter group.  On the other hand, the category $\scrD$ predicts that such an affine object `exists'. Even worse, it is possible that $\scrC$ is controlled by a simplicial hyperplane arrangement that is not Coxeter. In that case, the affine object that controls $\scrD$ is even less clear.  Making precise statements about both $\scrC$ and $\scrD$  is in fact one of the main outcomes of this paper.

\subsection{Hidden t-structures}
Our approach to this problem is noncommutative, and necessarily so.  One of our new insights is that many of the t-structures that arise in the stability manifold of $\scrD$ are `hidden', in the sense that they are not obviously part of the birational geometry, nor are they translations of the birational geometry by line bundle twists.  However, they do have very conceptual noncommutative interpretations. To describe this in more detail, it is helpful to first briefly review the known cases.

The first partial solution to describing stability conditions on $\scrC$ and $\scrD$ in this $3$-fold setting is due to Toda \cite{T08}, who worked under two additional assumptions: (1) $X$ is smooth, and (2) for a generic hyperplane section $H\hookrightarrow\Spec R$,  the pullback $X\times_R H$ is smooth.  Both conditions are restrictive for different reasons, with the second being the least natural, and by far the most problematic to remove.  The crucial point is that, under these additional assumptions, the dual graph is an ADE Dynkin diagram.  When this happens, the traditional language of finite and affine Weyl groups suffice, and the relevant t-structures are all described by perverse sheaves and their tensors by line bundles. Toda \cite{T08} packages this together to describe a component of normalised stability conditions on $\scrD$ as a regular covering of the complexified complement of the associated affine root hyperplane arrangement.  Furthermore, the Galois group has a very satisfying geometric description, as those compositions of flop functors and line bundle twists that act trivially on K-theory.

Alas, these satisfyingly geometric statements all fail without assumption (2).  Perhaps counter-intuitively, the hardest case turns out to be the most elementary one: that of a single-curve flop.  In this case, the flopping curve has an associated \emph{length} invariant $\ell$, which is some number between one and six.  The assumption (2) holds if and only if the curve has length one.  Evidently, this is quite restrictive.

 One of our key observations is that, in the general situation of a 3-fold flop $X\to\Spec R$, tracking under flop functors and line bundle twists does not suffice.  To illustrate this visually in the case of a two curve flop, we will show below that stability conditions on $\scrD$ are controlled by infinite hyperplane arrangements $\scrH^{\aff}\subseteq \mathbb{R}^n$ such as that shown in Figure~\ref{fig: Figure 1}.
\begin{figure}[ht]
\[
\begin{array}{c}
\includegraphics[angle=0,scale = 0.75]{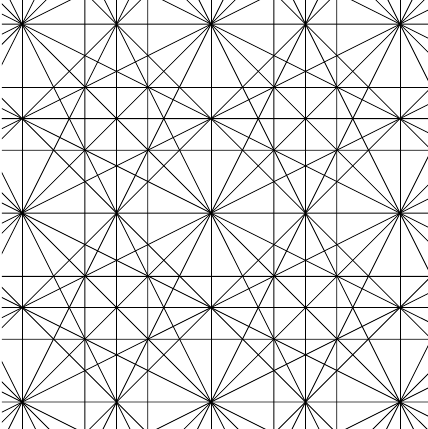}
\end{array}
\]
\caption{Example of hyperplane arrangement $\scrH^{\aff}$}\label{fig: Figure 1}
\end{figure}
In general the hyperplane arrangements are quite complicated, and there are many more chambers than one might naively expect.  The above example has an obvious $\bZ^2$ action, given by tensoring by the line bundles corresponding to the two curves.  However, this action jumps the central chamber over many intermediate t-structures.   These all turn out to be  hearts of noncommutative resolutions, and their variants.

To circumvent this problem, which occurs even in the case of a single curve flop, we appeal to recent advances in noncommutative resolutions and their mutation theory.  In the process, we will recover a conceptual understanding of the hidden t-structures, give a  full description of (a component of) normalised stability conditions on $\scrD$, and for the first time compute the Stringy K\"ahler Moduli Space.

\subsection{Main Stability Results}
Describing stability conditions on $\scrC$ turns out to be quite easy. Working in our most general setup $f\colon X\to \Spec R$, we first prove the following, which is entirely parallel to \cite[Section 6, ArXiv v2]{T08}.  Below the hyperplane arrangement $\scrH\subset\mathbb{R}^n$ need not be ADE, or even Coxeter, but nevertheless tracking under the Flop functors still produces the chambers of the stability manifold for the category $\scrC$.  As notation, consider the set $\Flops(X)$ consisting of all those pairs $(F,Y)$ where $Y\to\Spec R$ is obtained from $X$ through an iterated chain of simple flops, and $F$ is a composition of flop functors and their inverses, from $\Db(\coh Y)$ to $\Db(\coh X)$.
\begin{thm}[\ref{chambers in cStab C}, \ref{regular covering}]\label{chambers intro}
There is a union of chambers
\[
\cStab{}\scrC=\bigcup_{(Y,F)\in\,\Flops(X)} F(U),
\]
where $U$ is defined in Notation~\ref{chamber notation}, and furthermore the natural map
\[
\cZ\colon \cStab{}\scrC\to \bC^{n}\backslash\scrH_{\bC}\\
\]
is the universal cover of the complexified complement of $\scrH$.
\end{thm}
The hyperplane arrangement $\scrH$ can be described in various ways, and this is explained in Section~\ref{hyperplane section}.  The main content of the above theorem is to establish that the map is a regular covering map; our previous work \cite{HW} on the faithfulness of the action then establishes that the cover is universal.  From this, the seminal work of Deligne \cite{Deligne} on the $K(\uppi,1)$ conjecture for simplicial hyperplane arrangements immediately confirms the following.
\begin{cor}[\ref{contractible}]
$\cStab{}\scrC$ is contractible.
\end{cor}

However, the main content in this paper is our description of stability conditions on the category $\scrD$, which is much harder.  Passing to a more noncommutative viewpoint, by the HomMMP \cite[4.2]{HomMMP} we first observe that the union in Theorem~\ref{chambers intro} can be reindexed using instead those pairs $(\Upphi,L)$ where $\Upphi$ is a chain of mutation functors and their inverses, and $L$ belongs to the mutation class $\Mut_0(N)$ of $N$ described in \eqref{ass modi}, where mutation at the submodule $R$ is \emph{not} allowed.  Disregarding this last restriction, and thus allowing mutation at all summands, gives an infinite set $\Mut(N)$.  Via \cite{IW9}, this turns out to index the chambers in a corresponding infinite hyperplane arrangement $\scrH^{\aff}$.   

The following is our main result.  The Galois group $\Br\scrD$ is by definition all compositions of mutation functors and their inverses that start and finish at our fixed $\End_R(N)$. 

\begin{thm}[\ref{chambers in cStab C}, \ref{regular covering}]\label{normalised stab intro}
There is a union of chambers
\[
\cStab{n}\scrD=\bigcup_{(\Upphi,L)}\Upphi(\dsN_L).
\]
where $\dsN_L$ is defined in Notation~\ref{chamber notation}, and $\Upphi$ are compositions of mutation functors and their inverses.  Furthermore, the natural forgetful map
\[
\cZ\colon \cStab{n}\scrD\to \bC^{n}\backslash\scrH^{\aff}_{\bC}\\
\]
is a regular covering map, with Galois group $\Br\scrD$.
\end{thm}

Whilst passing to noncommutative resolutions (and their variants) provides the conceptual framework to tackle the above problem, and to understand the extra t-structures, their appearance comes at a significant cost.  Namely, it becomes much harder to argue when functors are the identity, and thus to establish that $\cZ$ is a regular covering map.  The following is one of our main technical results, which heavily uses the isolated cDV assumption.

\begin{thm}[\ref{key functorial identity}]
Suppose that $\Gamma$ is an arbitrary modifying algebra (or noncommutative crepant resolution) of $R$, where $R$ is isolated cDV.  Consider an equivalence 
\[
G\colon \Db(\fmod\Gamma)\to\Db(\fmod\Gamma)
\]
obtained as an arbitrarily long sequence of mutation functors and their inverses.  If $G$ restricts to an equivalence $\fmod\Gamma\to\fmod\Gamma$, then $G\cong\Id$.
\end{thm}

There are additional variants to the above, summarised in Theorem~\ref{action free}, which may be of independent interest.

\subsection{Autoequivalence and SKMS results}
Aside from producing new invariants for $3$-fold flops, linking to classification problems, autoequivalences, noncommutative resolutions and deformation theory, one of our main motivations for establishing Theorem~\ref{normalised stab intro} is that it provides the first mechanism to compute the fabled stringy K\"ahler moduli space (SKMS).  To do this requires some additional work, since following and generalising \cite{T08} we view the SKMS as the quotient of $\cStab{n}\scrD$ by a certain group $\cAut{}\scrD$.  

The definition of the group $\cAut{}\scrD$ is a rather subtle point, since in this local model everything is relative to the base $\Spec R$, and so everything should respect this structure.  In particular, $\cAut{}\scrD$ should not  contain isomorphisms between flopping contractions unless they preserve the $R$-scheme structure. We achieve this by defining  $\cAut{}\scrD$ to be the group of \emph{$R$-linear} Fourier--Mukai equivalences $\scrD\to\scrD$ that preserve  $\cStab{n}\scrD$.  Even for a single curve flop with $\ell=1$, the restriction to $R$-linear functors is necessary for the mathematically defined SKMS \cite[p6169]{T08} to coincide with the physical version \cite[Figure 1]{Aspinwall}.

The intrinsically defined group $\cAut{}\scrD$ has the following more concrete description. 

\begin{thm}[\ref{semidirect main}]\label{group iso intro}
$\cAut{}\scrD\cong \Br\scrD\rtimes\Pic X$.
\end{thm}

Combining Theorem~\ref{normalised stab intro} with Proposition~\ref{group iso intro} and an elementary hyperplane calculation allows us to finally compute the SKMS, as  $\cStab{n}\scrD/\cAut{}\scrD$, for smooth irreducible flops, generalising \cite[p6169]{T08} and \cite[Figure 1]{Aspinwall}.  There does not appear to be any predictions or conjectures in the literature for what the SKMS should be for higher lengths.  Perhaps this is just as well, since the result is quite surprising.

\begin{cor}[\ref{SKMS text}]
For a smooth irreducible flop $X\to\Spec R$ of length $\ell$, the  SKMS is always a $2$-sphere, with holes removed at both the north and south pole, together with the following number of holes removed from the equator. 
\[
\begin{tabular}{C{1cm}C{0.4cm}C{0.4cm}C{0.4cm}C{0.4cm}C{0.4cm}C{0.4cm}C{0.4cm}C{0.4cm}}\toprule
\textnormal{$\ell$}&$1$&$2$&$3$&$4$&$5$&$6$\\
\midrule
\textnormal{Holes}&$1$&$2$&$4$&$6$&$10$&$12$\\
\bottomrule
\end{tabular}
\]
\end{cor}
For example, when $\ell=4$, it follows that the SKMS is
\[
\begin{tikzpicture}[scale=1.2]
%equator around the back
\draw[gray,densely dotted] (1,0) arc (0:180:1cm and 0.2cm);
%equator around the front
\draw[densely dotted] (1,0) arc (0:-180:1cm and 0.2cm)
coordinate[pos=0.8] (A) coordinate[pos=0.67] (B) coordinate[pos=0.56] (C) coordinate[pos=0.45] (D) coordinate[pos=0.33] (E) coordinate[pos=0.2] (F);
%black circle for the sphere
\draw ([shift=(-84:1cm)]0,0) arc (-84:84:1cm)  
[bend left] to (96:1cm)
arc (96:264:1cm)
[bend left] to cycle; 
%dots at top and bottom of sphere
\draw[densely dotted] ([shift=(86.75:1cm)]0,0) arc  (86.75:94:1cm);
\draw[densely dotted] ([shift=(-86.75:1cm)]0,0) arc  (-86.75:-94:1cm);
%false equator of labels of dots
\draw[draw=none] (1,-0.2) arc (0:-180:1cm and 0.2cm)
coordinate[pos=0.8] (a) coordinate[pos=0.67] (b) coordinate[pos=0.56] (c) coordinate[pos=0.45] (d) coordinate[pos=0.33] (e) coordinate[pos=0.2] (f);
%dots
\filldraw[fill=white,draw=black] (A) circle (2pt);
\filldraw[fill=white,draw=black] (B) circle (2pt);
\filldraw[fill=white,draw=black] (C) circle (2pt);
\filldraw[fill=white,draw=black] (D) circle (2pt);
\filldraw[fill=white,draw=black] (E) circle (2pt);
\filldraw[fill=white,draw=black] (F) circle (2pt);
%labels
\node at (a) {$\scriptstyle 1$};
\node at (b) {$\scriptstyle 4$};
\node at (c) {$\scriptstyle 3$};
\node at (d) {$\scriptstyle 2$};
\node at (e) {$\scriptstyle 3$};
\node at (f) {$\scriptstyle 4$};
\end{tikzpicture}
\]
where we refer the reader to Theorem~\ref{SKMS text} for more details, including an explanation of the numerics, and the labelling of the holes on the equator.

\subsection*{List of Notation} 
A list of notation is provided in Appendix~\ref{notation appendix}.

\subsection*{Acknowledgements} 
We thank Jenny August for comments and suggestions on Appendix~\ref{comb appendix}, Will Donovan for many helpful remarks, Greg Stevenson for vanquishing minimal projective resolutions from an early version of the proof of Lemma~\ref{LemmaA}, and Yukinobu Toda for conversations regarding \cite{T08}.

\section{Flops via Noncommutative Methods}\label{sec:FlopsviaNC}

In this section, we recall  the basics of modification algebras, tilting, mutations of modifying modules, and the relationship to flops, mainly to set notation. Throughout $R$ is a three dimensional complete local Gorenstein normal $\bC$-algebra,  and $f\colon X\to \Spec R$ is a flopping contraction as in the introduction. Furthermore, write $\Curve$ for $f^{-1}(\m)$ endowed with reduced scheme structure.  It is well known that $\Curve=\bigcup_{i=1}^n \Curve_i$ is a union of $n$ $\bP^1$s.
 
 \subsection{Tilting and Modification Modules}\label{sec:tilt and modify} 
Since $R$ is complete local,  there exist line bundles $\scrL_1,\hdots, \scrL_n\in \Pic(X)$ such that $\scrL_i \cdot \Curve_j=\updelta_{ij}$. Set $\scrV_0\colonequals \scrO_X$, and write $\scrV_i$ for the vector bundle arising as the universal extension
\begin{equation}
0\to \scrO_{X}^{\oplus r_i-1}\to \scrV_i\to \scrL_i\to0,\label{extension}
\end{equation}
associated to a minimal set of $r_i-1$ generators of the $R$-module $\mathrm{H}^1(X,\scrL_i^*)$. Then by \cite[3.5.5]{VdB} the vector bundle $\scrV_X\colonequals \bigoplus_{i=0}^n\scrV_i^*$ is tilting, and so after setting 
\[
\Lambda\colonequals \End_X(\scrV_X)\cong\End_R(f_*\scrV_X),
\]  
the functor 
\begin{equation}
\Uppsi\colonequals \RHom_X(\scrV_X,-)\colon \Db(\coh X)\to \Db(\fmod \Lambda)\label{tilt equiv}
\end{equation} 
is an equivalence. Our approach to stability conditions will be through noncommutative methods. Recall that $\CM R$ denotes the category of (maximal) Cohen--Macaulay $R$-modules, and $\refl R$ denotes the category of reflexive $R$-modules.  A reflexive $R$-module $L\in \refl R$ is called {\it modifying} if $\End_{R}(L)\in \CM R$.

In the flops setting, for the fixed $f\colon X\to\Spec R$, consider the underived direct image $N_i\colonequals f_*(\scrV_i^*)\in \fmod R$. Note that $N_0\cong R$.  Throughout, we set
\begin{eqnarray}
N\colonequals f_*({\scrV_X})\cong \bigoplus_{i=0}^n N_i.\label{ass modi}
\end{eqnarray}
It is known that $N\in\CM R$, and $N$ is a modifying $R$-module \cite[3.2.10]{VdB}.

\subsection{Mutations and Equivalences}\label{sec:mut and equiv}
 Given  any modifying $R$-module $L=\bigoplus_{j=0}^nL_j$ with each $L_j$ indecomposable, there is an operation, called {\it mutation at $L_i$}, that gives a new modifying $R$-module written $\upnu_i L$.    We briefly recall the construction here. Set 
\[
  L_i^{c}\colonequals \bigoplus_{j\neq i}L_j,
\] 
so that  $L=L_i\oplus L_i^c$, and consider the {\it minimal right $\add (L_i^c)^*$-approximation} 
\begin{eqnarray}
a_i\colon U_i\to L_i^*\label{minimal app}
\end{eqnarray}
of $L_i^*$, which by definition means that
\begin{enumerate}
\item $U_i\in\add(L_i^c)^*$ and $a_i\circ(-)\colon \Hom_R((L_i^c)^*,U_i)\to\Hom_R((L_i^c)^*,L_i^*)$ is surjective,
\item If $b\in\End_R(U_i)$ satisfies $a_i=a_i \circ b$, then $b$ is an isomorphism.
\end{enumerate}
Since $R$ is complete, such an $a_i$ exists and is unique up to isomorphism. The (left) mutation of $L$ at $L_i$ is then defined to be
\[
\upnu_i L\colonequals (\Ker a_i)^*\oplus L_i^c.
\] 
The following properties are known.

\begin{prop}\label{basicmodi}
With notation as above,  in particular $R$ is isolated cDV, the following statements hold.
\begin{enumerate} 
\item\label{basicmodi 1} The mutation $\upnu_i L$ is a modifying $R$-module.
\item\label{basicmodi 3} There is an isomorphism $\upnu_i\upnu_iL\cong L$.
 \end{enumerate}
\end{prop}
\begin{proof}
The first part is general; see e.g.\ \cite[\S6]{IW1}.  The second part is specific to isolated cDV singularities \cite[9.28]{IW9}.
\end{proof}

\begin{dfn}\label{exchange def}
Fix the modifying module $N$ from \eqref{ass modi}. Write $\Mut(N)$ for the set of isomorphism classes of all iterated mutations of $N$, and $\Mut_0(N)$ for the subset consisting of those iterated mutations of $N$ at only the $(i\neq0)$-th summands. The exchange graph $\EG(N)$ is the graph whose vertices are the elements of $\Mut(N)$, and two vertices in $\EG(N)$ are joined by an edge if and only if the corresponding modifying modules are related by a mutation at an indecomposable summand.  The exchange graph $\EG_0(N)$ is the full subgraph whose vertices are the elements of $\Mut_0(N)$.
\end{dfn}  
Alternatively, the exchange graph $\EG_0(N)$ is the full subgraph whose vertices correspond to CM modules. We once and for all fix a decomposition $N=R\oplus N_1\oplus \hdots\oplus N_n$, where $N_0=R$.  Via the Coxeter-style combinatorics in Section~\ref{hyperplane section}, this fixed decomposition induces an ordering on the summands of all other elements $L$ of $\Mut_0(N)$, such that locally crossing a wall locally labelled $s_i$ always corresponds to replacing the $i$th summand.  In this way,  there is a global labelling on the edges of both $\EG_0(N)$ and $\EG(N)$ using the sets $\{s_1,\hdots,s_n\}$ and $\{s_0,s_1,\hdots,s_n\}$  respectively.

The mutation of a modifying $R$-module $L$ gives rise to a derived equivalence between $\Gamma\colonequals \End_R(L)$ and $\upnu_i\Gamma\colonequals \End_R(\upnu_i L)$,  induced by a tilting bimodule $T_i$.  Since $R$ is isolated, in fact $T_i=\Hom_R(L,\upnu_iL)$ by \cite[6.14]{IW1}, and the following functor is an equivalence:
\begin{eqnarray}
\Upphi_{i}\colonequals \RHom_{\Gamma}(T_i,-)\colon \Db(\fmod \Gamma)\xrightarrow{\sim}\Db(\fmod \upnu_i\Gamma).\label{mutation eqv}
\end{eqnarray}
The functor $\Upphi_i$ is called the {\it mutation functor} at the summand $i$.

\subsection{Flops and Mutation}

Recall that the exceptional locus of $f$, given reduced scheme structure, decomposes into $n$ copies of $\mathbb{P}^1$, namely $\Curve=\bigcup_{i=1}^n \Curve_i$. For each $\Curve_i$ there exists a flopping contraction $g_i\colon X\to Y_i$ which contracts only $\Curve_i$, and the flopping contraction $f\colon X\to \Spec R$ factors through $g_i$.  Furthermore, there exists a flop $g_i^+\colon X_i^+\to Y_i$ of $g_i$ such that the following diagram commutes
\[
\begin{tikzcd}
X\arrow[rd, "g_i"] \arrow[rdd, "f"']\arrow[rr, "",dashed]&& X_i^+\arrow[ld, "g_i^+"']\arrow[ldd, "f_i^+"]\\
&Y_i\arrow[d, "h_i"]&\\
&\Spec R&
\end{tikzcd}
\]
where $f_i^+\colonequals h_i\circ g_i^+$.  Then $(f^{+}_i)^{-1}(\m)$, with reduced scheme structure, is  the union of $n$ irreducible curves $\bigcup_{j=1}^n\Curve_j^+$, where for $j\neq i$ each $\Curve_j^+$ is the proper transformation of $\Curve_j$, and if $j= i$ then $\Curve_i^+$ is the flopped curve.

\begin{thm}[{\cite[4.2]{HomMMP}}]\label{flopmut}
With notation as above, the following hold.
\begin{enumerate}
\item There is an isomorphism of $R$-modules $\mathrm{H}^0(X_i^+,\scrV_{X_i^+})\cong \upnu_i N$.

\item The following diagram of equivalences is functorially commutative
\[
\begin{tikzpicture}
\node (A1) at (0,0) {$\Db(\coh X)$};
\node (A2) at (3.5,0) {$\Db(\coh X_i^+)$};
\node (B1) at (0,-1.5) {$\Db(\fmod \Lambda)$};
\node (B2) at (3.5,-1.5) {$\Db(\fmod \upnu_i\Lambda)$};
\draw[->] (A1) -- node[above] {\scriptsize $\Flop_i$} (A2);
\draw[->] (B1) -- node[above] {\scriptsize $\Upphi_i$} (B2);
\draw[->] (A1) -- node[left] {\scriptsize $\RHom_X(\scrV_X,-)$} (B1);
\draw[->] (A2) -- node[right] {\scriptsize $\RHom_{X_i^+}(\scrV_{X_i^+},-)$} (B2);
\end{tikzpicture}
\]
where ${\sf Flop}_i\colon \Db(\coh X)\to \Db(\coh X_i^+)$ is the quasi-inverse of the Bridgeland--Chen flop functor \cite{B02,Chen}.
\end{enumerate}
\end{thm}

\subsection{\texorpdfstring{$R$}{R}-linear equivalences}
In our flops setup $f\colon X\to\Spec R$, the category $\coh X$ is $R$-linear, and thus so too is $\Db(\coh X)$.  Autoequivalences that preserve this structure will be particularly important later. Here we briefly recall the $R$-linear structure, and give some preliminary results.

Since $\Spec R$ is an affine scheme,  there is a bijection
\begin{equation}
\Mor(X,\Spec R) \longleftrightarrow \Hom(R,\scrO_X(X)).\label{morphism bijection}
\end{equation}
Given $g\colon X\to \Spec R$, we will write $\mathsf{g}\colon R\to \scrO_X(X)$ for the corresponding morphism.

For $a\in\Hom_{\coh X}(\scrF,\scrG)$ and $\uplambda\in \scrO_X(X)$, consider $\uplambda\cdot a\in \Hom_{\coh X}(\scrF,\scrG)$ defined by 
\[
(\uplambda\cdot a)(x)\colonequals \uplambda|_U \cdot a(x)\in \scrG(U)
\] 
for all $x\in \scrF(U)$.  Under this action, $\coh X$ is an $\scrO_X(X)$-linear category.  The morphism $f\colon X\to\Spec R$ then gives $\coh X$ the structure of an $R$-linear category,  via $\mathsf{f}\colon R\to \scrO_X(X)$. 

The following two results are general, and are not specific to our flops setup.  Both are well-known, but for lack of reference we provide the proof.  
\begin{prop}\label{iso is R linear}
Consider $R$-schemes $f\colon X\to \Spec R$, $g\colon Y\to \Spec R$, and a morphism $h\colon X\to Y$. Writing $\mathsf{h}\colon \scrO_Y(Y)\to\scrO_X(X)$ for the corresponding morphism,  then the following are equivalent.
\begin{enumerate}
\item $g\circ h=f$.
\item $\mathsf{h}\circ \mathsf{g} =\mathsf{f}$.
\item $h^*\colon\coh Y\to \coh X$ is an $R$-linear functor.
\end{enumerate}
If $h$ is an isomorphism, the last condition is equivalent to $h_*\colon\coh X\to \coh Y$ being an $R$-linear functor.
\end{prop} 

\begin{proof} 
Note first that for $y\in \scrO_Y(Y)$, $a\in \Hom_Y(\scrF,\scrG)$ and $x\otimes\uplambda\in h^*(\scrF)=\scrF\otimes_Y \scrO_X$, 
\begin{eqnarray*}
h^*(y\cdot a)( x\otimes\uplambda)&=&(y\cdot a\otimes 1_X)( x\otimes\uplambda)\\
&=&(y\cdot a(x))\otimes \uplambda\\
&=&a(x)\otimes \mathsf{h}(y)\uplambda\\
&=&\mathsf{h}(y)\cdot(a(x)\otimes \uplambda)\\
&=& \mathsf{h}(y)\cdot\bigl( (a\otimes 1_X)(x\otimes \uplambda)\bigr).
\end{eqnarray*}
Hence by linearity  $h^*(y\cdot a)=\mathsf{h}(y)\cdot h^*(a)$ for all $y\in\scrO_Y(Y)$ and all $a\in \Hom_Y(\scrF,\scrG)$.\\
\noindent
(1)$\Leftrightarrow$(2) This is an immediate consequence of the bijection \eqref{morphism bijection}.\\
(2)$\Rightarrow$(3) Assuming (2), then for any $r\in R$ and $a\in\Hom_Y(\scrF,\scrG)$, 
\[
h^*(r\cdot a)=h^*(\mathsf{g}(r)\cdot a)=\mathsf{h}(\mathsf{g}(r))\cdot h^*(g)=\mathsf{f}(r)\cdot h^*(g)=r\cdot h^*(g),
\]
and so (3) holds.\\
(3)$\Rightarrow$(2) There is a commutative diagram
\[
\begin{tikzcd}
\scrO_Y(Y)\arrow[rr, "\mathsf{h}"]\arrow[d, "\sim"]&&\scrO_X(X)\arrow[d,"\sim"]\\
\Hom_Y(\scrO_Y,\scrO_Y)\arrow[rr, "h^*"]&&\Hom_X(\scrO_X,\scrO_X)\,,
\end{tikzcd}
\]
where the vertical arrows are $R$-linear isomorphisms. Since $h^*$ is $R$-linear by assumption, it follows that so too is $\mathsf{h}$. But then
\[
\mathsf{h}(\mathsf{g}(r))=\mathsf{h}(r\cdot 1)=r\cdot\mathsf{h}(1)=r\cdot 1=\mathsf{f}(r),
\]
for all $r\in R$, and thus $\mathsf{h}\circ\mathsf{g}=\mathsf{f}$.\\
The last statement holds since the inverse of an $R$-linear functor is $R$-linear.
\end{proof}

\begin{lem}\label{tensor is R linear}
Consider an $R$-scheme $X\to\Spec R$, and a line bundle $\scrL$ on $X$.  Then the functor $-\otimes\scrL\colon\coh X\to \coh X$ is $\scrO_X(X)$-linear, and in particular, is $R$-linear.
\end{lem}
\begin{proof}
For any $\uplambda\in \scrO_X(X)$, $a\in \Hom_X(\scrF,\scrG)$ and $x\otimes \ell\in \scrF\otimes \scrL$, we have 
\[
\bigl((\uplambda\cdot a)\otimes \scrL\bigr)\bigl(x\otimes \ell\bigr)=\uplambda\cdot a(x)\otimes \ell =(\uplambda\cdot (a\otimes\scrL))( x\otimes \ell )
\] 
and so by linearity the result follows.
\end{proof}

\section{Hyperplane Arrangements via \texorpdfstring{$K$}{K}-theory}\label{hyperplane section}

When the generic hyperplane section of $X$ is not smooth, it will turn out that stability conditions on $\scrC$ and $\scrD$ will not, in general,  be the regular covering of a space constructed using global rules.  The space will instead be constructed using iterated local rules, which we outline here.  This section is largely a summary of \cite{IW9} suitable for our needs, with the exception of some new results in Sections~\ref{New results S3} and \ref{complexified actions subsection}.

\subsection{General Elephants}\label{elephant subsection}
Slicing the flopping contraction $X\to\Spec R$ gives rise to combinatorial data, in the form of a labelled ADE Dynkin diagram $\Updelta$, vertices $\Delt$, and a subset $\scrJ\subseteq \Delt$.  We briefly recall this here.

Pulling back $X\to\Spec R$ along the map $\Spec R/g\to\Spec R$ for a generic element $g\in R$, gives a morphism $S\to\Spec R/g$, say.  By \cite{Pagoda}, $R/g$ is an ADE surface singularity, and $S$ is a partial crepant resolution. As such, $S$ is obtained by blowing down curves in the minimal resolution $\widetilde{S}$, and so by McKay correspondence we can describe $S$ combinatorially via a Dynkin diagram $\Updelta$, together with the subset $\scrJ\subseteq \Delt$ of those vertices that are blown down to obtain $S$. Thus, by convention, $\scrJ$ corresponds to the curves  that have been contracted by $\widetilde{S}\to S$. 

This data can be extended into the affine setting as follows.  Consider the corresponding extended Dynkin diagram $\Updelta_{\aff}$, and denote the extending vertex by $\star$.  Set $\scrJ_{\mathrm{aff}}\colonequals \scrJ$, considered as a subset of the vertices of $\Updelta_{\aff}$.

From this data, consider $\bR^{|\Updelta|}$ and $\bR^{|\Updelta_{\aff}|}$ based by the duals $\upalpha_i^*$, where the $i$ are indexed over the vertices of $\Updelta$ (respectively, $\Updelta_{\aff}$).  Inside these spaces, consider the Weyl chamber $C_+$, where all coordinates are positive, and set
\begin{align*}
\Cone{\Updelta}&=\bigcup_{w\in \WDelt}w(C_+)
\\
\Cone{\Updelta_{\aff}}&=\bigcup_{w\in \WDeltaff}w(C_+),
\end{align*}
where $\WDelt$ is the finite Weyl group, and $\WDeltaff$ the affine Weyl group.

There are subspaces $D_{\scrJ}\subset\bR^{|\Updelta|}$ and $D_{\scrJ_{\aff}}\subset\bR^{|\Updelta_{\aff}|}$ defined as 
\begin{align*}
D_{\scrJ}&\colonequals \{ \upvartheta\in\bR^{|\Updelta|}\mid \upvartheta_i=0\mbox{ if }i\in \scrJ\},\\%\nonumber\\
D_{\scrJ_{\aff}}&\colonequals \{ \upvartheta\in\bR^{|\Updelta_{\aff}|}\mid \upvartheta_i=0\mbox{ if }i\in \scrJ_{\mathrm{aff}}\}.%\label{DJ}
\end{align*}
These are based by $\upalpha_i^*$ for $i\in \Updelta- \scrJ$, respectively $i\in \Updelta_{\aff}- \scrJ_{\mathrm{aff}}$. As such, $\dim D_{\scrJ}=n$, the number of curves in the flopping contraction, and  $\dim D_{\scrJ_{\aff}}=n+1$.

\begin{dfn}[{\cite[\S1]{IW9}}]\label{T Cone def}
For $\scrJ\subseteq \Delt$ as above, 
\begin{enumerate}
\item $\Cone{\scrJ}\colonequals \Cone{\Updelta}\cap D_\scrJ$ is called the \emph{$\scrJ$-finite hyperplane arrangement}. 
\item $\Cone{\scrJ_{\aff}}\colonequals \Cone{\Updelta_{\aff}}\cap D_{\scrJ_{\aff}}$
is called the \emph{$\scrJ$-affine arrangement}. 
\end{enumerate}
\end{dfn}

\subsection{Affine Hyperplanes via K-theory}\label{affine notation subsection}
The combinatorics of the previous section can also be constructed via K-theory, which is more useful for stability conditions later.  Recall from \eqref{ass modi} that the flopping contraction $f\colon X\to \Spec R$ associates a modifying $R$-module $N$, with summands $R=N_0,N_1,\hdots,N_n$. Set $\Lambda\colonequals \End_R(N)$ and   $\scrP_i\colonequals \Hom_R(N,N_i)$, so that $\{\scrP_i\}_{0\leq i \leq n}$ is the set of all indecomposable projective $\Lambda$-modules.  It is well-known that
\[
\scrK_N\colonequals K_0(\Perf \Lambda)\cong \bigoplus_{i=0}^n\bZ[\scrP_i]\cong\bZ^{n+1}.
\]
For every $L\in\Mut N$, this process can be repeated: indeed each $\Lambda_L\colonequals\End_R(L)$ has K-theory of the same rank as above, and to avoid confusion write $\scrK_L\colonequals K_0(\Perf\Lambda_L)$.  Since the given flopping contraction $f$, and its associated modification algebra $\Lambda$ is \emph{fixed}, throughout we will refer to the distinguished object 
\[
\scrK\colonequals\scrK_N.
\]
Every mutation functor $\Upphi_i\colon \Db(\fmod\Lambda_L)\to\Db(\fmod\Lambda_{\upnu_iL})$ restricts to an equivalence on perfect complexes, and so write
\[
\upphi_i\colon \scrK_L\simto\scrK_{\upnu_iL}
\]
for the induced isomorphism of $K_0$-groups.  This map can be represented by an invertible $(n+1)\times(n+1)$ matrix over $\mathbb{Z}$, which is described as follows.

\begin{lem}\label{k-correspond proj}
Suppose that $L$ is modifying, and consider the exchange sequence \cite[(6.I)]{IW1} obtained as the dual of \eqref{minimal app}, namely
\begin{equation}
0\to L_i\to \bigoplus_{j\neq i}{L_j}^{\oplus b_{ij}} \to K_i^*.\label{bij}
\end{equation}
Write $\scrP_j=\Hom_R(L,L_j)$ for the projectives in $\Lambda_L$, and $\scrQ_j$ for the correspondingly ordered projectives in $\Lambda_{\upnu_iL}$.  Then $\upphi_i^{-1}\colon\scrK_{\upnu_iL}\to\scrK_{L}$ sends
\begin{equation}
[\scrQ_t]\mapsto
\left\{
\begin{array}{cl}
[\scrP_t]& \mbox{if } t\neq i,\\
-[\scrP_i]+\sum_{j\neq i}b_{ij}[\scrP_j] & \mbox{if } t=i.
\end{array}
\right.
\label{k-correspond}
\end{equation}
\end{lem}
\begin{proof}
Being induced by the tilting bimodule $T_i$ from \eqref{mutation eqv}, it is clear that $\Upphi_i$ sends $T_i$ to $\Lambda_{\upnu_iL}$.  Since $T_i$ only differs from $\Lambda_L$ at the summand $\scrP_i$, it is obvious that $\Upphi_{i}$ sends $\scrP_j$ to $\scrQ_j$ whenever $j\neq i$; see e.g.\ \cite[4.15(1)]{HomMMP}.  Hence $\Upphi_i^{-1}$ sends $\scrQ_j$ to $\scrP_j$.

 When $j=i$, under $\Upphi_i^{-1}$, the projective $\scrQ_i$ gets mapped to the $i$th summand of $T_i$, which by definition is $\Hom_R(L,K_i)$.  But by \cite[(6.Q)]{IW1}, applying $\Hom_R(L,-)$ to \eqref{bij} gives an exact sequence
\[
0\to \Hom_R(L,L_i)\to \bigoplus_{j\neq i} \Hom_R(L,L_j)^{\oplus b_{ij}}\to \Hom_R(L,K_i)\to 0.
\]
From this, the identification in K-theory clearly follows.
\end{proof}

 For $L\in\Mut(N)$, consider the shortest sequence of mutations 
\[
L\xrightarrow{i_1}\upnu_{i_1}L\to\hdots\xrightarrow{i_n}N,
\]
and define $\Upphi_L$ to be the composition of the corresponding mutation functors
\begin{equation}
\Upphi_L\colon \Db(\fmod \Lambda_{L})
\xrightarrow{\Upphi_{i_n}\circ\hdots\circ\Upphi_{i_1}} 
\Db(\fmod \Lambda).\label{eqn:defn PhiL}
\end{equation}
We write $\upphi_L\colon \scrK_L\to\scrK$ for the induced map on K-theory.  Throughout, whenever $\bZ\subseteq\mathds{k}$, we will abuse notation and also write $\upphi_L\colon 
\scrK_L\otimes\mathds{k}\to\scrK\otimes\mathds{k}$.  The following result is mainly combinatorial, and it mirrors the corresponding Coxeter statement.  We identify the basis element $\upalpha_i^*\in\Cone{\scrJ_{\aff}}$ with $[\scrP_i]\in\scrK$ to allow for the comparison.

\begin{thm}\cite[9.8]{IW9}\label{affine summary}
Suppose that $f\colon X\to \Spec R$ is a $3$-fold flopping contraction, such that $X$ has only terminal singularities. Then there is a decomposition
\[
\Cone{\scrJ_{\aff}} = \bigcup_{L\in\Mut N}\upphi_L(C_+)\subseteq\scrK\otimes\bR.
\]
In particular, the following statements hold.
\begin{enumerate}
\item The open decomposition $\bigcup\upphi_L(C_+)$ gives the chambers of the $\scrJ$-affine arrangement.
\item If $L\ncong M$, then $\upphi_{L}(C_+)$ and  $\upphi_{M}(C_+)$  do not intersect.
\item $\upphi_{L}(C_+)$ and  $\upphi_{M}(C_+)$ share a codimension one wall $\iff$ $L$ and $M$ differ by the mutation of an indecomposable summand.
\end{enumerate}
\end{thm}
We remark that $\Cone{\scrJ_{\aff}}$ does not fill $\bR^{|\Updelta_{\aff}|}$, as can be seen in Example~\ref{cones and level example} below.  Because of this, all information is contained in a slice.

\begin{dfn}\label{def: level}
The \emph{real level} is defined to be
\[
\Level_L\colonequals \biggr\{ z\in \scrK_L\otimes\bR \Bigm| \sum_{j=0}^n(\rk_RL_j) z_j=1\biggr\}.
\] 
The walls of the open decomposition $\bigcup\upphi_L(C_+)$ partition $\Level=\Level_N$ into open regions 
\[
\Alcove_L\colonequals\upphi_L(C_+)\cap\,\Level_N,
\] 
which by Theorem~\ref{affine summary} are still in bijection with $\Mut(N)$. We call these open regions the $\scrJ$-alcoves, and consider the infinite hyperplane arrangement
\begin{equation}
\scrH^{\aff}\colonequals\Level\backslash\bigcup_{L\in\Mut(N)}\Alcove_L.\label{affine H}
\end{equation}
\end{dfn}

\begin{exa}\label{cones and level example}
Consider $\Updelta=E_6$, and $\scrJ$ the following choice of unshaded vertices:
\[
\begin{tikzpicture}[scale=0.75]
\node (-1) at (-0.75,0) [DB] {};
\node (0) at (0,0) [DB] {};
\node (1) at (0.75,0) [DW] {};
\node (1b) at (0.75,0.75) [DB] {};
%\node (1c) at (0.75,1.5) [DW] {};
\node (2) at (1.5,0) [DB] {};
\node (3) at (2.25,0) [DB] {};
\draw [-] (-1) -- (0);
\draw [-] (0) -- (1);
\draw [-] (1) -- (2);
\draw [-] (2) -- (3);
\draw [-] (1) -- (1b);
%\draw [-] (1b) -- (1c);
\end{tikzpicture}
\]
Then, either by tracking the $C_+$ region directly over the labelling set in \cite[1.12]{IW9}, or by iterating the wall crossing rule in \cite[1.20(1)(2)]{IW9}, or by intersecting the full Tits cone of affine $E_6$ with the subspace $\mathbb{R}^2$ based by the shaded vertex and the extended vertex, it follows that $\Cone{\scrJ_{\aff}}$ is the shaded region in the following picture. Further, $\Level$ is illustrated by the dotted blue line $\upvartheta_0+3\upvartheta_1=1$.
\[
\begin{tikzpicture}[very thin,scale=1.5,>=stealth]
\filldraw[gray!10!white] (-3*14/15,1) -- (3,1) -- (3,-14/15)--(0,0)-- cycle;
\draw (0,0) -- (3,0);
\draw[gray,densely dotted] (-3,0) -- (0,0);
\draw[gray,densely dotted] (0,-1) -- (0,0);
\draw (0,0) -- (0,1);
\draw[->] (0,0) -- (0,1);
\draw[->] (0,0) -- (3,0);
\node at (0.15,1) {$\scriptstyle \upvartheta_1$};
\node at (3.15,0) {$\scriptstyle \upvartheta_0$};
{\foreach \i [evaluate=\i as \j using \i+1] in {1,2,3,4,5,6,7,8,9,10,11,12,13,14}
\draw (0,0) -- (-3*\i/\j,1);
}
{\foreach \i [evaluate=\i as \j using \i+1] in {1,2,3,4,5,6,7,8,9,10,11,12,13,14}
\draw (0,0) -- (3,-\i/\j);
}
\coordinate (A) at (-9/4,1);
\coordinate (B) at (3,-4/3*1/2);
\draw[blue,thin,densely dotted] (A) --(B);
\draw[blue,->] (0,0) -- ($(A)!(0,0)!(B)$); 
%%nodes
\filldraw[fill=white,draw=black]  (intersection cs:
    first line={(A)--(B)},
    second line={(0,0)--(-3*3/4,1)}) circle (0.75pt);
\filldraw[fill=white,draw=black]  (intersection cs:
    first line={(A)--(B)},
    second line={(0,0)--(-3*2/3,1)}) circle (0.75pt);
\filldraw[fill=white,draw=black]   (intersection cs:
    first line={(A)--(B)},
    second line={(0,0)--(-3*1/2,1)}) circle (0.75pt);
\filldraw[fill=white,draw=black]   (intersection cs:
    first line={(A)--(B)},
    second line={(0,0)--(0,1/2*1)}) circle (0.75pt);
\filldraw[fill=white,draw=black]   (intersection cs:
    first line={(A)--(B)},
    second line={(0,0)--(1,0)}) circle (0.75pt);
\filldraw[fill=white,draw=black]   (intersection cs:
    first line={(A)--(B)},
    second line={(0,0)--(1,1/3*-1/2)}) circle (0.75pt);
\filldraw[fill=white,draw=black]   (intersection cs:
    first line={(A)--(B)},
    second line={(0,0)--(1,1/3*-2/3)}) circle (0.75pt);
%%\filldraw[fill=white,draw=black]   (intersection cs:
%%    first line={(A)--(B)},
%%    second line={(0,0)--(1,1/2*-3/4)}) circle (0.75pt);
\end{tikzpicture}
\] 
The circles on the blue line are, reading top left to bottom right, at $\upvartheta_1=1,\frac{2}{3},\frac{1}{2},\frac{1}{3},0,-\frac{1}{3},-\frac{1}{2}$.  Thus basing $\Level$ by $[\scrP_1]$, the level is the infinite hyperplane arrangement
\[
\begin{tikzpicture}[scale=2]
\draw[densely dotted, blue,<-] (-2.5,0) -- (2.5,0);
\node[blue] at (-2.65,0) {$\scriptstyle \upvartheta_1$};
\filldraw[fill=white,draw=black] (0,0) circle (0.75pt);
\node at (0,-0.15) {$\scriptstyle 0$};
\filldraw[fill=white,draw=black] (-2*1/3,0) circle (0.75pt);
\node at (-2*1/3,-0.15) {$\scriptstyle \frac{1}{3}$};
\filldraw[fill=white,draw=black] (-2*1/2,0) circle (0.75pt);
\node at (-2*1/2,-0.15) {$\scriptstyle \frac{1}{2}$};
\filldraw[fill=white,draw=black] (-2*2/3,0) circle (0.75pt);
\node at (-2*2/3,-0.15) {$\scriptstyle \frac{2}{3}$};
\filldraw[fill=white,draw=black] (-2*1,0) circle (0.75pt);
\node at (-2*1,-0.15) {$\scriptstyle 1$};
%to right
\filldraw[fill=white,draw=black] (2*1/3,0) circle (0.75pt);
\node at (2*1/3,-0.15) {$\scriptstyle -\frac{1}{3}$};
\filldraw[fill=white,draw=black] (2*1/2,0) circle (0.75pt);
\node at (2*1/2,-0.15) {$\scriptstyle -\frac{1}{2}$};
\filldraw[fill=white,draw=black] (2*2/3,0) circle (0.75pt);
\node at (2*2/3,-0.15) {$\scriptstyle -\frac{2}{3}$};
\filldraw[fill=white,draw=black] (2*1,0) circle (0.75pt);
\node at (2*1,-0.15) {$\scriptstyle -1$};
\end{tikzpicture}
\]
The $\scrJ$-alcoves are the open intervals on the blue line between two adjacent dots, and  $\scrH^{\aff}$ is the infinite collection of dots.  
\end{exa}

\subsection{Finite Hyperplanes by K-theory}\label{finite notation subsection}
For the finite version of the above combinatorics, with notation as in Lemma~\ref{k-correspond proj} consider 
\[
\Uptheta_{L}\colonequals \scrK_L/[\scrP_0]\cong\bZ^n.
\]
Again, since $X\to\Spec R$ and $N$ are fixed from \eqref{ass modi}, there is a distinguished object $\Uptheta\colonequals \Uptheta_{N}$.  If $i\neq 0$, then since  $\upphi_{i}$ sends $\scrP_0$ to $\scrQ_0$ by Lemma~\ref{k-correspond proj}, $\upphi_{i}$ induces an isomorphism 
\[
\upvarphi_i\colon
\Uptheta_{L}
\to
\Uptheta_{\upnu_iL}.
\]
For $L\in\Mut_0(N)$, consider the shortest sequence of mutations
\[
L\xrightarrow{i_1}\upnu_{i_1}L\to\hdots\xrightarrow{i_n}N,
\]
where each step does not mutate the vertex $R$.  As before, write $\Upphi_L$ for the composition of the corresponding mutation functors, but now write $\upvarphi_L\colon \Uptheta_L\to\Uptheta$ for the induced map on K-theory.  Again, whenever $\bZ\subseteq\mathds{k}$, we will abuse notation and also write $\upvarphi_L\colon 
\Uptheta_L\otimes\mathds{k}\to\Uptheta\otimes\mathds{k}$.

The following was established first in \cite{HomMMP} when $X$ is $\bQ$-factorial, using King stability. The $\bQ$-factorial can now be dropped, following \cite{IW9}.

\begin{thm}\cite{HomMMP, IW9}\label{HomMMP finite summary}
Suppose that $f\colon X\to \Spec R$ is a $3$-fold flopping contraction, such that $X$ has only terminal singularities.
Then there is a finite decomposition
\[
\Cone{\scrJ} = \bigcup_{L\in\Mut_0(N)}\upvarphi_L(C_+)\subseteq\Uptheta\otimes\bR
\]
In particular, the following statements hold.
\begin{enumerate}
\item\label{HomMMP finite summary 1} The open decomposition $\bigcup\upvarphi_L(C_+)$ gives the chambers of the $\scrJ$-finite hyperplane arrangement.
\item\label{HomMMP finite summary 2} If $L\ncong M$, then $\upvarphi_{L}(C_+)$ and  $\upvarphi_{M}(C_+)$  do not intersect.
\item\label{HomMMP finite summary 3} $\upvarphi_{L}(C_+)$ and  $\upvarphi_{M}(C_+)$ share a codimension one wall $\iff$ $L$ and $M$ differ by the mutation of an indecomposable summand.
\end{enumerate}
\end{thm}

For $L\in \Mut_0(N)$, write $C_L\colonequals\upvarphi_{L}(C_+)$ and  set
\begin{equation}
\scrH\colonequals(\Uptheta\otimes\bR)\backslash\bigcup_{L\in\Mut_0(N)}C_L.\label{finite H}
\end{equation}
By Theorem~\ref{HomMMP finite summary}\eqref{HomMMP finite summary 1}, $\scrH$ is a finite simplicial hyperplane arrangement, which by definition means that $\bigcap_{H\in\scrH}H=\{0\}$ and all chambers in $\mathbb{R}^n\backslash\scrH$ are open simplicial cones.

\subsection{The Tracking Rules of Mutation}\label{New results S3}
Given a modifying $R$-module $L$ and any summand $L_i$,  $\upnu_i\upnu_iL\cong L$  by Proposition~\ref{basicmodi}\eqref{basicmodi 3}.  We will abuse notation and write
\[
\begin{tikzpicture}
\node at (0.35,0) {$\Db(\fmod\Lambda_L))$};
\node at (4.85,0) {$\Db(\fmod\Lambda_{\upnu_iL}).$};
\draw[->] (2,0.1) -- node[above] {$\scriptstyle \Upphi_i$}(3,0.1); 
\draw[<-] (2,-0.1) -- node[below] {$\scriptstyle \Upphi_i$}(3,-0.1);
\end{tikzpicture}
\]
These, and their inverses, induce the following isomorphisms on K-theory
\begin{equation}
\begin{array}{c}
\begin{tikzpicture}
\node at (1.5,0) {$\scrK_L$};
\node at (3.6,0) {$\scrK_{\upnu_iL}$};
\draw[->] (2,0.1) -- node[above] {$\scriptstyle \upphi_i$}(3,0.1); 
\draw[<-] (2,-0.1) -- node[below] {$\scriptstyle \upphi_i$}(3,-0.1);
\end{tikzpicture}
\end{array}
\qquad
\begin{array}{c}
\begin{tikzpicture}
\node at (1.5,0) {$\scrK_L$};
\node at (3.6,0) {$\scrK_{\upnu_iL}$};
\draw[<-] (2,0.1) -- node[above] {$\scriptstyle \upphi_i^{-1}$}(3,0.1); 
\draw[->] (2,-0.1) -- node[below] {$\scriptstyle \upphi_i^{-1}$}(3,-0.1);
\end{tikzpicture}
\end{array}\label{four matrices}
\end{equation}

\begin{lem}\label{invol}
All four isomorphisms in \eqref{four matrices} are given by the same matrix, and this matrix squares to the identity.  If $i\neq 0$, then the same statement holds for $\upvarphi_i,\upvarphi_i^{-1}$ and $\Uptheta_L,\Uptheta_{\upnu_iL}$.
\end{lem}
\begin{proof}
By \eqref{k-correspond}, the matrices for the inverses are controlled by numbers appearing in the relevant approximation sequences.  Suppose that the top $\upphi_i^{-1}$ is controlled by numbers $b_{ij}$, and the bottom $\upphi_i^{-1}$ is controlled by numbers $c_{ij}$.  That the two matrices labelled $\upphi_i^{-1}$ are the same is simply the statement that $b_{ij}=c_{ij}$, which is precisely the proof of \cite[5.22]{HomMMP} when $X$ is $\mathds{Q}$-factorial, or \cite[9.28]{IW9} generally.     Given the fact that $b_{ij}=c_{ij}$, the statement that $\upphi_i^{-1}\upphi_i^{-1}={\rm Id}$ can then simply be seen directly.  Applying $\upphi_i$ to each side then gives $\upphi_i^{-1}=\upphi_i$.  All the statements on $\upvarphi_i$ follow.
\end{proof}

\subsection{Complexified Actions}\label{complexified actions subsection}
Via \eqref{affine H} and \eqref{finite H}, associated to $X\to\Spec R$ is an infinite real hyperplane arrangement $\scrH^{\aff}$, and also a finite simplicial real hyperplane arrangement $\scrH$.  Stability conditions will require the complexified versions of these.

By a slight abuse of notation, consider
\begin{align*}
\bH_+&\colonequals \left\{ x+\ii y\in(\Uptheta_L)_{\bC}
\mid x_j+\ii y_j\in\bH \mbox{ for all } 1\leq j \leq n
\right\}
\cong \bH^n\\
\bH_+'&\colonequals \left\{ x+\ii y\in(\scrK_L)_{\bC}
\mid x_j+\ii y_j\in\bH \mbox{ for all } 0\leq j \leq n
\right\}
\cong \bH^{n+1}
\end{align*}
where $\bH=\{ r e^{\ii\uppi\upvartheta}\in\bC \mid r\in\bR_{>0},\, 0<\upvartheta\leq 1\}\subset \bC$ is the semi-closed upper half plane.  The regions $\bH_+$ and $\bH_+'$ technically depend on $L$, since they are subsets of $(\Uptheta_L)_{\bC}$ and $(\scrK_L)_{\bC}$ respectively, but we drop this from the notation.

For $L\in\Mut_0(N)$, recall from \S\ref{finite notation subsection} that after choosing a mutation path $L\to\hdots\to N$ that does not involve mutating $R$, there is a corresponding linear map $\upvarphi_{L}\colon (\Uptheta_L)_{\bC}\to \Uptheta_{\bC}$.  
We require the following result, where as usual $\scrH_{\bC}$ denotes the complexification of the real hyperplane arrangement $\scrH$.  The result is folklore when the arrangement $\scrH$ is Coxeter.  Given our setting here is just mildly more general, and the proof is combinatorial in nature, we give a self-contained proof in Appendix~\ref{comb appendix}.
\begin{prop}\label{complexified tracking}
There is an equality 
\[
(\Uptheta\otimes{\bC})\backslash\scrH_{\bC}=\bigcup_{L\in\Mut_0(N)}\upvarphi_L(\bH_+)
\]
where the union on the right hand side is disjoint.
\end{prop}

On the other hand, the affine version of Proposition~\ref{complexified tracking} is a little bit more involved. We first pass to the \emph{complexified level}, defined to be
\[
(\Level_{L})_{\bC}\colonequals \biggr\{ z\in (\scrK_L)_{\bC} \Bigm| \sum_{j=0}^n(\rk_RL_j) z_j=\ii\biggr\},
\]
and inside $(\Level_{L})_{\bC}$ consider the region
\[
\bE_+\colonequals\biggr\{ z\in \bH_+' \Bigm| \sum_{j=0}^n(\rk_RL_j) z_j=\ii\biggr\}.
\]
\begin{exa}\label{Ex3.11}
In the case of any one-curve flop, writing $z=x+\ii y$, then  $\mathbb{C}\backslash\scrH_\bC=\mathbb{C}\backslash\{0\}$ decomposes into the disjoint union
\[
\begin{array}{ccccc}
\begin{array}{c}
\begin{tikzpicture}
\filldraw[gray!10!white] (-1.5,1) -- (-1.5,0) -- (1.5,0)--(1.5,1)-- cycle;
\draw[black] (-1.5,0) -- (0,0);
\draw[black,densely dotted] (0,0) -- (1.5,0);\draw[draw=none,blue,densely dotted] (0,1) -- (0,-1);
\filldraw[fill=white,draw=black] (0,0) circle (2pt);
\end{tikzpicture}
\end{array}
&
\bigcup
&
\begin{array}{c}
\begin{tikzpicture}
\filldraw[gray!10!white] (-1.5,-1) -- (-1.5,0) -- (1.5,0)--(1.5,-1)-- cycle;
\draw[black,densely dotted] (-1.5,0) -- (0,0);
\draw[black] (0,0) -- (1.5,0);\draw[draw=none,blue,densely dotted] (0,1) -- (0,-1);
\filldraw[fill=white,draw=black] (0,0) circle (2pt);
\end{tikzpicture}
\end{array}
&&
\begin{array}{c}
\begin{tikzpicture}
\draw[blue,<-] (4,1) -- (4,-1);
\draw[blue,->] (3,0) -- (5,0);
\node[blue] at (4.25,1) {$\scriptstyle y$};
\node[blue] at (5.15,0) {$\scriptstyle x$};
\end{tikzpicture}
\end{array}
\\
\bH_+&&
\upvarphi_{1}(\bH_+)
\end{array}
\]
On the other hand, as in Example~\ref{cones and level example}, for any one-curve flop $\scrH^{\aff}_\bC$ consists of infinitely many points on the real axis.  To exhibit the region $\mathbb{E}_+$, note first that $(z_0,z_1)\in\Level_{\bC}$ if and only if we can write 
\[
(z_0,z_1)=((-\ell x_1,1-\ell y_1),(x_1,y_1)),
\]
where $\ell$ is the length of the curve.
To belong to $\bE_+$ is equivalent to both factors being in $\bH$.  If the second factor is in $\bH$ then $y_1\geq 0$, so the first factor being in $\bH$ implies that $0\leq y_1\leq \frac{1}{\ell}$. In fact, it is elementary to check that $\bE_+$ forms the following region:
\[
\begin{tikzpicture}[>=stealth]
\draw[draw=none,blue,densely dotted] (-2.5,0) -- (1.5,0);
\draw[draw=none,blue,densely dotted] (-1,0) -- (1,0);
\filldraw[gray!10!white] (-1,1) -- (-1,-1) -- (0,-1)--(0,1)-- cycle;
\draw[black] (-1,0) -- (-1,1);
\draw[black,densely dotted] (-1,0) -- (-1,-1);
\draw[black] (0,0) -- (0,-1);
\draw[black,densely dotted] (0,0) -- (0,1);
{\foreach \i in {-1,0,1,2}
\filldraw[fill=white,draw=black] (-1*\i,0) circle (2pt);
}
\node at (0.15,-0.15) {$\scriptstyle 0$};
\node at (-1.15,-0.15) {$\scriptstyle \frac{1}{\ell}$};
\draw[blue,<-] (4,1) -- (4,-1);
\draw[blue,<-] (3,0) -- (5,0);
\node[blue] at (4.25,1) {$\scriptstyle x_1$};
\node[blue] at (2.75,0) {$\scriptstyle y_1$};
\end{tikzpicture}
\]
The non-standard way of drawing the $x$ and $y$ axis is justified by Example~\ref{cones and level example}.  The co-ordinate $y_1$ should be viewed as the original $\Level$ seen  in Example~\ref{cones and level example}, which naturally points to the left, and $x_1$ should be viewed as the `complexified co-ordinate'.  The other regions  $\upvarphi_L(\mathbb{E}_+)$ have the same shape as the above, sandwiched between the two adjacent dots, and so give a disjoint union that covers $(\Level_{L})_{\bC}$.
\end{exa}

Set $\scrW$ to be the set of full hyperplanes in $\scrK\otimes{\bR}$ that separate the open chambers $\upphi_L(C_+)$ of $\Cone{\scrJ_{\aff}}$ (see e.g.\ Example~\ref{affine hypes picture}).  We then consider the complexification of $\scrH^{\aff}$ in $\Level_{\bC}$, defined to be
\begin{equation}\label{cpx of Haff}
\scrH^{\aff}_{\bC}\colonequals \scrW_{\mathbb{C}}\cap \Level_{\bC}= \bigcup_{W\in \scrW}(W_{\bC}\cap \Level_{\bC}),
\end{equation}
where $W_{\bC}\colonequals W\oplus\ii W$.  As in Example~\ref{Ex3.11},  $\scrH^{\aff}_{\bC}$ can be viewed as the complexification of hyperplanes in the real level, provided that we swap the roles of $x$ and $y$. Indeed, if we set $H_{W}\colonequals W\cap\Level$,  then $\scrH^{\aff}=\bigcup_{W\in\scrW}H_{W}$, and the linear bijection $\Level_{\bC}\to \Level\oplus\ii\Level$ defined by $x+\ii y\mapsto (x+y)+\ii y$ maps $\scrH^{\aff}_{\bC}$ to $\bigcup_{W\in\scrW}(H_W\oplus\ii H_W)$.

The following two results are evident, by inspection, for any one-curve flops using Example~\ref{Ex3.11} above.  The more general case requires a more involved combinatorial argument, so again the proofs are postponed until Appendix~\ref{comb appendix}.

\begin{lem}[\ref{path connected app}]\label{path connected}
The subspace $\bE_+\subset (\scrK_L)_{\bC}$ is path connected.
\end{lem}

\begin{prop}\label{complexified tracking 2}There is an equality 
\[
\Level_{\bC}\!\backslash\scrH_{\bC}^{\aff}=\bigcup_{L\in\Mut(N)}\upphi_L(\bE_+),
\]
where the union on the right hand side is disjoint.
\end{prop}

\section{Arrangement Groupoids}

In this subsection we briefly recall the basics of the arrangement (=Deligne) groupoid of a locally finite real hyperplane arrangement, mainly to set notation.  Some first results specific to the flops setting are presented in Subsection~\ref{first results}.

\subsection{Arrangements groupoids}\label{arrangement prelim subsection}

Throughout, let $\cH$ be either the finite real hyperplane arrangement $\scrH$ from \eqref{finite H}, or the infinite version $\scrH^{\aff}$ from \eqref{affine H}. Both are \emph{locally finite} arrangements, i.e.\ every point of $\mathbb{R}^n$ is contained in at most finitely many hyperplanes, and \emph{essential arrangements}, i.e.\ the minimal intersections of hyperplanes are points. 

\begin{dfn} \label{def:Gamma graph}
The graph $\Gamma_{\cH}$ of oriented arrows is defined as follows.  The vertices of $\Gamma_{\cH}$ are the chambers (i.e. the connected components) of $\bR^n\backslash\cH$. There is a unique arrow $a \colon  v_1\to v_2$ from chamber $v_1$ to chamber $v_2$ if the chambers are adjacent, otherwise there is no arrow. For an arrow $a\colon  v_1\to v_2$, we set $s(a)\colonequals  v_1$ and $t(a)\colonequals  v_2$. By definition, if there is an arrow $a \colon  v_1\to v_2$, there is a unique arrow $b\colon  v_2\to v_1$ with the opposite direction of $a$. 
\end{dfn}

A \emph{positive path of length~$n$} in $\Gamma_{\cH}$ is defined to be a formal symbol
\[
p=a_n\circ \hdots\circ a_2\circ a_1,
\]
 whenever there exists a sequence of vertices $v_0,\hdots,v_n$ of $\Gamma_{\cH}$ and exist arrows $a_i\colon v_{i-1}\to v_i$ in $\Gamma_{\cH}$. Set $s(p)\colonequals  v_0$, $t(p)\colonequals  v_n$, and $\ell(p)\colonequals  n$, and write $p\colon s(p)\to t(p)$. The notation $\circ$ should remind us of composition, but we will often drop the $\circ$'s in future.  If 
$q=b_m\circ\hdots\circ b_2 \circ b_1$ is another positive path with $t(p)=s(q)$, we consider the formal symbol
\[
q\circ p\colonequals  b_m\circ\hdots\circ b_2 \circ b_1
\circ
a_n\circ \hdots\circ a_2\circ a_1,
\]
and call it the {\it composition} of $p$ and $q$.

\begin{dfn}
A positive path is called \emph{reduced} if it does not cross any hyperplane twice. 
\end{dfn}
In our setting where $\cH$ is $\scrH$ or $\scrH^{\aff}$, reduced positive paths coincide with shortest positive paths.  In the finite setting this can be found in e.g.\ \cite[4.2]{Paris}, and in the infinite case see e.g.\ \cite[Lemma 2]{Salvetti} or \cite[\S I.5]{IM}.

Following \cite[p7]{Delucchi}, let $\sim$ denote the smallest equivalence relation, compatible with morphism composition, that identifies all morphisms that arise as positive reduced paths with same source and target. Then consider the free category $\mathrm{Free}(\Gamma_\cH)$ on the graph $\Gamma_\cH$, where morphisms are directed paths, and the quotient category 
\[
\cG^{+}_\cH \colonequals \mathrm{Free}(\Gamma_\cH)/\sim, 
\]
called the category of positive paths.

\begin{dfn} \label{def: Deligne groupoid completion}
The \emph{arrangement (=Deligne) groupoid}  $\cG_{\cH}$ is the groupoid defined as the groupoid completion of  $\cG_{\cH}^{+}$, that is, a formal inverse is added for every morphism in $\cG_{\cH}^{+}$. 
\end{dfn} 

%For a morphism $\upalpha$ in $\cG_{\cH}$,  there are arrows $a_1,\hdots,a_n$ in $\Gamma_{\cH}$ such that 
%\[
%\upalpha=a_n^{\varepsilon_n}\circ \cdots\circ a_1^{\varepsilon_1},
%\] 
%where $\varepsilon_i\in\{\pm1\}$. If $\varepsilon_1=+1$ (resp. $\varepsilon_1=-1$), the source  of $\upalpha$, denoted by $s(\upalpha)$, is defined to be $s(a_1)$ (resp. $t(a_1)$), and the target of $\upalpha$ is defined to be $t(a_1)$ (resp. $s(a_1)$). 
\begin{nota}\label{not: 4.4}
When $\cH=\scrH$ from \eqref{finite H}, we denote the arrangement groupoid by $\mathds{G}$, and when $\cH=\scrH^{\aff}$ from \eqref{affine H}, we denote the arrangement groupoid by $\mathds{G}^{\aff}$.
\end{nota}

The following is well-known \cite{Deligne, Paris, Paris3, Salvetti}; in the level of generality here with $\scrH$ locally finite and essential, the statement below is \cite[p9]{Delucchi}.

\begin{thm}\label{ver gp}
For any vertex $v$ in the arrangement groupoid, $\End_{\mathds{G}}(v)\cong\uppi_1(\mathbb{C}^n \backslash \scrH_\mathbb{C})$, and $\End_{\mathds{G}^{\aff}}(v)\cong\uppi_1(\Level_{\bC}\!\backslash \scrH^{\aff}_\mathbb{C})$.
\end{thm}

\subsection{First Results for Flops}\label{first results}
This section proves that in our flops setting, pure braids act as the identity on K-theory. This will be crucial in showing that they act as deck transformations later in Section~\ref{stab on C and D section}.    Throughout this subsection, $\cH$ is either $\scrH$ or $\scrH^{\aff}$, and  $C_L$ denotes the chamber in the complement of $\cH$ corresponding to either $\upvarphi_L(C_+)$ or $\Alcove_L$, when $\cH=\scrH$ or $\cH=\scrH^{\aff}$ respectively.  The following two lemmas are elementary.

\begin{lem}\label{Lemma 1}
Suppose that $\upalpha\colon A\to B$ is a reduced positive path for $\cH$, and that $s_{i}$ is a simple wall crossing separating $B$ and $C$, so by definition there are morphisms $s_i\colon B\to C$ and $s_i\colon C\to B$ in $\Gamma_{\cH}$.  If $s_i\circ\upalpha\colon A\to C$ is not reduced, then there exists some reduced positive path $\upgamma\colon A\to C$ such that $s_{i}\circ\upgamma\colon A\to B$ is reduced.
\end{lem}
\begin{proof}
This is very similar to \cite[5.1]{HW}.
\end{proof}

\begin{lem}\label{Lemma 2}
Suppose that $\upalpha=s_{i_t}\circ\hdots\circ s_{i_1}$ is a reduced positive path for $\cH$, namely
 \[
\upalpha=(C_{1}\xrightarrow{s_{i_1}}C_{2}\xrightarrow{s_{i_2}}\hdots C_{t-1}\xrightarrow{s_{i_t}}C_t)
\] 
where each $s_{i_j}$ crosses a hyperplane $H_j$, say. Then for all $j=1,\hdots,t-1$, the chambers $C_{1}$ and $C_{j}$ are on the same side of $H_j$.
\end{lem}
\begin{proof}
There is nothing to prove in the case $j=1$.  If $C_1$ and $C_j$ are on opposite sides of $H_j$ for some $j>1$, then clearly $\upbeta\colonequals s_{i_{j-1}}\circ\hdots\circ s_{i_1}\colon C_1\to C_j$ must at some point cross $H_j$.  But then $s_{i_j}\circ\upbeta$, and hence $\upalpha$, must cross $H_j$ twice, which is a contradiction. 
\end{proof}

With respect to our applications, part (2) of the following proposition is crucial. For any positive $\upalpha\in\Gamma_{\scrH^{\aff}}$, say $\upalpha={s_{i_t}}\circ\hdots\circ {s_{i_1}}$ we associate the functor
\[
\Upphi_\upalpha\colonequals \Upphi_{i_t}\circ\hdots\circ\Upphi_{i_1}.
\]
Now there is an order $\geq$ defined on tilting modules (see e.g.\ \cite[\S3]{HW}).  By Lemma~\ref{Lemma 2}, this order decreases along reduced paths, and so exactly as in \cite[4.6]{HW} (see Remark~\ref{pos min are functorial} below), we see that $\Upphi_\upalpha\cong\Upphi_\upbeta$ for any two reduced positive paths with the same start and end points.  Hence the association $\upalpha\mapsto\Upphi_\upalpha$ descends to a functor 
from $(\mathds{G}^{\aff})^+$.  As $\Upphi_\upalpha$ is already invertible, this in turn formally descends to a functor from $\mathds{G}^{\aff}$.    

The same analysis holds for the finite situation $\mathds{G}$.  In both cases, for any $\upalpha$ in the arrangement groupoid, we thus have an associated functor $\Upphi_{\upalpha}$, and its image $\upphi_\upalpha$ on K-theory $\scrK$, respectively $\upvarphi_\upalpha$ on $\Uptheta$.

\begin{prop}\label{theta trivial} 
Choose a reduced positive path $\upbeta\colon C_L\to C_M$ for $\scrH^{\aff}$.\begin{enumerate}
\item\label{theta trivial 1} If $\upalpha\colon C_L\to C_M$ is any positive path in $\Gamma_{\scrH^{\aff}}$, then $\upphi_{\upalpha}=\upphi_{\upbeta}$. 
 \item\label{theta trivial 2} If $p\in\End_{\dsG^{\aff}}(C_L)$, then $\upphi_{p}=\Id_{\scrK_L}$.
 \item\label{theta trivial 3} If $p,q\in\Hom_{\dsG^{\aff}}(C_L,C_M)$, then $\upphi_{p}=\upphi_{q}$.
 \end{enumerate}
 The same statements hold for $\scrH$, replacing $\mathds{G}^{\aff}$ by $\mathds{G}$, and $\upphi$ by $\upvarphi$.
\begin{proof}
We will establish all statements over $\bZ$, as then all the statements over $\mathds{k}$ follow.\\
(1) By the discussion above the Proposition, we know that if $\upalpha$ is furthermore reduced, then $\Upphi_\upalpha\cong\Upphi_\upbeta$, and so in particular $\upphi_\upalpha=\upphi_\upbeta$ holds.  Hence we can assume that $\upalpha$ is not reduced.

Consider the first time that $\upalpha={s_{i_t}}\circ\hdots\circ {s_{i_1}}$ crosses a hyperplane twice.  So, say $\upbeta\colonequals s_{i_{m-1}}\circ\hdots\circ s_{i_1}$ is reduced, but $s_{i_m}\circ\hdots\circ s_{i_1}$ is not. Pictorially
\[
\begin{tikzpicture}
\node (1) at (-1,-1) {$C_L=C_1$};
\node (m-1) at (2,0.5) {$C_{m-1}$};
\node (m) at (2.5,-0.5) {$C_{m}$};
\draw[densely dotted] (-2,0) -- (4,0);
\node at (4.25,0) {$H$};
\draw[color=black,rounded corners=15pt,->] 
 (1)  --($(1)+(0.5,2)$) -- ($(1)+(1.5,1)$) --  node[left]{$\scriptstyle\upbeta$}($(1)+(2,2.5)$)   -- (m-1);
\draw[->] (m-1) to node[right] {$\scriptstyle s_{i_m}$} (m);
\end{tikzpicture}
\]
By Lemma~\ref{Lemma 1} we can find a positive reduced path $\upgamma\colon C_1\to C_m$ such that the composition $C_1\xrightarrow{\upgamma} C_m\xrightarrow{s_{i_m}}C_{m-1}$ is reduced.  As $\upbeta$ and $s_{i_m}\circ\upgamma$ are reduced paths with the same start and end points, functorially they are the same, so
\begin{align*}
\Upphi_\upalpha
&= \Upphi_{i_t}\circ\hdots\circ\Upphi_{i_m}\circ(\Upphi_{i_{m-1}}\circ\hdots\circ\Upphi_{i_1})\\
&\cong \Upphi_{i_t}\circ\hdots\circ\Upphi_{i_m}\circ(\Upphi_{i_{m}}\circ\Upphi_\upgamma)
\end{align*}
Passing to K-theory, using the fact that $\upphi_{i_{m}}\upphi_{i_{m}}=\Id$ by Lemma~\ref{invol}, we see
\[
\upphi_\upalpha=\upphi_{i_t}\circ\hdots\circ\upphi_{i_{m+1}}\circ\upphi_\upgamma.
\]
Consider next the first time that $s_{i_t}\circ\hdots\circ s_{i_{m+1}}\circ\upgamma$ crosses a hyperplane twice.   Since $\upgamma$ is reduced, we move further to the left.  Applying the above argument repeatedly, by induction we end up in the case of a reduced path, and hence  $\upphi_\upalpha=\upphi_\upbeta$.

\noindent
(2) Say $\upphi_p=\upphi_{i_n}^{\pm 1}\circ \hdots \circ \upphi_{i_1}^{\pm 1}$ for some choice of superscripts $\pm 1$.  By Lemma~\ref{invol}, $\upphi_p=\upphi_{q}$, where $q\colonequals s_{i_n}\circ \hdots \circ s_{i_1}$. This is a positive path, with start and end $C_L$, so by part (1) it follows that $\upphi_p=\upphi_q=\Id$.

\noindent
(3)  This follows by applying (2) to  $q^{-1}p\in\End_{\dsG_{\cH}}(C_L)$.
\end{proof}
\end{prop}

\begin{rem}
Proposition~\ref{theta trivial} also implies that Theorems~\ref{affine summary} and \ref{HomMMP finite summary}, and also Propositions~\ref{complexified tracking} and \ref{complexified tracking 2}, can be indexed over reduced positive paths terminating at $C_+$.
\end{rem}

\begin{rem}\label{pos min are functorial}
Implicit in the above analysis is the fact, proved in \cite[4.6]{HW} in the finite case and \cite[9.34]{IW9} in the infinite case, that if $\upalpha\colon L\to M$ is a positive minimal path, then $\Upphi_\upalpha\cong\RHom_{\Lambda_L}(\Hom_R(L,M),-)$.  Hence any two positive minimal paths give rise to isomorphic functors.
\end{rem}

\section{Stability Conditions, Tilting and t-structure Transfer}

\subsection{Generalities on stability conditions}\label{sec: stab generalities}
We will not give a full summary of stability conditions here; see for example \cite{B07} or the survey \cite{Bayer}.  For our purposes, the following suffices.
Throughout this subsection, $\scrT$ denotes a triangulated category for which the Grothendieck group $K_0(\scrT)$ is a finitely generated free $\bZ$-module.

\begin{prop}[{\cite[5.3]{B07}}]\label{stab prop}
To specify a stability condition on $\scrT$ is equivalent to specifying a bounded t-structure  on $\scrT$ with heart $\scrA$, together with a stability function $Z$ on $\scrA$ that satisfies  the Harder-Narasimhan property.
\end{prop}

As usual,  in fact we will only study \emph{locally finite} stability conditions, and we let $\Stab\scrT$ denote the set of locally finite stability conditions on $\scrT$.  There is a topology on $\Stab\scrT$, induced by a natural metric.  

\begin{thm}[{\cite[1.2]{B07}}]\label{stab thm}
The space $\Stab\scrT$ has the structure of a complex manifold, and the forgetful map
\[
\uppi\colon \Stab\scrT\to\Hom_{\bZ}(K_0(\scrT),\bC)
\]
is a local isomorphism onto an open subspace of $\Hom_{\bZ}(K_0(\scrT),\bC)$.
\end{thm}

\begin{rem}
In the components $\cStab{}\scrC$ and $\cStab{}\scrD$ we study in the flops setting below, all stability conditions will automatically be \emph{full}, in the sense that they are always modelled on the whole of $\Hom_{\bZ}(K_0(\scrT),\bC)$.  In particular, our stability conditions will automatically satisfy the \emph{support property}, see e.g.\ \cite[Appendix B]{BMacri}.  We will freely use this throughout. 
\end{rem}

An exact equivalence of triangulated categories $\Phi\colon \scrT\to \scrT'$ induces a natural map 
\[
\Phi_*\colon \Stab\scrT\to\Stab\scrT'
\] 
defined by $\Phi_*(Z,\cA)\colonequals (Z\circ \upphi^{-1},\Phi(\cA))$, where as before $\upphi^{-1}$ denotes the isomorphism on K-theory $K_0(\scrT')\simto K_0(\scrT)$ induced by the functor $\Phi^{-1}$, and $\Phi(\cA)$ denotes its essential image. As usual, if two exact equivalences $\Phi\colon \scrT\to\scrT'$ and $\Psi\colon \scrT\to\scrT'$ are naturally isomorphic, then $\Phi_*(Z,\cA)=\Psi_*(Z,\cA)$ for any $(Z,\cA)\in\Stab\scrT$, and thus the group $\Auteq(\scrT)$ of isomorphism classes of autoequivalences of $\scrT$  acts on $\Stab\scrT$.

\subsection{Stability, Normalisation and Mutations}\label{sec:stab normal and mut}
We return to the setting where $f\colon X\to\Spec R$ is the flopping contraction as in the introduction, with distinguished $R$-module $N$ from \eqref{ass modi}, $\Lambda\colonequals\End_R(N)$, and K-theory $\scrK$ and $\Uptheta$ from Sections \ref{affine notation subsection} and \ref{finite notation subsection} respectively.

For any $L=\bigoplus_{i=0}^nL_i\in \Mut(N)$, with the ordering on summands induced from $N$ as explained under Definition~\ref{exchange def}, let $\scrS_i$ be the simple $\Lambda_L$-module corresponding to the projective $\scrP_i=\Hom(L,L_i)$. Write $\scrB_L$ for the subcategory of $\fmod\Lambda_L$ consisting of finite-length modules.  If $L\in \Mut_0(N)$, then we further write $\scrA_L$ for the full subcategory of $\scrB_L$ of those finite-length modules whose simple factors are not isomorphic to $\scrS_0$.

 Consider the triangulated subcategories
 \begin{align*}
\scrC_L&\colonequals \{a\in \Db(\fmod\Lambda_L) \mid \mathrm{H}^i(a)\in \scrA_L \mbox{ for all $i$} \},\\
\scrD_L&\colonequals \{b\in \Db(\fmod\Lambda_L) \mid \mathrm{H}^i(b)\in \scrB_L \mbox{ for all $i$} \}.
\end{align*}
Since $\scrA_L$ and $\scrB_L$ are extension closed abelian subcategories of $\fmod\Lambda_L$, the standard t-structure on $\Db(\fmod\Lambda_L)$ restricts to  a bounded t-structure on $\scrD_L$ with heart $\scrB_L$, and a bounded t-structure on $\scrC_L$ with heart $\scrA_L$.

The categories $\scrA_L$ and $\scrB_L$ have finitely many simple objects, and so
\[
K_0(\scrC_L)\cong  \bigoplus_{i=1}^n\bZ[\scrS_i]
\qquad
K_0(\scrD_L)\cong  \bigoplus_{i=0}^n\bZ[\scrS_i].
\]
There are canonical isomorphisms
\[
\Hom_{\bZ}(K_0(\scrC_L),\bC)\simto (\Uptheta_L)_\bC
\qquad
\Hom_{\bZ}(K_0(\scrD_L),\bC)\simto (\scrK_L)_\bC
\]
given by $\upgamma\mapsto\sum \upgamma( [\scrS_i])[\scrP_i]$.  Composing these with the local homeomorphism  in Theorem~\ref{stab thm} defines local homeomorphisms
\[
\cZ_L\colon \Stab\scrC_L\to (\Uptheta_L)_{\bC}
\qquad
\cZ_L\colon \Stab\scrD_L\to (\scrK_L)_{\bC}.
\]
Write $\Stab\scrA_L$ for the stability functions on $\scrA_L$ which satisfy the Harder--Narasimhan property, then by Proposition \ref{stab prop} $\Stab\scrA_L$ can be regarded as a subspace of  $\Stab\scrC_L$.  Similarly for $\Stab\scrB_L$, which is a subspace of $\Stab\scrD_L$.  It follows from \cite[5.2]{B07} that the above local homeomorphisms restrict to  isomorphisms
\begin{equation}
\cZ_L\colon \Stab\scrA_L\simto\bH_+
\qquad
\cZ_L\colon \Stab\scrB_L\simto\bH'_+.\label{local homeo A}
\end{equation}

Applying all of the above to $N\in\Mut(N)$, it will be convenient to suppress $N$ from the notation, so write $\scrD\colonequals \scrD_N$,  $\scrB\colonequals \scrB_N$, $\cZ\colonequals \cZ_N$, etc.

There is a $\mathbb{C}$-action on $\Stab\scrD$, which later we will avoid.  As such, following \cite{B3}, for any $L=\bigoplus_{i=0}^nL_i\in \Mut(N)$, consider $\nStab{n}\scrD_L$ to be those stability conditions in $\Stab\scrD_L$ for which the central charge $Z$ satisfies
\[
\sum_{j=0}^{n}(\rk_RL_j)\,Z[\scrS_{j}]
=
\ii.
\]
We call such stability conditions \emph{normalised}.
In particular
\[
\cZ_L\colon \nStab{n}\scrD_L\to (\Level_{L})_{\bC},
\]
where the complexified level is defined in \S\ref{complexified actions subsection}.  Set $\nStab{n}\scrB_L\colonequals \Stab\scrB_L\cap \,\nStab{n}\scrD_L$, then the latter isomorphism in \eqref{local homeo A} restricts to an isomorphism
\[
\cZ_L\colon\nStab{n}\scrB_L\simto \bE_+.
\]
\begin{prop} \label{stability track comm diagram}
Let $L\in\Mut_0(N)$, respectively $L\in\Mut(N)$.  Then the following diagrams commute.
\[
\begin{array}{ccc}
\begin{array}{c}
\begin{tikzpicture}
\node (A1) at (0,0) {$\Stab\scrA_L$};
\node (A2) at (3,0) {$\Stab\scrC$};
\node (B1) at (0,-1.5) {$\bH_+$};
\node (B2) at (3,-1.5) {$\Uptheta_{\bC}$};
\draw[->] (A1) -- node[above] {\scriptsize $(\Upphi_{\kern -1pt L})_*$} (A2);
\draw[->] (B1) -- node[above] {\scriptsize $\upvarphi_L$} (B2);
\draw[->] (A1) -- node[left] {\scriptsize $\cZ_{L}$}node[right] {\scriptsize $\sim$} (B1);
\draw[->] (A2) -- node[right] {\scriptsize $\cZ$} (B2);
\end{tikzpicture}
\end{array}
&&
\begin{array}{c}
\begin{tikzpicture}
\node (A1) at (0,0) {$\Stab\scrB_L$};
\node (A2) at (3,0) {$\Stab\scrD$};
\node (B1) at (0,-1.5) {$\bH'_+$};
\node (B2) at (3,-1.5) {$\scrK_{\bC}$};
\draw[->] (A1) -- node[above] {\scriptsize $(\Upphi_{\kern -1pt L})_*$} (A2);
\draw[->] (B1) -- node[above] {\scriptsize $\upphi_L$} (B2);
\draw[->] (A1) -- node[left] {\scriptsize $\cZ_{L}$}node[right] {\scriptsize $\sim$} (B1);
\draw[->] (A2) -- node[right] {\scriptsize $\cZ$} (B2);
\end{tikzpicture}
\end{array}
\end{array}
\]
The latter diagram restricts to a commutative diagram
\[
\begin{array}{c}
\begin{tikzpicture}
\node (A1) at (0,0) {$\nStab{n}\scrB_L$};
\node (A2) at (3,0) {$\nStab{n}\scrD$};
\node (B1) at (0,-1.5) {$\bE_+$};
\node (B2) at (3,-1.5) {$\Level_{\kern 1pt \bC}$};
\draw[->] (A1) -- node[above] {\scriptsize $(\Upphi_{\kern -1pt L})_*$} (A2);
\draw[->] (B1) -- node[above] {\scriptsize $\upphi_L$} (B2);
\draw[->] (A1) -- node[left] {\scriptsize $\cZ_{L}$}node[right] {\scriptsize $\sim$} (B1);
\draw[->] (A2) -- node[right] {\scriptsize $\cZ$} (B2);
\end{tikzpicture}
\end{array}
\]
\end{prop}
\begin{proof}
To ease notation, set $\scrQ_i\colonequals \Hom_R(L,L_i)$ and write $\scrS_i'$ for the simple $\Lambda_L$-module corresponding to the projective $\Lambda_L$-module $\scrQ_i$.   Similarly, write $\scrP_i=\Hom_R(N,N_i)$, and $\scrS_i$ for the corresponding simple $\Lambda$-module.

Consider the perfect pairing 
\[
\upchi(-,-)\colon \Uptheta\times K_0(\scrC)\to\bZ
\]
given by $\upchi(a,b)\colonequals \sum_{i\in\bZ}(-1)^i\dim_{\bC}\Hom(a,b[i])$.
Since the pairing is perfect, setting  
\[
a_{ij}\colonequals \upchi(\Upphi_{L}(\scrQ_i),\scrS_j)\in \bZ
\]
implies that $[\Upphi_{L}(\scrQ_i)]=\sum_{j=1}^na_{ij}[\scrP_j]$ in $\Uptheta$.
Furthermore, since $\Upphi_{L}^{-1}$ is right adjoint to $\Upphi_{L}$, 
\[
a_{ij}= \upchi(\scrQ_i,\Upphi_{L}^{-1}(\scrS_j)),
\]
which in turn implies that $[\Upphi_{L}^{-1}(\scrS_j)]=\sum_{i=1}^na_{ij}[\scrS_i']$ in $K_0(\scrC)$. 
Therefore, for any point  $\upsigma=(Z,\cP)\in\Stab\scrA_L$,  necessarily
\[
\upvarphi_L(\cZ_{L}(\upsigma))
=\sum_{i=1}^{n}Z(\scrS_i')[\Upphi_{L}(\scrQ_i)]
=\sum_{j=1}^{n}\Big(\sum_{i=1}^na_{ij}Z(\scrS'_i)\Big)[\scrP_j]=\cZ({\Upphi_{L}}_*(\upsigma)).
\]
The last diagram follows immediately, as mutation functors in $K$-theory take $\sum (\rk_RL_i)[\scrS_i]$ to $\sum (\rk_RN_i)[\scrS_i]$ and thus preserve the normalisation.
\end{proof}

\subsection{Tilting at Simples via Mutation}

By the above, $\scrA_L\subset \scrC_L$  and $\scrB_L\subset \scrD_L$ are the hearts of bounded t-structures, with finitely many simples.  Each of these simple objects induces two torsion theories, $(\l \scrS_i \r, \scrF_i)$  and $(\scrT_i,\l \scrS_i \r)$, where $\l \scrS_i \r$ is the full subcategory of objects whose simple factors are isomorphic to $\scrS_i$.  In the case of $\scrA_L$, the subcategories $\scrF_i$ and $\scrT_i$ are defined by 
\begin{align*}
\scrF_i&\colonequals \{a\in\scrA_L\mid\Hom_{\scrA_L}(\scrS_i,a)=0\}\\
\scrT_i&\colonequals \{a\in\scrA_L\mid\Hom_{\scrA_L}(a,\scrS_i)=0\},
\end{align*}
and the corresponding tilted hearts are defined by
\begin{align*}
{\rm L}_{i}(\scrA_L)&\colonequals 
\{ c\in\scrC_L\mid 
{\rm H}^k(c)=0 \mbox{ for } k\notin\{0,1\},\, {\rm H}^0(c)\in\scrF_i, \, {\rm H}^1(c)\in\l \scrS_i\r \}\\
{\rm R}_{i}(\scrA_L)&\colonequals 
\{ c\in\scrC_L \mid {\rm H}^k(c)=0 \mbox{ for } k\notin\{-1,0\}, \, {\rm H}^{-1}(c)\in\l \scrS_i\r, \,{\rm H}^0(c)\in\scrT_i \},
\end{align*}
where ${\rm H}^i(-)$ is the cohomological functor associated to the standard t-structure on $\scrC_L$ defining $\scrA_L$.  A similar picture applies in the case of $\scrB_L$.

\begin{lem}\label{simple tilt}
We have ${\rm L}_{i}(\scrA_{\upnu_iN})=\Upphi_{i}(\scrA)$ and ${\rm R}_{i}(\scrA)=\Upphi_{i}^{-1}(\scrA_{\upnu_iN})$ for all $i=1,\hdots,n$.  The same statements hold replacing $\scrA$ by $\scrB$, for all $i=0,1,\hdots,n$
\begin{proof}
We will only show that ${\rm L}_{i}(\scrA_{\upnu_iN})=\Upphi_{i}(\scrA)$, since the other proof is similar.  Since both categories are hearts of bounded t-structures, it suffices to show that $\Upphi_{i}(\scrA)\subseteq{\rm L}_{i}(\scrA_{\upnu_iN})$. For this, since $\scrA$ is finite length, it is enough to show that $\Upphi_{i}(\scrS_j)\in{\rm L}_{i}(\scrA_{\upnu_iN})$ for all $1\leq j\leq n$.

If $j=i$, since  $\Upphi_{i}(\scrS_i)=\scrS_i[-1]$ by \cite[4.15(2)]{HomMMP}, it follows that $\Upphi_{i}(\scrS_i)\in{\rm L}_{i}(\scrA_{\upnu_iN})$. Hence we can assume that $j\neq i$.  Then $\Upphi_{i}(\scrS_j)\in\scrA_{\upnu_iN}$ by \cite[4.4]{HW}. Hence $\Upphi_{i}(\scrS_j)\cong{\rm H}^0(\Upphi_{i}(\scrS_j))$, and so $\Upphi_{i}(\scrS_j)$ is only in degree zero, and further
\begin{align*}
\Hom_{\Db({\upnu_i\Lambda})}(\scrS_i,{\rm H}^0(\Upphi_{i}(\scrS_j)))&\cong\Hom_{\Db(\Lambda)}(\scrS_i,\Upphi_{i}(\scrS_j))\\
&\cong\Hom_{\Db(\Lambda)}(\Upphi_{i}^{-1}(\scrS_i),\scrS_j)\\
&\cong\Hom_{\Db(\Lambda)}(\Tor_1^{\upnu_i\Lambda}(\scrS_i,T_i)[1],\scrS_j)\tag{by \cite[(4.B)]{HW}}\\
&=\Ext^{-1}_{\Lambda}(\Tor_1^{\upnu_i\Lambda}(\scrS_i,T_i),\scrS_j)=0.
\end{align*}
Combining, it follows that  $\Upphi_{i}(\scrS_j)\in{\rm L}_{i}(\scrA_{\upnu_iN})$.
\end{proof}
\end{lem}

\subsection{t-structure transfer}
In moving to the mutation functors, which reveals many hidden t-structures, we lose control over Fourier--Mukai techniques.  The following theorem is one of our main results, and is a crucial ingredient in the proof of Theorem~\ref{regular covering} later.

\begin{thm}\label{action free}
Let $L\in\Mut(N)$, $\upalpha\in\Hom_{\dsG^{\aff}}(C_L,C_L)$, and consider
\[
\Upphi_\upalpha\colon \Db(\fmod\Lambda_L)\to\Db(\fmod\Lambda_L).
\]
Then the following conditions are equivalent.
\begin{enumerate}
\item $\Upphi_\upalpha$ maps the simples $\scrS_0,\hdots,\scrS_n$ to simples.
\item $\Upphi_\upalpha\colon\scrD_L\to\scrD_L$ restricts to an equivalence $\scrB_L\to\scrB_L$.
\item $\Upphi_\upalpha\colon \Db(\fmod\Lambda_L)\to\Db(\fmod\Lambda_L)$ restricts to an equivalence $\fmod\Lambda_L\to \fmod\Lambda_L$.
\item There is a functorial isomorphism $\Upphi_\upalpha\cong\Id$.
\end{enumerate}
If further $L\in\Mut_0(N)$ and $\upalpha\in \Hom_{\dsG}(C_L,C_L)$, the above are equivalent to
\begin{enumerate}[resume]
\item $\Upphi_\upalpha$ maps the simples $\scrS_1,\hdots,\scrS_n$ to simples.
\item $\Upphi_\upalpha$ restricts to an equivalence $\scrA_L\to\scrA_L$.
\end{enumerate}
\end{thm}
When the last additional conditions are satisfied, it is already known that (5)$\Rightarrow$(2) by \cite[5.5]{HW}.  Furthermore, it is clear that (1)$\Leftrightarrow$(2), (5)$\Leftrightarrow$(6) and (4)$\Rightarrow$(1)(5).  Hence to prove Theorem~\ref{action free} it suffices to show that (1)$\Rightarrow$(3) and (3)$\Rightarrow$(4).

\begin{lem}\label{LemmaA}
Let $\Gamma$ be a noetherian ring, and $x\in\Db(\fmod\Gamma)$. Then the following hold.
\begin{enumerate}
\item\label{LemmaA 1} $x\in\Db(\fmod\Gamma)^{\leq 0}\iff\Ext^i_\Gamma(x,S)=0$ for all $i<0$ and all simple $\Gamma$-modules $S$.
\item\label{LemmaA 2}  If a triangulated equivalence $F\colon\Db(\fmod\Gamma)\to \Db(\fmod\Gamma)$ satisfies $F(\Gamma)\cong\Gamma$, then $F$ restricts to an equivalence $F\colon\fmod\Gamma\to\fmod\Gamma$. 
\end{enumerate}
\end{lem}
\begin{proof}
(1) The direction ($\Rightarrow$) is clear, by replacing $x$ by its projective resolution $P$, and observing that
$\Ext_\Gamma^i(x,S)=\Hom_{\Km(\fmod\Gamma)}(P,S[i])=0$ for all $i<0$, since there are no chain maps between the complexes $P$ and $S[i]$.

For ($\Leftarrow$), since $x$ is bounded, let $t$ be  maximum such that ${\rm H}^{t}(x)\neq 0$.  Since $\Gamma$ is noetherian, every finitely generated module has a map to a simple, so there exists some simple $S$ such that $\Hom_\Gamma({\rm H}^{t}(x),S)\neq 0$.  But via the spectral sequence (see e.g.\ \cite[(2.8)]{HuybrechtsFM})
\[
E_2^{p,q}=\Ext^p_\Gamma({\rm H}^{-q}(x),S)\Rightarrow\Ext_\Gamma^{p+q}(x,S)
\]
the nonzero $E_2^{0,-t}$ term survives to give a non-zero element of $\Ext^{-t}_\Gamma(x,S)$.  Hence $t\leq 0$.\\
(2) Via the isomorphism ${\rm H}^{n}(x)\cong\Hom_{\Db(\fmod\Gamma)}(\Gamma,x[n])$, it follows that $F$ and its inverse take modules to modules, and so they restrict to a Morita equivalence.
\end{proof}

The following establishes (1)$\Rightarrow$(3). In fact, we prove a slightly more general version, as we will need this later.

\begin{cor}\label{1 implies 3}
Suppose that $G\colon \Db(\fmod\Lambda_L)\to\Db(\fmod\Lambda_L)$ is any equivalence that maps the simples $\scrS_0,\hdots,\scrS_n$ to simples.  Then $G$ restricts to an equivalence $\fmod\Lambda_L\to \fmod\Lambda_L$.
\end{cor}
\begin{proof}
To ease notation, set $\scrE=\Db(\fmod\Lambda_L)$. By Lemma~\ref{LemmaA}\eqref{LemmaA 1} it follows that both $G$ and its inverse restrict to an equivalence $\scrE^{\leq 0}\to\scrE^{\leq 0}$. 
Since $\scrE^{\geq 1}$ can be characterised as the perpendicular to $\scrE^{\leq 0}$, it follows that both $G$ and its inverse restrict to an equivalence $\scrE^{\geq 1}\to\scrE^{\geq 1}$.  Since  $\fmod\Lambda_L=\scrE^{\leq 0}\cap\scrE^{\geq 1}[1]$, the result follows.
\end{proof}

The implication (3)$\Rightarrow$(4) is by far the most subtle.  It requires the following two technical results, both of which rely heavily on the fact that $R$ is isolated cDV.

\begin{prop}\label{t-stucture transfer}
In the setting of Theorem~\ref{action free}, if $\Upphi_\upalpha$ restricts to $\fmod\Lambda_L\xrightarrow{\sim} \fmod\Lambda_L$, then $\Upphi_i^{-1}\circ\Upphi_\upalpha\circ\Upphi_i$ restricts to $\fmod\Lambda_{\upnu_iL}\xrightarrow{\sim} \fmod\Lambda_{\upnu_iL}$. 
\end{prop}
\begin{proof}
Consider the functor $G=\Upphi_i^{-1}\circ\Upphi_\upalpha\circ\Upphi_i$.  For each $j\neq i$, as in Lemma~\ref{k-correspond proj} we have $\Upphi_i(\scrP_j)\cong\scrP_j$.  Since $\Upphi_\upalpha$ restricts to an equivalence on $\fmod\Lambda_L$, and is necessarily the identity on $\scrK_{L}$ by Proposition~\ref{theta trivial}, furthermore $\Upphi_\upalpha(\scrP_j)\cong\scrP_j$.  In conclusion, whenever $j\neq i$ we have $G(\scrP_j)\cong\scrP_j$.  In a similar vein, by \cite[4.15(2)]{HomMMP}, $\Upphi_i(\scrS_i)\cong\scrS_i[-1]$.  The functor $\Upphi_\upalpha$ must send simples to simples, and since it is the identity on $\scrK_{L}$, by the pairing between projectives and simples it follows that $\Upphi_\upalpha(\scrS_i)\cong\scrS_i$.    Thus $G(\scrS_i)\cong\scrS_i$.

 By the assumptions, since $\Lambda_L$ is basic, necessarily $\Upphi_\upalpha(\Lambda_L)\cong\Lambda_L$.  Now consider $\scrT=\Hom_R(\upnu_iL,L)$ defining the functor $\Upphi_i\colon \Db(\fmod\Lambda_{\upnu_iL})\to \Db(\fmod\Lambda_{L})$.  By construction $G(\scrT)\cong\Upphi_i^{-1}\circ\Upphi_\upalpha(\Lambda_L)\cong\Upphi_i^{-1}(\Lambda_L)\cong \scrT$.  Since $\scrT=\scrT_i\oplus\bigoplus_{j\neq i}\scrP_j$, by the above paragraph necessarily $G(\scrT_i)\cong \scrT_i$.

Now by e.g.\ \cite[A.2(2)]{HomMMP} the exchange sequences give rise to an exact sequence of $\Lambda_{\upnu_iL}$-modules
\[
0\to\scrP_i\xrightarrow{b}\bigoplus_{j\neq i}\scrP_j^{\oplus a_{ij}}
\to\bigoplus_{j\neq i}\scrP_j^{\oplus a_{ij}}
\to \scrP_i
\to \frac{\Lambda_{\upnu_iL}}{(1-e_i)}
\to 0
\]
with $\Cok b\cong \scrT_i$, where $e_i$ is the idempotent corresponding to the $i$th summand of $\Lambda_{\upnu_iL}$, and $(1-e_i)$ is the two-sided ideal generated by $1-e_i$.  Since $R$ is isolated, necessarily $\Lambda_{\upnu_iL}/(1-e_i)$ has finite length, and is filtered only by the simple $\scrS_i$.  Splicing gives triangles
\begin{align*}
\scrP_i\to\bigoplus_{j\neq i}\scrP_j^{\oplus a_{ij}}\to \scrT_i\to\\
\scrT_i\to\bigoplus_{j\neq i}\scrP_j^{\oplus a_{ij}}\to K_i\to\\
K_i\to \scrP_i
\to \frac{\Lambda_{\upnu_iL}}{(1-e_i)}\to
\end{align*}
Applying $G$ to each, the first triangle shows that $\mathrm{H}^t(G(\scrP_i))=0$ unless $t=0,1$.  On the other hand, the second triangle shows that $\mathrm{H}^t(G(K_i))=0$ unless $t=-1,0$.  But $G$ must take $\Lambda_{\upnu_iL}/(1-e_i)$ to degree zero, since $G(\scrS_i)\cong \scrS_i$ and $\Lambda_{\upnu_iL}/(1-e_i)$ is filtered by $\scrS_i$.  Hence the last triangle implies that $\mathrm{H}^t(G(\scrP_i))=0$ unless $t=-1,0$.  

Combining, we see that $\mathrm{H}^t(G(\scrP_i))=0$ unless $t=0$,  thus $G(\scrP_i)$ is a module.  Applying $G$ to the first triangle give a triangle
\[
G(\scrP_i)\to\bigoplus_{j\neq i}\scrP_j^{\oplus a_{ij}}\to \scrT_i\to
\]  
in which all terms are modules, so this is necessarily induced by a short exact sequence.  Hence by the depth lemma, $G(\scrP_i)$ has depth $3$.  On the other hand, since $\scrP_i$ is perfect as a complex, so is $G(\scrP_i)$, thus $G(\scrP_i)$ has finite projective dimension as a $\Lambda_{\upnu_iL}$-module.  By Auslander--Buchsbaum \cite[2.16]{IW1}, it follows that $G(\scrP_i)$ is projective.  Since $G$ is the identity on $\scrK_{\upnu_iL}$, necessarily $G(\scrP_i)\cong\scrP_i$ and  hence $G(\Lambda_{\upnu_iL})\cong\Lambda_{\upnu_iL}$.  By Lemma~\ref{LemmaA}\eqref{LemmaA 2}, $G=\Upphi_i^{-1}\circ\Upphi_\upalpha\circ\Upphi_i$ restricts to an equivalence $\fmod\Lambda_{\upnu_iL}\to\fmod\Lambda_{\upnu_iL}$.
\end{proof}

\begin{prop}\label{G is the identity}
Suppose that $G\colon\Db(\fmod\Lambda)\to\Db(\fmod\Lambda)$ is an $R$-linear equivalence such that $G(\scrS_0)\cong \scrS_0$, and suppose that there exists a permutation $\upiota\in S_n$ such that $G(\scrS_i)\cong \scrS_{\upiota(i)}$ and $\rk_R N_i\cong\rk_RN_{\upiota(i)}$ for all $i=1,\hdots,n$.  Then functorially $G\cong\Id$.
\end{prop}
\begin{proof}
By Corollary~\ref{1 implies 3}, $G$ maps projectives to projectives.  Since $G(\scrS_0)\cong\scrS_0$, by the pairing between simples and projectives, necessarily $G(\scrP_0)\cong\scrP_0$. Consider the $R$-linear composition $F$ given by
\[
\begin{tikzpicture}
\node (A1) at (0,0) {$\Db(\coh X)$};
\node (A2) at (3.5,0) {$\Db(\coh X)$};
\node (B1) at (0,-1.5) {$\Db(\fmod \Lambda)$};
\node (B2) at (3.5,-1.5) {$\Db(\fmod \Lambda)$};
\draw[densely dotted,->] (A1) -- node[above] {\scriptsize $F$} (A2);
\draw[->] (B1) -- node[above] {\scriptsize $G$} (B2);
\draw[->] (A1) -- node[left] {\scriptsize $\RHom_X(\scrV_X,-)$} (B1);
\draw[<-] (A2) -- node[right] {\scriptsize $-\otimes^{\bf L}_\Lambda\scrV_{X}$} (B2);
\end{tikzpicture}
\]
By \cite[5.2.4]{Karmazyn} the functor $\RHom_X(\scrV_X,-)$ maps skyscrapers of closed points to modules of dimension vector $\rk_N=(\rk_R N_i)_{i=0}^n$ which satisfy the $\star$-generated condition, which by definition consists of those $\Lambda$-modules $A$ of dimension vector $\rk_N$ such that $\Hom_\Lambda(A,\scrS_i)=0$ for all $i=1,\hdots, n$ \cite[6.11]{SY}. For any such $A$, by the assumptions on where $G$ takes each simple, $GA$ has dimension vector $(\rk_R N_{\upiota(i)})_{i=0}^n$, which by the last assumption is precisely $\rk_N$.  Furthermore, for any such $A$, since $G$ fixes $\scrS_0$ and permutes the other simples, $GA$ is also $\star$-generated.  Hence again appealing to \cite[5.2.4]{Karmazyn}  the functor $-\otimes^{\bf L}_\Lambda\scrV_{X}$ takes the module $GA$ to a skyscraper.  Combining, we see that skyscrapers of closed points get sent to skyscrapers of closed points, under the above $R$-linear composition $F$.

It follows from general Fourier--Mukai theory \cite[\S3.3]{BM} that $F\cong \varphi_*\circ (-\otimes\scrL)$ where $\varphi\colon X\to X$ is an automorphism  and $\scrL$ is some line bundle.  Since $\RHom_X(\scrV_X,-)$ sends $\scrO_X$ to $\scrP_0$, and $G$ sends $\scrP_0$ to $\scrP_0$, it follows that $F$ sends $\scrO_X$ to $\scrO_X$, which in turn implies that $\scrL$ is trivial.  Lastly, since $F\cong\varphi_*$ is $R$-linear,  by Proposition~\ref{iso is R linear}  $\varphi$ commutes with the map to the base. In particular, the restriction of $\varphi$ to the dense open subset $U=X\backslash C$ is the identity.  Hence $\varphi=\Id_X$, and as a result, $F\cong\Id$.  From this, it follows that $G\cong \Id$.
\end{proof}

Finally, we prove (3)$\Rightarrow$(4), completing the proof of Theorem~\ref{action free}.  The key is that Proposition~\ref{t-stucture transfer} allows us to pull everything back to $\Db(\fmod\Lambda)$, where we can use the geometric Fourier--Mukai techniques of Proposition~\ref{G is the identity}.

\begin{cor}\label{key functorial identity}
In the setting of Theorem~\ref{action free}, if $\Upphi_\upalpha$ restricts to $\fmod\Lambda_L\xrightarrow{\sim} \fmod\Lambda_L$, then there is a functorial isomorphism $\Upphi_\upalpha\cong\Id$.
\end{cor}
\begin{proof}
Choose a positive path $\upgamma\colon C_+\to C_L$, and consider the composition
\[
G=\Upphi_{\upgamma}^{-1}\circ\Upphi_{\upalpha}\circ\Upphi_{\upgamma}\colon\Db(\fmod\Lambda)\to\Db(\fmod\Lambda)
\]
Since $\upgamma$ is a composition $s_{i_t}\circ\hdots\circ s_{i_1}$, we may rewrite the above  as
\[
G=\Upphi_{i_1}^{-1}\circ\hdots\circ\Upphi_{i_{t-1}}^{-1}\circ(\Upphi_{i_t}^{-1}\circ\Upphi_{\upalpha}\circ\Upphi_{i_t})\circ\Upphi_{i_{t-1}}\circ\hdots\circ\Upphi_{i_1}
\]
By induction, using Proposition~\ref{t-stucture transfer} repeatedly, we see that $G$ restricts to an equivalence on $\fmod\Lambda$.  Since Morita equivalences preserve projectives, and $G$ is the identity on K-theory $\scrK=\scrK_N$ by Proposition~\ref{theta trivial}, $G$ maps each projective to itself. Since Morita equivalences also preserve simples, by the pairing between projectives and simples, $G$ maps each simple to itself.  By Proposition~\ref{G is the identity} $G\cong\Id$ and hence 
$\Upphi_{\upalpha}\cong\Id$.
\end{proof}

\section{Stability Conditions on \texorpdfstring{$\scrC$}{C} and \texorpdfstring{$\scrD$}{D}}\label{stab on C and D section}
Consider $\cStab{}\scrC$, the connected component of $\Stab\scrC$ containing $\Stab\scrA$, and similarly  $\cStab{n}\scrD$, the connected component of $\nStab{n}\scrD$ containing $\nStab{n}\scrB$.  In this section we describe both $\cStab{}\scrC$ and $\cStab{n}\scrD$ as regular covers of the hyperplane arrangements in Section~\ref{hyperplane section}.

\subsection{Chamber Decomposition}\label{sec: chamber decomp}
For the fixed $R$-module $N$ from \eqref{ass modi}, consider the set of morphisms in $\dsG$  which terminate at $C_+$, namely
\begin{align*}
\Targ_0(C_+)&\colonequals \bigcup_{L\in\Mut_0(N)}\Hom_{\mathds{G}}(C_L,C_+).
%\Targ(C_+)&\colonequals \bigcup_{N\in\Mut(N)}\Hom_{\mathds{G}_{\aff}}(C_N,C_+)
\end{align*}
The set $\Targ(C_+)$ is defined similarly, taking the union instead over $L\in\Mut(N)$  and replacing $\dsG$ by $\dsG^{\aff}$.
\begin{nota}\label{chamber notation}
For $L\in\Mut_0(N)$, respectively $L\in\Mut(N)$, consider the open subsets
\begin{align*}
\mathrm{U}_L&\colonequals \{ (Z,\scrA_L)\in \Stab\scrA_L \mid   \Image(Z[\scrS_i])>0 \mbox{ \rm for all } i=1,\hdots,n\} \\
\bU_L&\colonequals \{ (Z,\scrB_L)\in \Stab\scrB_L \mid   \Image(Z[\scrS_i])>0 \mbox{ \rm for all } i=0,1,\hdots,n\}\\
\dsN_L&\colonequals \bU_L\cap \nStab{n}\scrB_L =\bU_L\cap \nStab{n}\scrD_L
\end{align*}
of $\cStab{}\scrC_L$, $\cStab{}\scrD_L$ and $\cStab{n}\scrD_L$ respectively.   For $\upalpha\in\Targ_0(C_+)$, $\upbeta\in\Targ(C_+)$, set 
\[
\begin{array}{r@{\hspace{2pt}}c@{\hspace{2pt}}l@{\hspace{25pt}}r@{\hspace{2pt}}c@{\hspace{2pt}}l@{\hspace{2pt}}l@{\hspace{2pt}}}
\Stab\scrA_{\upalpha}
&
\colonequals
&
(\Upphi_{\upalpha})_*(\Stab\scrA_{s(\upalpha)})
%&\subset \Stab\scrC
&\Stab\scrB_{\upbeta}
&\colonequals
&(\Upphi_{\upbeta})_*(\Stab\scrB_{s(\upbeta)})\\
%&\subset \Stab\scrD\\
\mathrm{U}_{\upalpha}
&\colonequals
&
(\Upphi_{\upalpha})_*(\mathrm{U}_{s(\upalpha)})
%&\subset\Stab\scrA_{\upalpha}
%&\subset \Stab\scrC
&\bU_{\upbeta}
&\colonequals
&(\Upphi_{\upbeta})_*(\bU_{s(\upbeta)}),
%&\subset\Stab\scrB_{\upbeta}
%&\subset \Stab\scrD
\end{array}
\]
where $s(\upalpha)$ and $s(\upbeta)$ denote the modules corresponding to the chambers which are the sources of $\upalpha$ and $\upbeta$ respectively. Similarly,
\begin{align*}
\nStab{n}\scrB_{\upbeta}&\colonequals (\Upphi_{\upbeta})_*(\nStab{n}\scrB_{s(\upbeta)})\\
\dsN_\upbeta&\colonequals (\Upphi_{\upbeta})_*(\dsN_{s(\upbeta)}).
\end{align*}
\end{nota}
As usual, write $\mathrm{U}=\mathrm{U}_{N}$, $\bU=\bU_{N}$ and $\dsN=\dsN_N$.

\begin{lem}\label{no overlap lemma}
Given $\upalpha, \upbeta\in\Targ_0(C_+)$, respectively $\upalpha, \upbeta\in\Targ(C_+)$, write $M_{s(\upalpha)}$ and $M_{s(\upbeta)}$ for the modules corresponding to the chambers which are the sources of $\upalpha$ and $\upbeta$ respectively. Then
\begin{enumerate}
\item $\mathrm{U}_\upalpha\cap \mathrm{U}_\upbeta\neq\emptyset$ $\iff$
 $M_{s(\upalpha)}\cong M_{s(\upbeta)}$ and   $\Upphi_{\upalpha}\cong\Upphi_{\upbeta}$.
\item $\bU_\upalpha\cap \bU_\upbeta\neq\emptyset$ $\iff$
 $M_{s(\upalpha)}\cong M_{s(\upbeta)}$ and   $\Upphi_{\upalpha}\cong\Upphi_{\upbeta}$ $\iff$ $\dsN_\upalpha\cap \dsN_\upbeta\neq\emptyset$.
\end{enumerate}
\end{lem}
\begin{proof}
(1) It suffices to show that, for $\upgamma\in\Targ_0(C_+)$, $\mathrm{U}\cap \mathrm{U}_{\upgamma}\neq\emptyset$ if and only if $M_{s(\upgamma)}\cong N$ and $\Upphi_{\upgamma}\cong{\rm id}_{\scrC}$. The implication $(\Leftarrow)$ is obvious, since the isomorphisms $M_{s(\upgamma)}\cong N$ and $\Upphi_{\upgamma}\cong{\rm id}_{\scrC}$ implies that $\mathrm{U}=\mathrm{U}_{\upgamma}$.
 
Conversely, suppose that $\mathrm{U}\cap \mathrm{U}_{\upgamma}\neq\emptyset$, and write  $\cY\colon\Stab\scrC\to \Uptheta_{\bR}$ for the composition
\[
\Stab\scrC\xrightarrow{\cZ}\Uptheta_{\bC}\xrightarrow{\Image} \Uptheta_{\bR}
\]
where the last map is the projection defined by taking the imaginary parts.  By definition $\cY(\mathrm{U})=C_+$ and $\cY(\mathrm{U}_{\upgamma})=\upvarphi_{M_{s(\upgamma)}}(C_+)$.
Since $\mathrm{U}\cap \mathrm{U}_{\upgamma}\neq\emptyset$, necessarily $\cY(\mathrm{U})\cap\cY(\mathrm{U}_{\upgamma})\neq\emptyset$, thus $M_{s(\upgamma)}\cong N$ by Theorem~\ref{HomMMP finite summary}, and $\Upphi_{\upgamma}$ is an autoequivalence of $\scrC$.
It is clear that $\mathrm{U}\cap \mathrm{U}_{\upgamma}\neq\emptyset$ implies that  $\Upphi_{\upgamma}\colon \scrC\to\scrC$ maps $\scrA$ to $\scrA$. By Theorem~\ref{action free}, this implies that  
 $\Upphi_{\upgamma}\cong\Id$.\\
 (2) The proof of the first $\iff$ is identical, appealing to Theorem~\ref{affine summary} instead of Theorem~\ref{HomMMP finite summary} to deduce that $\Upphi_{\upgamma}$ maps $\scrB$ to $\scrB$. Theorem~\ref{action free} again implies that   
 $\Upphi_{\upgamma}\cong \Id$. The second $\iff$ follows immediately from the first.
\end{proof}

\begin{lem}\label{adjacent}
For $\upalpha, \upbeta\in\Targ_0(C_{+})$ with $l(\upalpha)>l(\upbeta)$, the chambers $\Stab\scrA_{\upalpha}$ and $\Stab\scrA_{\upbeta}$ share a codimension one boundary if and only if there exists a length one path $\upgamma\in{\rm Mor}(\mathds{G}^+)$ such that $\upalpha=\upbeta\circ\upgamma$ or $\upalpha=\upbeta\circ\upgamma^{-1}$ in ${\rm Mor}(\mathds{G})$.  A similar statement holds replacing $\scrA$ by $\scrB$, $\Targ_0(C_{+})$ by $\Targ(C_{+})$ and $\mathds{G}$ by $\mathds{G}^{\aff}$ respectively.
\end{lem}
\begin{proof}
It is enough to prove that, for $\upgamma\in\Targ_0(C_+)$, $\mathrm{U}$ and $\mathrm{U}_{\upgamma}$ share a codimension one boundary if and only if there is a length one path $\updelta$ such that $\upgamma=\updelta$ or $\upgamma=\updelta^{-1}$. This follows from \cite[5.5]{B09} and Lemma \ref{simple tilt}.
\end{proof}

 The following is the analogue of \cite[4.11]{T08}.

\begin{thm}\label{chambers in cStab C} 
With notation as above, the following statements hold.
\begin{enumerate}
\item\label{chambers in cStab C 1}  There is a disjoint union of open chambers 
\[
\scrM\colonequals \kern -12pt\bigcup_{\upalpha\in\Targ_0(C_+)}\kern -12pt\mathrm{U}_{\upalpha}\quad\subset \cStab{}\scrC.
\]
Furthermore,  $\overline{\scrM}=\bigcup\overline{\mathrm{U}}_{\upalpha}=\cStab{}\scrC$, where $\overline{\mathrm{U}}_\upalpha$ is the closure of $\mathrm{U}_\upalpha$ in $\cStab{}\scrC$. 
\item\label{chambers in cStab C 2}  There is a disjoint union of open chambers 
\[
\scrN\colonequals \kern -10pt \bigcup_{\upbeta\in\Targ(C_+)} \kern -10pt{\bU}_{\upbeta}\quad\subset \cStab{}\scrD.
\]
Furthermore, $
\overline{\scrN}=\bigcup\overline{\bU}_{\upbeta}=\cStab{}\scrD$.
\end{enumerate}
In particular, as $
 \cStab{}\scrD\cap \nStab{n}\scrD=\cStab{n}\scrD,
 $ there is a disjoint union of open chambers 
\[
\scrN_n\colonequals \kern -10pt \bigcup_{\upbeta\in\Targ(C_+)} \kern -10pt{\dsN}_{\upbeta}\quad\subset \cStab{n}\scrD
\]
such that $\overline{\scrN_n}= \bigcup\overline{{\dsN}}_{\upbeta}=\cStab{n}\scrD$.
\end{thm}
\begin{proof}
(1) By Lemma \ref{no overlap lemma}, $\scrM$ is a disjoint union, and  by Lemma \ref{adjacent} $\overline{\scrM}$ is connected. Since $\overline{\scrM}$ contains $\overline{\mathrm{U}}$ and thus $\Stab\scrA$,
 there is an inclusion $\overline{\scrM}\subseteq \cStab{}\scrC$.

Let $\upsigma\in\cStab{}\scrC$ be a point, and choose a point $\upsigma_0\in \mathrm{U}$ and a path
\[
p\colon [0,1]\to\cStab{}\scrC
\]
such that $p(0)=\upsigma_0$ and $p(1)=\upsigma$.   Since $\cZ\colon\cStab{}\scrC\to\Uptheta_{\bC}$ is a local homeomorphism, by deforming $p$ if necessary, we may assume that the path $\cZ\circ p\colon [0,1]\to\Uptheta_{\bC}$ passes through only finitely many codimension one boundaries of chambers $\upvarphi_L(\bH_+)$. Thus there exists a sequence $0<t_1<t_2<\hdots<t_{\ell-1}< t_{\ell}\colonequals1$ of real numbers such that:
\begin{enumerate}
\item[(a)] for all $i\neq \ell$, every $\cZ(p(t_i))$  is in a codimension one boundary of some chamber,
\item[(b)] for all $i$, each open interval $\cZ(p(t_i,t_{i+1}))$ is contained in the interior of some chamber.
\end{enumerate} 
Since $p((0,t_1))\subset \mathrm{U}$,  by Lemma \ref{adjacent} there is a length one path $\upgamma\in\Targ_0(C_+)$ such that $p(t_1,t_2)$ is in $\mathrm{U}_{\upgamma}$. By iterating this argument,  we see that $p((t_{l-1},1))$ is in  some open chamber $\mathrm{U}_{\upalpha}$, and  hence its end point, $\upsigma$, belongs to $ \overline{\mathrm{U}}_{\upalpha}$. \\
(2) This follows using an identical argument to (1). 

For the last statements, we first prove that $\cStab{}\scrD\cap \nStab{n}\scrD=\cStab{n}\scrD$. Let $\upsigma\in \cStab{n}\scrD$. Since $\cStab{n}\scrD$ is a connected and locally Euclidian space, it is path connected. Hence there is a path from $\upsigma$ to a point $\upsigma_0\in \nStab{n}\scrB\subset \cStab{}\scrD$. Thus $\upsigma$ also lies in $\cStab{}\scrD$, proving  $\cStab{n}\scrD\subseteq\cStab{}\scrD\cap \nStab{n}\scrD$. 

For the opposite inclusion,  it is enough to show that $\cStab{}\scrD\cap \,\nStab{n}\scrD$ is connected, since $\cStab{}\scrD\cap \nStab{n}\scrD$ contains $\nStab{n}\scrB$. But by \eqref{chambers in cStab C 2},
\[
\cStab{}\scrD\cap \nStab{n}\scrD = (\bigcup\overline{{\bU}}_{\upbeta})\cap  \nStab{n}\scrD=\bigcup\overline{{\dsN}}_{\upbeta}.
\] 
  Since $\overline{{\dsN}}_{\upbeta}$ is the closure of ${\dsN}_{\upbeta}$ in $\nStab{n}\scrD$, we have $\overline{{\dsN}}_{\upbeta}=\overline{{\bU}}_{\upbeta}\cap\,\nStab{n}\scrD$, and by Lemma~\ref{path connected} and Proposition~\ref{stability track comm diagram}, all $\overline{{\dsN}}_{\upbeta}$ are path connected.
Thus again by Proposition~\ref{stability track comm diagram}, it suffices to show that $\overline{{\dsN}}\cap\,\overline{\dsN}_{\upgamma}=\overline{{\bU}}\cap\,\overline{{\bU}}_{\upgamma}\cap\,\nStab{n}\scrD\neq \emptyset$ for any length one path $\upgamma$.   If $\upgamma=s_i$, then consider the point  $\upsigma=(Z,\scrB)\in \Stab\scrB$ defined by $Z[\scrS_i]=-1/\uplambda_i$, and $Z[\scrS_j]=(1+\ii)/n\uplambda_j$ for all $ j\neq i$, where $\uplambda_k\colonequals \rk_RN_k$. Then $\upsigma$ lies in $\nStab{n}\scrB\subset \overline{\bU}\,\cap \,\nStab{n}\scrD$ and in the codimension one boundary of $\Stab{\rm L}_{i}(\scrB)$ by \cite[Lemma 5.5]{B09}. But since  $\Stab{\rm L}_{i}(\scrB)=(\Upphi_{i})_*(\Stab\scrB_{\upnu_iN})=\Stab\scrB_{\upgamma}$  by Lemma \ref{simple tilt},  $\upsigma\in\overline{\Stab\scrB_{\upgamma}}=\overline{\bU}_{\upgamma}$. This implies that  $\overline{{\bU}}\cap\,\overline{{\bU}}_{\upgamma}\cap\,\nStab{n}\scrD\neq \emptyset$. Similarly, we see that $\overline{{\bU}}\cap\,\overline{{\bU}}_{\upgamma}\cap\,\nStab{n}\scrD\neq \emptyset$ when $\upgamma=s_i^{-1}$.
 Hence $\cStab{}\scrD\cap \,\nStab{n}\scrD$ is connected, and thus $\cStab{}\scrD\cap \nStab{n}\scrD=\cStab{n}\scrD$ follows.
The remaining statements are then immediate from \eqref{chambers in cStab C 2}.
\end{proof}

\subsection{Regular Covering Structure}

\begin{lem}\label{coord ax} 
If $L\in  \Mut_0(N)$, respectively $L\in\Mut(N)$, then the following statements hold.
\begin{enumerate}
\item\label{coord ax 1}   If a point $\upsigma=(Z,\cA)\in \cStab{}\scrC_L$ is in $\overline{\mathrm{U}}_{\upalpha}$ for some $\upalpha\in \End_{\dsG}(C_+)$, then $\cZ_L(\upsigma)$ is not on any complexified coordinate axis in $(\Uptheta_{L})_{\bC}$. 
\item\label{coord ax 2}  If a point $\upsigma=(Z,\cB)\in \cStab{}\scrD_{L}$ is in $\overline{\bU}_{\upbeta}$ for some $\upbeta\in\End_{\dsG^{\aff}}(C_+)$, then $\cZ_{L}(\upsigma)$ is not on any complexified coordinate axis in $(\scrK_{L})_{\bC}$.  
\end{enumerate}
\begin{proof}
(1) First, for any $\uprho\in \mathrm{U}_1=\mathrm{U}$, every simple module $\scrS_i\in \scrA_L$ is $\uprho$-semistable. Hence, by \cite[7.6]{BS},   $\scrS_i$ is $\uprho$-semistable for all $\uprho\in \overline{\mathrm{U}}_1$, and in particular $Z[\scrS_i]\neq 0$ for all $(Z,\cA)\in \overline{\mathrm{U}}_1$. 

Now, for $\upsigma\in\overline{\mathrm{U}}_{\upalpha}$, by Proposition~\ref{stability track comm diagram} and Proposition~\ref{theta trivial}\eqref{theta trivial 2}, $\cZ_L(\upsigma)=\cZ_L((\Upphi_{\upalpha})_*^{-1}(\upsigma))$, and by definition $(\Upphi_{\upalpha})_*^{-1}(\upsigma)\in \overline{\mathrm{U}}_1$.  Hence, by the first paragraph, $\cZ_L(\upsigma)$ is not on any complexified coordinate axis. \\
(2) is  identical to (1).
\end{proof}
\end{lem}

\begin{lem}\label{surj to image}
With notation as above, the following statements hold.
\begin{enumerate}
\item\label{surj to image 1} The map $\cZ\colon \cStab{}\scrC\to\Uptheta_{\bC}$ restricts to a surjective map 
\[
\cZ\colon \cStab{}\scrC\to \Uptheta_{\bC}\backslash \scrH_{\bC}.
\]
\item\label{surj to image 2} The map $\cZ\colon\cStab{n}\scrD\to \Level_{\bC}$ restricts to a surjective map
\[
\cZ\colon\cStab{n}\scrD\to \Level_{\bC}\backslash\scrH^{\aff}_{\bC}.
\]
\end{enumerate}
\end{lem}
\begin{proof}
(1) First, we show that ${\rm Im}(\cZ)\subseteq\Uptheta_{\bC}\backslash \scrH_{\bC}$. By Theorem \ref{chambers in cStab C} and Proposition \ref{stability track comm diagram}, it is enough to show that  $\cZ(\upsigma)=x+\ii y\in\Uptheta_{\bC}\backslash \scrH_{\bC}$ for each  point $\upsigma=(Z,\cA)\in \overline{\mathrm{U}}$ in the closure  of $\mathrm{U}$. Assume that $\cZ(\upsigma)\in H_{\bC}$ for some $H\in \scrH$. 

Let $\scrI\colonequals \{1\leq i\leq n\mid y_i=0\}$ and ${\rm B}\colonequals\{\upvartheta\in \Uptheta_{\bR}\mid \upvartheta_j=0 \mbox{ for all } j\in \scrI\}$. If $\scrI=\emptyset$, then the point $\upsigma$  necessarily lies in $\mathrm{U}$, and so $\cZ(\upsigma)\in \bH_+\subset \Uptheta\backslash \scrH_{\bC}$. Hence we may assume $\scrI\neq \emptyset$. Since by Lemma~\ref{hyper form}\eqref{hyper form 1} the hyperplane $H$ has the form $\uplambda_1\upvartheta_{i_1}+\hdots+\uplambda_{s}\upvartheta_{i_s}=0$ for some $\uplambda_{1},\hdots,\uplambda_s>0$, the fact that $y_j>0$ if $j\notin\scrI$ implies $i_1,\hdots,i_s\in \scrI$. Since  ${\rm B}\subseteq H$, by Lemma \ref{E lemma} there exist  $k\in \scrI$ and a minimal mutation sequence 
\[
\upalpha\colon L\to\hdots\to N\in \MutTo_{\scrI}(N)
\] 
such that $H=\upvarphi_{\upalpha}(H_k)$, where $H_k=\{\upvartheta_k=0\}\subset (\Uptheta_L)_{\bR}$ is the $k$th coordinate axis in $(\Uptheta_L)_{\bR}$.
If we set  $\upsigma'\colonequals(\Upphi_{\upalpha}^{-1})_*(\upsigma)\in \cStab{}\scrC_L$, then $\cZ_L(\upsigma')\in (H_k)_{\bC}$. Since $y\in \overline{C}_+$ and $\upalpha\in \MutTo_{\scrI}(N)$, we see that $\upvarphi_{\upalpha}^{-1}(y)=y$ by the rule \eqref{k-correspond}, and so ${\rm Im}(\cZ_L(\upsigma'))\in \overline{C}_+$. But this implies that $\upsigma'\in \overline{\mathrm{U}}_{\upbeta}\subset \cStab{}\scrC_L$ for some $\upbeta\in \End_{\dsG}(C_+)$, which is a contradiction by Lemma \ref{coord ax}\eqref{coord ax 1}. Hence ${\rm Im}(\cZ)\subseteq\Uptheta_{\bC}\backslash \scrH_{\bC}$.

Next, we show the map $\cZ\colon \cStab{}\scrC\to \Uptheta_{\bC}\backslash \scrH_{\bC}$ is surjective. Pick $z\in \Uptheta_{\bC}\backslash \scrH_{\bC}$, then by Proposition~\ref{complexified tracking} there exists some $L\in\Mut_0(N)$ such that $\upvarphi_L(h)=z$ for some $h\in\bH_+$.    The left hand side of the commutative diagram in Proposition~\ref{stability track comm diagram} shows that we can find $\upsigma\in\Stab\scrA_L$ such that $\cZ_L(\upsigma)=h$.  The commutativity then shows that $\upsigma'\colonequals (\Upphi_L)_*(\upsigma)\in\Stab\scrC$ maps, via $\cZ$, to $z$.  Since $\upsigma'\in\cStab{}\scrC$ by Theorem~\ref{chambers in cStab C}, it follows that $\cZ$ is surjective.\\
\noindent
(2) By Theorem~\ref{chambers in cStab C} and Proposition~\ref{stability track comm diagram}, for ${\rm Im}(\cZ)\subseteq\Level_{\bC}\backslash \scrH_{\bC}^{\aff}$, it suffices to prove that $\cZ(\upsigma)=x+\ii y\in \Level_{\bC}\backslash \scrH_{\bC}^{\aff}$ for any $\upsigma=(Z',\cB)\in \overline{\dsN}=\overline{\bU}\cap\,\nStab{n}\scrD$.  To see this, let $\scrI'\colonequals \{0\leq i\leq n\mid y_i=0\}$, and note that $\scrI'\subsetneq \{0,\hdots,n\}$ since $x+y\ii\in\Level_{\mathbb{C}}$.  If $\cZ(\upsigma)\in \scrH_{\bC}^{\aff}$, then by a similar argument to (1), now using Lemma~\ref{hyper form}\eqref{hyper form 2} and Lemma~\ref{E lemma} (with $\scrI'$), there exists a minimal mutation sequence $\upalpha\colon L\to N\in \MutTo_{\scrJ'}(N)$ such that $\upsigma'\colonequals (\Upphi_{\upalpha}^{-1})_*(\upsigma)$ lies in $\bU_{\upbeta}$ for some $\upbeta\in\End_{\dsG^{\aff}}(C_+)$ and $\cZ_L(\upsigma')\in (H_k)_{\bC}\cap (\Level_L)_{\bC}$ for  some  coordinate axis $H_k$ in $(\scrK_L)_{\bR}$. This contradicts Lemma \ref{coord ax}\eqref{coord ax 2}, and thus ${\rm Im}(\cZ)\subseteq\Level_{\bC}\backslash \scrH_{\bC}^{\aff}$.  The surjectivity of the map follows by a similar argument to \eqref{surj to image 1}, using Proposition \ref{complexified tracking 2},  Proposition \ref{stability track comm diagram} and Theorem \ref{chambers in cStab C}.
\end{proof}

\begin{nota}\label{not: pure braid images}
Consider the subgroups of $\Auteq\scrC$ and $\Auteq\scrD$ defined by
\begin{align*}
\Br\scrC&\colonequals \{\Upphi_{\upalpha}|_{\scrC}\mid \upalpha\in\End_{\mathds G}(C_+)\},\\
\Br\scrD&\colonequals \{\Upphi_{\upbeta}|_{\scrD}\mid \upbeta\in\End_{\dsG^{\aff}}(C_+)\}.
\end{align*}
\end{nota}

\begin{thm}\label{quot homeo}
With notation as above, the following statements hold.
\begin{enumerate}
\item\label{quot homeo 1} The surjective map $\cZ\colon \cStab{}\scrC\to \Uptheta_{\bC}\backslash \scrH_{\bC}$ induces a homeomorphism
\[
\cStab{}\scrC/\Br\scrC\xrightarrow{\sim} \Uptheta_{\bC}\backslash \scrH_{\bC}.
\]
\item\label{quot homeo 2}  The surjective map $\cZ\colon \cStab{n}\scrD\to \Level_{\bC}\backslash \scrH^{\aff}_{\bC}$ induces a homeomorphism
\[
\cStab{n}\scrD/\Br\scrD\xrightarrow{\sim} \Level_{\bC}\backslash \scrH^{\aff}_{\bC}.
\]
\end{enumerate}
\end{thm}
\begin{proof}
We only prove \eqref{quot homeo 2}, since the proof of \eqref{quot homeo 1} is identical.  Let $\upsigma\in\cStab{n}\scrD$ and $\Phi\in \Br\scrD$. Then $\cZ(\Phi_*(\upsigma))=\cZ(\upsigma)$ by Proposition \ref{theta trivial}\eqref{theta trivial 2} and  Proposition~\ref{stability track comm diagram}, and so $\cZ$ induces a map $\cStab{n}\scrD/\Br\scrD\to \Level_{\bC}\backslash \scrH_{\bC}^{\aff}$ which is surjective by Lemma~\ref{surj to image}.

We show that this induced map is  injective.   
Let $\upsigma,\upsigma'\in\cStab{n}\scrD$ be two points such that  $x\colonequals \cZ(\upsigma)=\cZ(\upsigma')$.   By Theorem \ref{chambers in cStab C}, $\upsigma\in \overline{\dsN}_{\upbeta}$ and $\upsigma'\in\overline{\dsN}_{\upbeta'}$ for some paths $\upbeta,\upbeta'\in \Targ(C_+)$.   But by Proposition \ref{complexified tracking 2}, there is a unique $L\in \Mut(N)$ such that $x\in\upphi_L(\bE_+)$, and so by Proposition \ref{stability track comm diagram} we see that $\upbeta,\upbeta' \in \Hom_{\dsG^{\aff}}(C_L,C_+)$. Set $\upgamma\colonequals \upbeta'\circ\upbeta^{-1}\in\End_{\dsG^{\aff}}(C_+)$, then by definition $(\Upphi_{\upgamma})_*(\dsN_{\upbeta})=\dsN_{\upbeta'}$. 

Since the surjective map $\cZ$ is a local homeomorphism, there exists an open neighbourhood $\scrU$ of $\upsigma$ such that the restrictions $\cZ|_{\scrU}$ and $\cZ|_{(\Upphi_{\upgamma})_*(\scrU)}$ are homeomorphisms. Choose a sequence $\{\upsigma_i\}_{i=1}^{\infty}\subset \dsN_{\upbeta}\cap \,\scrU$ that converges to $\upsigma$, and set
  $\upsigma'_i\colonequals (\Upphi_{\upgamma})_*(\upsigma_i)\in \dsN_{\upbeta'}\cap(\Upphi_{\upgamma})_*(\scrU)$. Then again by Proposition \ref{theta trivial}\eqref{theta trivial 2} and Proposition~\ref{stability track comm diagram},  we have $x_i\colonequals \cZ(\upsigma_i)=\cZ(\upsigma_i')$.  The sequence $\{x_i\}_{i=1}^{\infty}$ converges to $x$ since $\cZ|_{\scrU}$ is a homeomorphism. Moreover, since $\cZ|_{(\Upphi_{\upbeta})_*(\scrU)}$ is also a homeomorphism, the sequence $\{\upsigma_i'\}_{i=1}^{\infty}$ converges to $\upsigma'$. Hence 
\[
  \upsigma'=\lim_{i\to\infty}\upsigma'_i
  =\lim_{i\to\infty}(\Upphi_{\upgamma})_*(\upsigma_i)
  =(\Upphi_{\upgamma})_*(\,\lim_{i\to\infty}\upsigma_i\,)=(\Upphi_{\upgamma})_*(\upsigma).
  \]
This implies that $\upsigma=\upsigma'$ in $\cStab{n}\scrD/\Br\scrD$. 
\end{proof}

Given a group $G$ acting on a topological space $T$,  consider the following condition.
\begin{itemize}
\item[$(\ast)$] For each $x\in T$, there is an open neighbourhood $\scrU$ of $x$ such that $\scrU\cap g\scrU=\emptyset$ for all $1\neq g\in G$.
\end{itemize}

\begin{thm}\label{regular covering}
With notation as above, the following statements hold.
\begin{enumerate}
\item\label{regular covering 1} $\cZ\colon \cStab{}\scrC\to \Uptheta_{\bC}\backslash \cH_{\bC}$ is a regular covering map, with Galois group $\Br\scrC$.  
\item\label{regular covering 2} $\cZ\colon \cStab{n}\scrD\to \Level_{\bC}\backslash \scrH^{\aff}_{\bC}$ is a regular covering map, with Galois group $\Br\scrD$.
\end{enumerate}
Moreover, the covering map in \eqref{regular covering 1} is universal.
\end{thm}
\begin{proof}
(2)  We first show that the action of $\Br\scrD$ on $\cStab{n}\scrD$ satisfies the condition $(\ast)$. For this, take a point $\upsigma\in\cStab{n}\scrD$ and consider the open neighbourhood  $\upsigma\in \scrU$ defined by 
\[
\scrU\colonequals\{\upsigma'\in\Stab\scrD \mid d(\upsigma,\upsigma')<1/4\}\cap\,\cStab{n}\scrD,
\] 
where $d(-,-)$ is the metric introduced in \cite[\S6]{B07}. 

Assume that $\scrU\cap (\Upphi_{\upbeta})_*(\scrU)\neq\emptyset$ for some $\Upphi_{\upbeta}\in\Br\scrD$.  Then, every point $\upsigma'\in  \scrU$ must satisfy  $d\bigl(\upsigma',(\Upphi_{\upbeta})_*(\upsigma')\bigr)<1$.  Furthermore, the central charges of $\upsigma'$ and $(\Upphi_{\upbeta})_*(\upsigma')$ are equal by Proposition \ref{theta trivial}\eqref{theta trivial 2} and Proposition~\ref{stability track comm diagram}. 
Therefore, it follows that  $\upsigma'=(\Upphi_{\upbeta})_*(\upsigma')$ by \cite[6.4]{B07}, for every $\upsigma'\in\scrU$.

By Theorem~\ref{chambers in cStab C}, there is some $\dsN_{\upgamma}$ such that $\dsN_{\upgamma}\cap\, \scrU\neq\emptyset$, so choose $\uptau\in\dsN_{\upgamma}\cap \scrU$.  Then since $\uptau\in\dsN_{\upgamma}$, the heart of $\uptau$ is $(\Upphi_{\upgamma})_*(\scrB_{s(\upgamma)})$.  But on the other hand, since $\uptau\in\scrU$, by the previous paragraph $\uptau=(\Upphi_{\upbeta})_*(\uptau)$.   Thus the composition 
\[
\Upphi_{\upgamma^{-1}\upbeta\upgamma}=
\Upphi_\upgamma^{-1}\circ\Upphi_\upbeta\circ\Upphi_\upgamma\colon\Db(\fmod\Lambda_{s(\upbeta)})\to \Db(\fmod\Lambda_{s(\upbeta)})
\] 
restricts to an equivalence on $\scrB_{s(\upgamma)}$.  This implies   $\Upphi_\upgamma^{-1}\circ\Upphi_\upbeta\circ\Upphi_\upgamma\cong \Id$ by Theorem~\ref{action free} applied to $\upgamma^{-1}\upbeta\upgamma$. Thus $\Upphi_\upbeta\cong\Id$, and so  the action satisfies the condition $(\ast)$.

Since $\cStab{n}\scrD$ is path connected,  as is standard \cite[1.40(a)(b)]{Hat} it follows that
 $\Br\scrD$ is the group of deck transformations for the regular cover 
 \[
 \cStab{n}\scrD\to\cStab{n}\scrD/\Br\scrD.
 \] 
 Hence by Theorem \ref{quot homeo}, the map $\cZ\colon \cStab{n}\scrD\to \Level_{\bC}\backslash \scrH_{\bC}^{\aff}$ is a regular covering map, with Galois group $\Br\scrD$.  This completes the proof of (2).\\ 
(1) This follows using an identical argument to the above.

For the final statement, since $\cStab{}\scrC$ is a manifold, it is locally path connected.  Hence as is standard (see e.g.\ \cite[1.40(c)]{Hat}) the cover is universal if and only if the natural map
\[
\uppi_1(\Uptheta_{\bC}\backslash \scrH_{\bC})\twoheadrightarrow\Br\scrC
\] 
is injective.  But this is \cite{HW}, which works word-for-word in the more general terminal singularities setting here, as explained in \cite[\S10.3]{IW9}.  
\end{proof}

\begin{cor}\label{contractible}
$\cStab{}\scrC$ is contractible.
\end{cor}
\begin{proof}
The universal cover of the complexified complement simplicial hyperplane arrangement is contractible, due to Deligne's work on the $K(\uppi,1)$ conjecture \cite{Deligne}.
\end{proof}

\section{Autoequivalence and SKMS Corollaries}

The above description of stability conditions has consequences for autoequivalences, which in turn allows us to compute the SKMS.

\subsection{Autoequivalences of \texorpdfstring{$\scrC$}{C}}\label{sec: auto of C}
Consider the subgroup $\cAut{}\scrC$ of $\Auteq\scrC$, consisting of those $\Upphi|_\scrC$ where $\Upphi$ is a  Fourier--Mukai equivalence $\Db(\coh X)\to\Db(\coh X)$ that commutes with $\mathbf{R}f_*$ and preserves $\cStab{}\scrC$.  Since $\Upphi$ commutes with $\mathbf{R}f_*$, automatically $\Upphi|_\scrC\colon\scrC\to\scrC$.

\begin{thm}\label{7.1}
Suppose that $X\to\Spec R$ is a 3-fold flop, where $X$ has at worst terminal singularities.  Then $\cAut{}\scrC=\Br\scrC$.
\end{thm}
\begin{proof}
The inclusion $\Br\scrC\subset\cAut{}\scrC$ follows since the Bridgeland--Chen flop functors are Fourier--Mukai equivalences that commute with $\mathbf{R} f_*$, and $\Br\scrC$ acts on $\cStab{}\scrC$ by Theorem~\ref{regular covering}.
%The inclusion $\Br\scrC\subset\cAut{}\scrC$ is clear, since by Theorem~\ref{regular covering} $\Br\scrC$ acts as the Galois group.  
For the reverse inclusion, consider $g\in \cAut{}\scrC$.  Since $g\colon \Db(\coh X)\to\Db(\coh X)$ commutes with $\mathbf{R}f_*$, passing through $\Uppsi\colonequals\RHom_X(\scrV,-)$ to obtain $\Uppsi\circ g\circ\Uppsi^{-1}\colon\Db(\fmod\Lambda)\to\Db(\fmod\Lambda)$, necessarily by \cite[2.14]{HomMMP} $\Uppsi\circ g\circ\Uppsi^{-1}$ commutes with the exact functor $e(-)$, where $e$ is the idempotent of $\Lambda$ corresponding to $R$.

Since $g$ preserves $\cStab{}\scrC$, $\mathrm{V}\colonequals\Uppsi\circ g\circ\Uppsi^{-1}(\mathrm{U_1})$ is open and by Theorem~\ref{chambers in cStab C}  $\scrM$ is dense in $\cStab{}\scrC$, necessarily $V\cap\scrM\neq \emptyset$.  Thus we must have $ \Uppsi\circ g\circ\Uppsi^{-1}(\scrA)=\scrA_\upalpha$ for some $\upalpha\in\Hom_{\dsG}(C_B,C_+)$.  Consider the composition
\[
G=\Upphi_\upalpha^{-1}\circ \Uppsi\circ g\circ\Uppsi^{-1}\colon  
\Db(\fmod\Lambda)\to\Db(\fmod\Lambda_B),
\]
which takes $\scrA$ to the standard heart on $\scrC_B$.  Now $\Upphi_{\upalpha}^{-1}=\Upphi_{\upalpha^{-1}}$ is a composition of mutation functors and their inverses, where we do not mutate the vertex $R$. By \cite[4.2]{HomMMP} these are functorially isomorphic to flop functors and their inverses, which commute with $\mathbf{R}f_*$. Again by \cite[2.14]{HomMMP}, this translates into $\Upphi_{\upalpha^{-1}}$ commuting with $e(-)$.  Consequently the composition  $G$ commutes with $e(-)$.

Since $G$ takes $\scrA$ to a standard algebraic heart, necessarily $G$  takes the simples $\scrS_1,\hdots,\scrS_n$ to simples, a priori with a permutation.  Hence (1) $\Uppsi^{-1}\,\circ G\circ\,\Uppsi$ commutes with $\mathbf{R}f_*$, and (2) it sends $\scrO_{\Curve_1}(-1),\hdots,\scrO_{\Curve_n}(-1)$ to themselves, a priori up to permutation.  But exactly as in \cite[7.17]{DW1}, property (1) implies that $\Uppsi^{-1}\,\circ G\circ\,\Uppsi$ preserves $\scrC$, and property (2) implies that $\Uppsi^{-1}\,\circ G\circ\,\Uppsi$ preserves the null category $\{ a\in\coh X\mid \mathbf{R}f_*a=0\}$. This implies that it necessarily preserves  $\scrC^{>0}$ and $\scrC^{<0}$, and hence preserves zero perverse sheaves $\Per X$.  In particular, $G(\scrS_0)\cong\scrS_0$, as the other simples are permuted.

But since  $\Uppsi^{-1}\,\circ G\circ\,\Uppsi$ preserves $\Per X$,  $G$ restricts to a Morita equivalence $\fmod\Lambda\to\fmod\Lambda_B$.  In particular projectives map to projectives, so since $G(\scrS_0)\cong\scrS_0$, under the pairing we have $G(\scrP_0)\cong\scrP_0$. Furthermore, since $\Lambda$ and  $\Lambda_B$ are basic, the Morita equivalence sends $\Lambda\mapsto\Lambda_B$.  Since $G$ commutes with $e(-)$, it follows that $N\cong B$ in $\Db(\fmod R)$, and so $\Upphi_\upalpha\in\Br\scrC$.

But now $\Uppsi^{-1}\,\circ G\circ\,\Uppsi$ sends $\scrO_X\mapsto \scrO_X$, since $G(\scrP_0)\cong\scrP_0$, and it preserves $\Per X$.  By the standard Toda argument (see e.g.\ \cite[7.18]{DW1}), $\Uppsi^{-1}\,\circ G\circ\,\Uppsi\cong \varphi_*\circ(-\otimes\scrL)$ for some isomorphism $\varphi\colon X\to X$ and some line bundle $\scrL$.  The line bundle $\scrL$ is trivial since $\scrO_X\mapsto \scrO_X$.  The isomorphism $\varphi$ commutes with $\mathbf{R}f_*$ since $\Uppsi^{-1}\,\circ G\circ\,\Uppsi$ does, and hence $\varphi$ is the identity, given it must be the identity on the dense open set obtained by removing the flopping curve.  It follows that $G\cong \Id$, and so $g= \Uppsi^{-1}\circ\Upphi_\upalpha\circ\Uppsi\in\Br\scrC$, as required.
\end{proof}

\subsection{Identifying Line Bundle Twists}\label{section line bundles}
To describe $\cAut{}\scrD$ requires us to first realise twists by line bundles as compositions of mutation functors. Set $\mathsf{L}_i\colonequals f_*\scrL_i$ and consider the subgroup $\mathbb{Z}^n\cong\langle \mathsf{L}_1,\hdots,\mathsf{L}_n\rangle\leq \Cl(R)$. For any $\mathsf{L}$ in this subgroup and any $M\in \Mut(N)$, $\mathsf{L}\cdot M\colonequals (\mathsf{L}\otimes M)^{**}$ belongs to the mutation class of $N$, in the following predictable way.

\begin{lem}\label{class action comb}\cite[9.10(2)(3)]{IW9}
The arrangement $\scrH^{\aff}$ has a $\mathbb{Z}^n$-action, where its generators take the chamber corresponding to $N$ to the next chamber along the primitive vectors of the $\mathbb{Z}^n$-lattice given by the chamber $N$.
\end{lem}
To illustrate this, consider the two-curve example in Figure~\ref{fig: Figure Z2}. Then the $\mathbb{Z}^2$-lattice is the black dots, and the primitive vectors of the $\mathbb{Z}^2$-lattice given by the chamber $N$ are the red arrows.
\begin{figure}[ht]
\[
\begin{tikzpicture}[,>=stealth]
\clip (-3.75,-2.5) rectangle (3.75,3.5);
\node at (0,0)
{$\begin{tikzpicture}
[xscale=\EeightFourScalex,yscale=\EeightFourScaley]
\foreach \y in {-1,0,1,2}\foreach \x in {0,1,2,3,4}
{
\node at (4*\x,4*\y) {$\EeightFour$};
}
\end{tikzpicture}$};
\filldraw[gray] (-0.9,-1.9) -- (-0.9,-0.91) -- (-0.41,-0.91) -- cycle;
\filldraw[gray] ($(-0.9,-1.9)+(0,2)$) -- ($(-0.9,-0.91)+(0,2)$) -- ($(-0.41,-0.91)+(0,2)$) -- cycle;
\filldraw[gray] ($(-0.9,-1.9)+(2,4)$) -- ($(-0.9,-0.91)+(2,4)$) -- ($(-0.41,-0.91)+(2,4)$) -- cycle;
\draw[red,thick,->] (-0.9,-1.9) -- (-0.9,0.1);
\draw[red,thick,->] (-0.9,-1.9) -- (1.1,2.1);
\node[rotate=60] at ($(-0.9,-1.9)+(0.17,0.7)$) {$\scriptstyle N$};
\node[rotate=60] at ($(-0.9,0.1)+(0.2,0.7)$) {$\scriptstyle \mathsf{L}_{\kern -0.5pt 1}\kern -1pt N$};
\node[rotate=60] at ($(1.1,2.1)+(0.2,0.7)$) {$\scriptstyle \mathsf{L}_{\kern -0.5pt 2}\kern -1pt N$};
\end{tikzpicture}
\]
\caption{Example of $\mathbb{Z}^2$ action on $\scrH^{\aff}$}\label{fig: Figure Z2}
\end{figure}
For any such $\mathsf{L}$ and $M\in \Mut(N)$, consider the isomorphism $\upvarepsilon\colon\Lambda_M\to\Lambda_{\mathsf{L}\cdot M}$ defined to be 
\[
\Lambda_M=\End_R(M)\xrightarrow{(-\otimes \mathsf{L})^{**}}\End_R((M\otimes \mathsf{L})^{**}).
\]
Being an isomorphism of algebras, $\upvarepsilon$ induces an isomorphism of categories $\Db(\fmod\Lambda_M)\to\Db(\fmod\Lambda_{\mathsf{L}\cdot M})$, which we will also denote by $\upvarepsilon$. 

For $\mathsf{L}\in \langle \mathsf{L}_1,\hdots,\mathsf{L}_n\rangle$ and a path $\upalpha\in\Hom_{\dsG^{\aff}}(C_{M_1},C_{M_2})$, the translation by $\mathsf{L}$ defines a path from $C_{\mathsf{L}\cdot M_1}$ to $C_{\mathsf{L}\cdot M_2}$, which we denote by $\mathsf{L}\cdot\upalpha$. 

 \begin{lem}\label{ep L commute}
Let $\mathsf{g},\mathsf{h}\in \langle \mathsf{L}_1,\hdots,\mathsf{L}_n\rangle$, $M\in\Mut(N)$, and $\upalpha\colon C_{M}\to C_{\mathsf{g}\cdot M}$ be a positive minimal path. Then the following diagram commutes.
\[
\begin{tikzpicture}
\node (A1) at (-0.1,0) {$\Db(\fmod\Lambda_M)$};
\node (A2) at (3.6,0) {$\Db(\fmod\Lambda_{\mathsf{g}\cdot M})$};
\node (B1) at (-0.1,-1.4) {$\Db(\fmod \Lambda_{\mathsf{h}\cdot M})$};
\node (B2) at (3.6,-1.4) {$\Db(\fmod \Lambda_{\mathsf{h}\cdot\mathsf{g}\cdot M})$};
\draw[->] (A1) -- node[above] {\scriptsize $\Upphi_{\upalpha}$} (A2);
\draw[->] (B1) -- node[above] {\scriptsize $\Upphi_{\mathsf{h}\cdot\upalpha}$} (B2);
\draw[->] (A1) -- node[left] {\scriptsize $\upvarepsilon$} (B1);
\draw[->] (A2) -- node[right] {\scriptsize $\upvarepsilon$} (B2);
\end{tikzpicture}
\]
\end{lem}
\begin{proof}
By Remark~\ref{pos min are functorial} the top functor composed with the right hand functor is given by the tilting bimodule $\Hom_R(M,\mathsf{g}\cdot M)$, with the natural action of $\Lambda_M$, and the action of  $\Lambda_{\mathsf{h}\cdot\mathsf{g}\cdot M}$ twisted by the isomorphism $\upvarepsilon$.   Applying $\mathsf{h}\cdot$, this is clearly isomorphic, as bimodules, to $\Hom_R(\mathsf{h}\cdot M,\mathsf{h}\cdot \mathsf{g}\cdot M)$, with the action of $\Lambda_M$ twisted by $\upvarepsilon$, and the natural action of $\Lambda_{\mathsf{h}\cdot\mathsf{g}\cdot M}$.  But this is the bimodule that realises the left hand functor composed with the bottom functor, and so the diagram commutes.
\end{proof}

\begin{thm}\label{class action main}
Suppose that $X\to\Spec R$ a 3-fold flop, where $X$ has at worst terminal singularities. Writing $\upkappa$ for a minimal chain of mutation functors from one algebra to the other, for all $i=1,\hdots, n$ the following diagram commutes
\[
\begin{tikzpicture}[xscale=1.3]
\node (A0) at (0,1.5) {$\Db(\coh X)$};
\node (A5) at (5,1.5) {$\Db(\coh X)$};
\node (B0) at (0,0) {$\Db(\fmod\Lambda)$};
\node (B4) at (2.75,0) {$\Db(\fmod\Lambda_{\mathsf{L}_i\cdot N})$};
\node (B5) at (5,0) {$\Db(\fmod\Lambda)$};
%\node (C0) at (0,-1) {$\Db(\fmod\Lambda)$};
%\node (C4) at (2.75,-1) {$\Db(\fmod\Lambda_{L_i\cdot N})$};
%\node (C5) at (5,-1) {$\Db(\fmod\Lambda)$};
%\node (D0) at (0,-2.5) {$\Db(\coh X)$};
%\node (D5) at (5,-2.5) {$\Db(\coh X)$};
\draw[->] (A0) -- node[above] {$\scriptstyle -\otimes\scrL_i^*$}(A5);
\draw[->] (B0) -- node[above] {$\scriptstyle \upkappa$}(B4);
\draw[->] (B4) -- node[above] {$\scriptstyle \upvarepsilon^{-1}$}(B5);
%\draw[<-] (C0) -- node[above] {$\scriptstyle \upkappa$}(C4);
%\draw[->] (C5) -- node[above] {$\scriptstyle \upvarepsilon$}(C4);
\draw[<-] (B0) -- node[left] {$\scriptstyle \Uppsi$}(A0);
\draw[<-] (B5) -- node[left] {$\scriptstyle \Uppsi$}(A5);
%\draw[<-] (C0) -- node[left] {$\scriptstyle \Uppsi$}(D0);
%\draw[<-] (C5) -- node[left] {$\scriptstyle \Uppsi$}(D5);
%\draw[<-] (D0) -- node[above] {$\scriptstyle \Flop^{-1} \circ (-\otimes\scrL_i^*) \circ \Flop$}(D5);
\end{tikzpicture}
\]
\end{thm}
\begin{proof}
By Remark~\ref{pos min are functorial} the minimal chain of mutations $\upkappa$ is functorially isomorphic to the $\RHom$ functor given by the tilting bimodule $\Hom_R(N,\mathsf{L}_i\cdot N)$.  So, setting $\scrV=\scrV_X$ and taking inverses, it suffices to prove that the diagram
\begin{equation}
\begin{array}{c}
\begin{tikzpicture}[xscale=1.15]
\node (A0) at (-0.5,1.5) {$\Db(\coh X)$};
\node (A5) at (7,1.5) {$\Db(\coh X)$};
\node (B0) at (-0.5,0) {$\Db(\Lambda)$};
\node (B4) at (3.9,0) {$\Db(\End_X(\scrV\otimes\scrL_i))$};
\node (B5) at (7,0) {$\Db(\End_X(\scrV))$};
\draw[<-] (A0) -- node[above] {$\scriptstyle -\otimes\scrL_i$}(A5);
\draw[<-] (B0) -- node[above] {$\scriptstyle -\otimes^{\bf L}\Hom_X(\scrV,\scrV\otimes\scrL_i)$}(B4);
\draw[<-] (B4) -- node[above] {$\scriptstyle \upvarepsilon$}(B5);
\draw[->] (B0) -- node[left] {$\scriptstyle -\otimes^{\bf L}\scrV$}(A0);
\draw[->] (B5) -- node[right] {$\scriptstyle -\otimes^{\bf L}\scrV$}(A5);
\end{tikzpicture}
\end{array}
\label{diagram to comm}
\end{equation}
commutes, where derived tensors are over the relevant endomorphism ring.

The bottom composition is isomorphic to the functor
\[
-\otimes^{\bf L}_{\Lambda} {}_{\upvarepsilon}\Hom_X(\scrV,\scrV\otimes\scrL_i)_{\Id}
\]
by which we mean the $\Lambda\cong\End_Y(\scrV)$-action on $\Hom_X(\scrV,\scrV\otimes\scrL_i)$ on the left is the composition action via $\upvarepsilon$, and the $\Lambda$-action on the right is the standard composition action.    Hence the composition from bottom right to top left, along first the bottom and then the left, is functorially isomorphic to 
\begin{equation}
-\otimes^{\bf L}_{\Lambda} ({}_{\upvarepsilon}\Hom_X(\scrV,\scrV\otimes\scrL_i)_{\Id}\otimes^{\bf L}_{\Lambda}\scrV)\label{bimod comp 1}
\end{equation}
We first claim that this bimodule has cohomology only in degree zero, and for this we can ignore the left action of $\Lambda$ completely, and consider  $\Hom_X(\scrV,\scrV\otimes\scrL_i)_{\Id}\otimes^{\bf L}_{\Lambda}\scrV$.  Since $\scrL_i$ is generated by global sections, there exists a surjection $\scrO^{N}\twoheadrightarrow \scrL_i$, for some $N$, and so tensoring by $\scrV$ gives a surjection $\scrV^{N}\twoheadrightarrow \scrV\otimes\scrL_i$. Applying $\Hom_X(\scrV,-)$ gives an exact sequence
\[
\Ext^1_X(\scrV,\scrV)^N\to\Ext^1_X(\scrV,\scrV\otimes\scrL_i)\to 0
\]
where $\Ext^2$ is zero since flopping contractions have fibre dimension one.  Since $\scrV$ is tilting, $\Ext^1_X(\scrV,\scrV)=0$, thus $\Ext^1_X(\scrV,\scrV\otimes\scrL_i)=0$.  It follows that 
\begin{equation}
\Hom_X(\scrV,\scrV\otimes\scrL_i)_{\Id}\otimes^{\bf L}_{\Lambda}\scrV
\cong
\RHom_X(\scrV,\scrV\otimes\scrL_i)_{\Id}\otimes^{\bf L}_{\Lambda}\scrV
\cong \scrV\otimes\scrL_i,\label{RHom isHom}
\end{equation}
which has degree zero, as claimed.

By the claim, after truncating in the category of bimodules,  \eqref{bimod comp 1} is functorially isomorphic to $-\otimes^{\bf L}_{\Lambda} ({}_{\upvarepsilon}\Hom_X(\scrV,\scrV\otimes\scrL_i)_{\Id}\otimes_{\Lambda}\scrV)$.  Thus \eqref{diagram to comm} commutes provided that
\begin{equation}
{}_{\upvarepsilon}\Hom_X(\scrV,\scrV\otimes\scrL_i)_{\Id}\otimes_{\Lambda}\scrV\cong \scrV\otimes\scrL_i\label{want bimod}
\end{equation}
as $\Lambda$-$\scrO_X$-bimodules.  But the derived counit map is an isomorphism, so by \eqref{RHom isHom} it follows that the non-derived counit map
\[
\Hom_X(\scrV,\scrV\otimes\scrL_i)_{\Id}\otimes_{\Lambda}\scrV\to \scrV\otimes\scrL_i
\]
is also an isomorphism.  But this map is the obvious one.  By inspection it respects the bimodule structure in \eqref{want bimod}, and thus it gives the required bimodule isomorphism.
\end{proof}

\begin{rem}\label{comp of line conventions}
The position of $\upvarepsilon$ at the end of the chain of mutations $\upkappa$ in Theorem~\ref{class action main} is irrelevant.  Indeed, given any minimal chain from $N$ and $\mathsf{L}_i\cdot N$, consider any chamber $A$ through which this minimal path passes.  Abusing notation and writing $\upvarepsilon$ for any isomorphism induced by the class group action, and $\upkappa$ for a minimal composition of mutation functors, the following top diagram also commutes
\[ 
\begin{tikzpicture}[xscale=1.3]
\node (A0) at (0,1.5) {$\Db(\coh X)$};
\node (A5) at (7.5,1.5) {$\Db(\coh X)$};
\node (B0) at (0,0) {$\Db(\fmod\Lambda)$};
\node (B3) at (2.5,0) {$\Db(\fmod\Lambda_{A})$};
\node (B4) at (5,0) {$\Db(\fmod\Lambda_{\mathsf{L}_i^{-1}\cdot A})$};
\node (B5) at (7.5,0) {$\Db(\fmod\Lambda)$};
\node (C5) at (7.5,-1.5) {$\Db(\fmod\Lambda_{\mathsf{L}_i^{-1}\cdot N})$};
\draw[->] (A0) -- node[above] {$\scriptstyle -\otimes\scrL_i^*$}(A5);
\draw[->] (B0) -- node[above] {$\scriptstyle \upkappa$}(B3);
\draw[->] (B3) -- node[above] {$\scriptstyle \upvarepsilon^{-1}$}(B4);
\draw[->] (C5) -- node[right] {$\scriptstyle \upvarepsilon^{-1}$}(B5);
\draw[->] (B4) -- node[above] {$\scriptstyle \upkappa$}(B5);
\draw[->] (B3) -- node[below] {$\scriptstyle \upkappa$}(C5);
\draw[<-] (B0) -- node[left] {$\scriptstyle \Uppsi$}(A0);
\draw[<-] (B5) -- node[right] {$\scriptstyle \Uppsi$}(A5);
\end{tikzpicture}
\]
since the outer diagram commutes by Theorem~\ref{class action main}, and the commutativity of the bottom diagram is Lemma~\ref{ep L commute}.  Hence, when composing line bundle twists, we can move all the $\upvarepsilon^{-1}$s to the right, or to the left, whichever is the most convenient.
\end{rem}

\subsection{\texorpdfstring{$\Pic X$}{Pic X} action}\label{Pic action}

Recall that perfect complexes on $X$ are equivalent to $\Kb(\proj\Lambda)$, and $\Pic X\cong \bZ^{n}$ with basis $\scrL_1,\hdots,\scrL_n$, where $\scrL_i$ is the line bundle satisfying $\deg(\scrL_i|_{\Curve_j})=\delta_{ij}$.  For $\scrL\in \Pic X$, by abuse of notation, set
\[
-\otimes\scrL\colonequals\Uppsi\circ(-\otimes\scrL)\circ\Uppsi^{-1}\colon \Db(\fmod\Lambda)\simto \Db(\fmod\Lambda).
\]

\begin{lem}\label{Pic action K}
Write $z=(z_j)_{j=0}^n\in\Level_{\mathbb{C}}$ with $z_j=x_j+\ii y_j$.  Then for $j>0$, and any $1\leq i\leq n$, the autoequivalence $-\otimes\scrL_i$ acts on $\Level_{\mathbb{C}}$ by
\[
x_j+\ii y_j
\mapsto
\begin{cases}
x_j+\ii(y_j+1) &j=i\\
x_j+\ii y_j & j\neq i.
\end{cases}
\]
\end{lem}
\begin{proof}
To ease notation, consider only the action of $\scrL_1$, since all other cases are identical.  We first consider the action on the dual vector space $K_0(\scrD)$ of $\scrK$, with basis $[\scrS_0],\hdots,[\scrS_n]$.  Recall that $\Uppsi(\upomega_{\Curve}[1])\cong \scrS_0$ and  $\Uppsi(\scrO_{\Curve_i}(-1))\cong \scrS_i$ for $1\leq i\leq n$.
Since  $\deg(\scrL_j|_{\Curve_i})=\delta_{ij}$, 
\[
\scrO_{\Curve_i}(-1)\otimes \scrL_1\cong 
\begin{cases}
\scrO_{\Curve_1} & i=1\\
\scrO_{\Curve_i}(-1) & i\neq 1.
\end{cases}
\]
  Let $\cO_x$ be the skyscraper sheaf associated to a closed point $x\in \Curve$. Then $\RHom(\scrV_i^*,\scrO_x)\cong \Hom(\scrV^*_i,\scrO_x)=\mathbb{C}^{\rk\kern -1pt N_i}$, so setting $\ell_i\colonequals \rk_R N_i$ gives
\begin{equation}\label{k sky}
[\scrO_x]=-[\upomega_{\Curve}]+\sum_{i=1}^n\ell_i[\scrO_{\Curve_i}(-1)].
\end{equation}
 If $x\in \Curve_i$, then the exact sequence $0\to\scrO_{\Curve_i}(-1)\to\scrO_{\Curve_i}\to\scrO_x\to0$ implies that
$[\scrO_{\Curve_i}]=[\scrO_{\Curve_i}(-1)]+[\scrO_x]$. Combining this with \eqref{k sky}, it follow that 
\begin{equation}\label{k str}
[\scrO_{\Curve_i}]=-[\upomega_{\Curve}]+(\ell_i+1)[\scrO_{\Curve_i}(-1)]+\sum_{k\neq i}\ell_k[\scrO_{\Curve_k}(-1)],
\end{equation}
and so
\begin{equation}
 [\scrO_{\Curve_i}(-1)\otimes \scrL_1]=
\begin{cases}
[\upomega_{\Curve}[1]]+(\ell_1+1)[\scrO_{\Curve_1}(-1)]+\sum_{k\neq 1}\ell_k[\scrO_{\Curve_k}(-1)] & i=1\\
[\scrO_{\Curve_i}(-1)] & i\neq j.
\end{cases}\label{Action 1}
\end{equation}
Next we determine the action of $\scrL_1$ on $[\upomega_{\Curve}[1]]=-[\upomega_{\Curve}]$. Since $\scrO_x\otimes \scrL_1\cong \scrO_x$, we have 
\[
[\scrO_x]=-[\upomega_{\Curve}\otimes \scrL_1]+\sum_{i\neq 1}\ell_i[\scrO_{\Curve_i}(-1)]+\ell_1[\scrO_{\Curve_1}].
\]
Substituting \eqref{k sky} to the left hand side and \eqref{k str} to the third term of the right hand side,
\begin{equation}
[\upomega_{\Curve}[1]\otimes \scrL_1]=(1-\ell_1)[\upomega_{\Curve}[1]]-\sum_{i=1}^n\ell_i\ell_1[\scrO_{\Curve_i}(-1)].\label{Action 2}
\end{equation}
Combining \eqref{Action 1} and \eqref{Action 2}, the action of $\scrL_1$ on $K_0(\scrD)$ is given by the $(n+1)\times (n+1)$ matrix whose first row is $(1-\ell_1,1,0,\hdots,0)$, whose second row is $(-\ell_1^2,\ell_1+1,0,\hdots,0)$, and $(-\ell_1\ell_j,\ell_j,0,\hdots,1,\hdots,0)$ is the $j$th row for any $j> 1$, where $1$ is in the $(j+1)$st entry.

The transpose of this matrix gives the required action on the dual $\scrK$.  But for $j>0$, the  transpose has $(j+1)$st row $(0,\hdots,1,\hdots,0)$ where the only non-zero entry is in the $(j+1)$st position, and the first row of the transpose matrix is $(1,\ell_1+1,\ell_2,\hdots,\ell_n)$.  Since $z\in\Level_{\mathbb{C}}$, by definition  $x_0=-\sum_{i=1}^n\ell_ix_i$ and $y_0=1-\sum_{i=1}^n\ell_iy_i$.  Using this, together with the rows identified, it is easy to check that the action is as stated.
\end{proof}

It will be convenient to visualise this in the case of an irreducible flopping contraction.  Recall from Example~\ref{Ex3.11} that $(z_0,z_1)\in\Level_{\bC}$ is determined by $(x_1,y_1)$, and that by Lemma~\ref{Pic action K}, the autoequivalence $-\otimes\scrO(1)$ acts by sending $(x_1,y_1)\mapsto (x_1,y_1+1)$.  As in Example~\ref{Ex3.11}, we draw the $y_1$ axis horizontally pointing to the left, and the $x_1$ axis vertically.  Thus, with this convention, tensoring by $\scrO(1)$ translates to the left.    The case $\ell=3$ is illustrated below.
\begin{equation}
\begin{array}{c}
\begin{tikzpicture}
\filldraw[gray!10!white] (-1,1) -- (-1,-1) -- (0,-1)--(0,1)-- cycle;
\filldraw[gray!10!white] (4,1) -- (4,-1) -- (3,-1)--(3,1)-- cycle;
\draw[black] (-1,0) -- (-1,1);
\draw[black,densely dotted] (-1,0) -- (-1,-1);
\draw[black] (0,0) -- (0,-1);
\draw[black,densely dotted] (0,0) -- (0,1);
\draw[black] (3,0) -- (3,1);
\draw[black,densely dotted] (3,0) -- (3,-1);
\draw[black] (4,0) -- (4,-1);
\draw[black,densely dotted] (4,0) -- (4,1);
{\foreach \i in {-5,-4,-3,-2,-1,0,1,2}
\filldraw[fill=white,draw=black] (-1*\i,0) circle (2pt);
}
\draw[<-] (0.5,0.5) -- node[above] {$\scriptstyle -\otimes\scrO(1)$}(2.5,0.5) ;
\node at (0.15,-0.15) {$\scriptstyle 1$};
\node at (1.15,-0.15) {$\scriptstyle \frac{2}{3}$};
\node at (2.15,-0.15) {$\scriptstyle \frac{1}{2}$};
\node at (3.15,-0.15) {$\scriptstyle \frac{1}{3}$};
\node at (4.15,-0.15) {$\scriptstyle 0$};
\draw[blue,<-] (7,1) -- (7,-1);
\draw[blue,<-] (6,0) -- (8,0);
\node[blue] at (7.25,1) {$\scriptstyle x_1$};
\node[blue] at (5.75,0) {$\scriptstyle y_1$};
\end{tikzpicture}
\end{array}
\label{Pic for ell 3}
\end{equation}

\subsection{Autoequivalences of \texorpdfstring{$\scrD$}{D}}\label{sec: auto of D}
As in the introduction, consider $\cAut{}\scrD$, defined to be the subgroup of $\Auteq\scrD$ consisting of those $\Upphi|_\scrD\colon\scrD\to\scrD$ where $\Upphi$ is an $R$-linear Fourier--Mukai equivalence $\Upphi\colon\Db(\coh X)\to\Db(\coh X)$ that preserves $\cStab{n}\scrD$.  

 In order to control elements in $\cAut{}\scrD$, we first control isomorphisms of algebras that preserves the space of normalised stability conditions. Suppose that $\Lambda_L$ and $\Lambda_M$ arise from  $L=\bigoplus_{i=0}^nL_i$ and $M=\bigoplus_{i=0}^nM_i$ in $\Mut(N)$, and 
\[
\uprho\colon\Lambda_L\to\Lambda_M
\]
is an isomorphism of rings.  Let $\scrS_i$  be the simple $\Lambda_L$-module corresponding to the projective module $\Hom_R(L,L_i)$. Being an isomorphism, simples get sent to simples, and so there is a bijection $\upiota\colon \{0,\hdots,n\}\to \{0,\hdots,n\}$ such that  $\uprho_*\scrS_i\cong \scrS'_{\upiota(i)}$, where $\scrS'_j$ is the simple $\Lambda_M$-module corresponding to the summand $M_j$.

\begin{lem}\label{iso lemma}
$\uprho_*$ preserves $\cStab{n}\scrD$ if and only if $\rk_RL_i=\rk_R M_{\upiota(i)}$ for all $0\leq i\leq n$.
\end{lem}
\begin{proof}
($\Leftarrow$) It suffices to show that  $\uprho_*(Z,\scrB_{\Lambda_L})=(Z\circ \uprho^{-1},\scrB_{\Lambda_M})\in\cStab{n}\scrD_{\Lambda_M}$ for all $(Z,\scrB_{\Lambda_L})\in\cStab{n}\scrD_{\Lambda_L}$.  It is clear that $\uprho_*$ preserves the stability component, since  $Z\circ\uprho^{-1}[\scrS'_{\upiota(i)}]=Z[\scrS_i]\in\bH$ for each $0\leq i\leq n$, 
and, furthermore, $\uprho_*$ preserves the normalisation since
\[
\sum_{i=0}^n(\rk_R M_{\upiota(i)})\cdot Z\circ \uprho^{-1}[\scrS'_{\upiota(i)}]=\sum_{i=0}^n(\rk_RL_i)\cdot Z[\scrS_i]=\ii.
\]
($\Rightarrow$) 
To ease notation, set $a_i\colonequals \rk_RL_i$, $b_j=\rk_R M_j$.  Furthermore, for each $i\in\{0,\hdots,n\}$ consider $Z_i\colon K_0(\scrD)\to \bC$ defined by 
\[
Z_i([\scrS_j])\colonequals
\begin{cases}
-1 & j\neq i\\
\frac{a-a_i}{a_i}+\ii\frac{1}{a_i}& j=i
\end{cases}
\]
where $a\colonequals \sum_{i=0}^n a_i$. Now each $Z_i([\scrS_j])$ belongs to $\bH$, and $\sum_{i=0}^na_jZ_i([\scrS_j])
=\ii $, hence $(Z_i,\scrB_{\Lambda_L})\in\cStab{n}\scrD_{\Lambda_L}$.  Since $\uprho_*$ preserves normalised stability conditions by assumption, 
\[
\sum_{j=0}^nb_{\upiota(j)}Z_i\circ \uprho^{-1}([\scrS'_{\upiota(j)}])=\sum_{j=0}^nb_{\upiota(j)}Z_i([\scrS_{j}])=\ii 
\]
Considering the imaginary parts of  both sides, we see that
$
a_i=b_{\upiota(i)}.
$
\end{proof}

Our next main result, Theorem~\ref{semidirect main}, requires two technical lemmas.
\begin{lem}\label{decomposePic}
For any $\mathsf{g}\in \langle \mathsf{L}_1,\hdots,\mathsf{L}_n\rangle$, consider the isomorphism $\upvarepsilon\colon \Lambda\to\Lambda_{\mathsf{g}\cdot N}$.  Then there is an isomorphism $\upvarepsilon\cong\Upphi_{\upgamma} \circ(- \otimes\scrL)$ for some $\scrL\in\Pic(X)$ and some $\upgamma\in\Hom_{\mathds{G}^{\aff}}(N,\mathsf{g}\cdot N)$.
\end{lem}
\begin{proof}
By assumption we may write $\mathsf{g}=\mathsf{L}_{i_1}^{\sigma_1}\cdot\mathsf{L}_{i_2}^{\sigma_2}\cdot\hdots \cdot\mathsf{L}_{i_k}^{\sigma_k}$ for some $k\geq0$, $1\leq i_j\leq n$ and $\sigma_i\in \{\pm1\}$.   If $k=0$, we have nothing to prove, and so we assume $k>0$. 

We prove the result by induction on $k$, so first assume that $k=1$.  There are two cases, depending on the parity of $\sigma_1$.  In the case $\sigma_1=1$, we are considering the isomorphism $\upvarepsilon\colon\Lambda\to\Lambda_{\mathsf{L}_i\cdot N}$.  By Theorem \ref{class action main}, the functor 
\begin{equation}
(-\otimes \scrL_i^{-1})\colon\Db(\fmod\Lambda)\to \Db(\fmod\Lambda)\label{tensor for induct}
\end{equation}
is isomorphic to the composition 
\[
\Db(\fmod\Lambda)\xrightarrow{\Upphi_{\upalpha}}\Db(\fmod\Lambda_{\mathsf{L}_i\cdot N})\xrightarrow{\upvarepsilon^{-1}}\Db(\fmod\Lambda)
\]
from which it follows that $ \upvarepsilon\cong\Upphi_\upalpha\circ (-\otimes \scrL_i) $.  On the other hand, in the case $\sigma_1=-1$, we are considering the isomorphism $\upvarepsilon\colon\Lambda\to\Lambda_{\mathsf{L}_i^{-1}\cdot N}$.  However, by Theorem~\ref{class action main} and Lemma~\ref{ep L commute} applied to $\mathsf{g}=\mathsf{L}_i$ and $\mathsf{h}=\mathsf{g}^{-1}$, the functor \eqref{tensor for induct} is isomorphic to 
\[
\Db(\fmod\Lambda)\xrightarrow{\upvarepsilon}\Db(\fmod\Lambda_{\mathsf{h}\cdot N})\xrightarrow{\Upphi_{\mathsf{h}\cdot\upalpha}}\Db(\fmod\Lambda),
\]
from which it follows that $\upvarepsilon\cong \Upphi_{(\mathsf{h}\cdot\upalpha)^{-1}}\circ (-\otimes\scrL_i^{-1})$.  Either way, the result holds for $k=1$.

For the induction step, assume that $k>1$ and set $\mathsf{h}\colonequals\mathsf{g}\cdot\mathsf{L}_{i_k}^{-\sigma_k}$ and  $\mathsf{k}\colonequals\mathsf{L}_{i_k}^{\sigma_k}$, so  $\mathsf{g}= \mathsf{h}\cdot \mathsf{k}$. By the inductive hypothesis, there exist $\scrM\in\Pic X$ and $\upbeta\in \Hom_{\dsG^{\aff}}(N,\mathsf{h}\cdot N)$ such that the isomorphism $\upvarepsilon\colon \Lambda\to\Lambda_{\mathsf{h}\cdot N}$ is isomorphic to $\Upphi_{\upbeta}\circ (-\otimes \scrM)$.  Combining this with Lemma~\ref{ep L commute}, the following diagram commutes 
\[
\begin{tikzpicture}
%\node (U1) at (-0.1,1.4) {$\Db(\fmod\Lambda_{\mathsf{k}^{-1}\cdot N})$};
\node (U2) at (3.6,1.4) {$\Db(\fmod\Lambda)$};
\node (A1) at (-0.1,0) {$\Db(\fmod\Lambda)$};
\node (A2) at (3.6,0) {$\Db(\fmod\Lambda_{\mathsf{h}\cdot N})$};
\node (B1) at (-0.1,-1.4) {$\Db(\fmod \Lambda_{\mathsf{k}\cdot N})$};
\node (B2) at (3.6,-1.4) {$\Db(\fmod \Lambda_{\mathsf{g}\cdot N})$};
%\draw[->] (U1) -- node[above] {\scriptsize $\Upphi_{\upalpha}$} (U2);
%\draw[->] (U1) -- node[left] {\scriptsize $\upvarepsilon$} (A1);
\draw[->] (U2) -- node[above] {\scriptsize $-\otimes \scrM\,\,\,\,\,\,\,\,$} (A1);
\draw[->] (U2) -- node[right] {\scriptsize $\upvarepsilon$} (A2);
\draw[->] (A1) -- node[above] {\scriptsize $\Upphi_{\upbeta}$} (A2);
\draw[->] (B1) -- node[above] {\scriptsize $\Upphi_{\mathsf{k}\cdot\upbeta}$} (B2);
\draw[->] (A1) -- node[left] {\scriptsize $\upvarepsilon$} (B1);
\draw[->] (A2) -- node[right] {\scriptsize $\upvarepsilon$} (B2);
\draw[transform canvas={xshift=5ex},bend left,->] (U2) to node[right]{\scriptsize $\upvarepsilon$}  (B2); 
\end{tikzpicture}
\]
where the right hand composition is our desired isomorphism $\upvarepsilon\colon\Lambda\to\Lambda_{\mathsf{g}\cdot N}$.   Now by length one case proved above applied to $\mathsf{k}$, the bottom left $\upvarepsilon$ is isomorphic to $\Upphi_{\upgamma_k}\circ (- \otimes \scrL_{i_k}^{\sigma_k})$ for some $\upgamma_k\in \Hom_{\dsG^{\aff}}(N,\mathsf{k}\cdot N)$. 
Thus if we set $\upgamma\colonequals (\mathsf{k}\cdot\upbeta)\circ \upgamma_k\in \Hom_{\dsG^{\aff}}(N,\mathsf{g}\cdot N)$ and $\scrL\colonequals \scrM\otimes \scrL_{i_k}^{\sigma_k}$, then $\upvarepsilon\colon \Lambda\to \Lambda_{\mathsf{g}\cdot N}$ is isomorphic to the composition $\Upphi_{\upgamma}\circ (-\otimes \scrL)$.
%Hence if we set $\upgamma\colonequals(\mathsf{k}\cdot\upbeta)\circ\upgamma_k$ and $\scrL\colonequals \scrL_{i_k}^{\sigma_k}\otimes \scrM$, we have $\upvarepsilon\cong \Upphi_{\upgamma}\circ \otimes \scrL$.
\end{proof}

The following is an easy extension of results in \cite[\S9]{IW9}.
\begin{lem}\label{rank one} 
Suppose that $M\in\Mut(N)$, and that there is an $R$-linear isomorphism $\Lambda\to\Lambda_{M}$ such that $\scrS_0\mapsto \scrS_i$, where $\rk_R M_i=1$.  Then  $M\cong  \mathsf{L}\cdot N$ for some  $\mathsf{L}\in \langle \mathsf{L}_1,\hdots,\mathsf{L}_n\rangle$. 
\end{lem}
\begin{proof}
Since $\scrS_0\mapsto \scrS_i$, by the pairing between simples and projectives, it follows that there is a chain of $R$-module isomorphisms
\[
\Hom_R(N,R)=\scrP_0\cong\scrP_i=\Hom_R(M,M_i).
\]
Since $\rk_R M_i=1$, the right hand side is isomorphic to $\Hom_R(M\cdot M_i^{-1},R)$.  Hence applying $\Hom_R(-,R)$ shows that $M\cong M_i\cdot N$ as $R$-modules.   Since $M_i$ is a summand of $M$ of rank one, and $M\in\Mut(N)$, by the bijections in \ref{affine summary} $M_i$ must appear as a $\mathbb{Z}^n$-lattice point in the hyperplane arrangement (see \cite[9.10(2)(3)]{IW9}).
\end{proof}

The following is one of our key results.

\begin{thm}\label{semidirect main}
$\cAut{}\scrD\cong\Br\scrD\rtimes\Pic X$.
\end{thm}
\begin{proof}
Set $G=\cAut{}\scrD$.  Then  $K=\Br\scrD$ is a clearly a subgroup of $G$, as is $H=\Pic X$ since elements of $\Pic X$ are $R$-linear by Lemma~\ref{tensor is R linear},  and preserve $\cStab{n}\scrD$ by Theorem~\ref{class action main}.

 The remainder of the proof splits into five steps.

\step{1} We first claim that, as sets, $G=KH$.  Let $g\in G$.  Since $g$ preserves $\cStab{n}\scrD$, arguing as in Theorem~\ref{7.1}, there exists $\upbeta\in\Hom_{\dsG^{\aff}}(C_{M},C_+)$ such that $g(\scrB)=\scrB_\upbeta$.  The composition
\[
\begin{tikzpicture}
\node (A1) at (0,0) {$\Db(\fmod\Lambda)$};
\node (A2) at (3,0) {$\Db(\fmod\Lambda)$};
\node (A3) at (6,0) {$\Db(\fmod\Lambda_\upbeta)$};
\draw[->] (A1) -- node[above]{$\scriptstyle g$} (A2);
\draw[->] (A2) -- node[above]{$\scriptstyle \Upphi_{\upbeta}^{-1}$} (A3);
\end{tikzpicture}
\]
is $R$-linear, and restricts to an equivalence between finite length $\Lambda$-modules and finite length $\Lambda_\upbeta$-modules.  In particular, by Lemma~\ref{LemmaA}\eqref{LemmaA 1} and  Corollary~\ref{1 implies 3} the composition restricts to an  $R$-linear Morita equivalence
\[
\fmod\Lambda\xrightarrow{\Upphi_{\upbeta}^{-1}\circ g}\fmod\Lambda_\upbeta.
\]
Since both algebras are basic, necessarily this is induced by an $R$-algebra isomorphism $\upvarphi\colon \Lambda\to\Lambda_\upbeta$.  But by Lemma~\ref{iso lemma}, it follows that $M$ has rank one summand. By Lemma \ref{rank one}, there is $\mathsf{L}\in \langle \mathsf{L}_1,\hdots,\mathsf{L}_n\rangle$ such that $M\cong \mathsf{L}\cdot N$. 

Consider the composition autoequivalence 
\[
\begin{array}{c}
\begin{tikzpicture}
\node at (-2,0) {$\Upupsilon\colon$};
\node (A1) at (0,0) {$\Db(\fmod\Lambda)$};
\node (A2) at (3,0) {$\Db(\fmod\Lambda)$};
\node (A3) at (6,0) {$\Db(\fmod\Lambda_\upbeta)$};
\node (A4) at (9,0) {$\Db(\fmod\Lambda).$};
\draw[->] (A1) -- node[above]{$\scriptstyle g$} (A2);
\draw[->] (A2) -- node[above]{$\scriptstyle \Upphi_{\upbeta}^{-1}$} (A3);
\draw[->] (A3) -- node[above]{$\scriptstyle \upvarepsilon^{-1}$} (A4);
\end{tikzpicture}
\end{array}
\]
where $\upvarepsilon\colon\Lambda\simto \Lambda_{\mathsf{L}\cdot N}=\Lambda_{\upbeta}$. By a similar argument as above, $\Upupsilon$ is induced by an $R$-linear isomorphism 
$\uppsi\colon \Lambda\simto \Lambda$. Then $\uppsi_*\scrS_0\cong \scrS_i$ for some $0\leq i\leq n$, and by the pairing between projectives and simples, $\uppsi$ restricts to an isomorphism $\scrP_0\cong \scrP_i$ of $R$-modules. By Lemma \ref{iso lemma} we see that  $\rk_RN_i=1$.  It follows that  
\[
\Hom_R(N,R)=\scrP_0\cong \scrP_i\cong \Hom_R(N\cdot N_i^{-1},R)
\]
as $R$-modules, and so dualizing gives $N\cong N\cdot N_i^{-1}$. Then by Lemma  \ref{rank one}, $N_i$ necessarily lies in $\langle \mathsf{L}_1,\hdots,\mathsf{L}_n\rangle$. By Lemma~\ref{class action comb} we thus have $N_i\cong R$, since non-trivial elements of $\langle \mathsf{L}_1,\hdots,\mathsf{L}_n\rangle$ non-trivially translate the hyperplane arrangement.  Thus $\Upupsilon(\scrS_0)\cong\scrS_0$.  As an Morita equivalence, it follows that there exists a permutation $\upiota\in S_n$ such that $\Upupsilon(\scrS_i)\cong\scrS_{\upiota(i)}$ for all $1\leq i \leq n$.  Since $\Upupsilon$ preserves normalised stability conditions, by Lemma \ref{iso lemma} we see $\rk_RN_i=\rk_RN_{\upiota(i)}$. Hence by Proposition \ref{G is the identity}, there is a functorial isomorphism $\Upupsilon\cong \Id$, and so $g\cong \Upphi_{\upbeta}\circ \upvarepsilon$. Using Lemma~\ref{decomposePic} it follows that $g$ is functorially isomorphic to the composition
\[
\begin{array}{c}
\begin{tikzpicture}
\node (A1) at (0,0) {$\Db(\fmod\Lambda)$};
\node (A2) at (3,0) {$\Db(\fmod\Lambda_N)$};
\node (A3) at (6,0) {$\Db(\fmod\Lambda_{\mathsf{L}\cdot N})$};
\node (A4) at (9,0) {$\Db(\fmod\Lambda_N).$};
\draw[->] (A1) -- node[above]{$\scriptstyle F$} (A2);
\draw[->] (A2) -- node[above]{$\scriptstyle \Upphi_{\upgamma}$} (A3);
\draw[->] (A3) -- node[above]{$\scriptstyle \Upphi_{\upbeta}$} (A4);
\end{tikzpicture}
\end{array}
\]
for some $F\in\Pic(X)$. Hence $g\cong \Upphi_{\upbeta\circ\upgamma}\circ F\in KH$.

\step{2} Consider the subgroup $\Tr\scrD$ of $G$ consisting of those elements that are the identity on K-theory $\scrK$.  We know that $\Br\scrD\subseteq \Tr\scrD$ by Proposition~\ref{theta trivial}.  We claim that $\Br\scrD\supseteq \Tr\scrD$, so equality holds.  To see this, consider $t\in\Tr\scrD$.  By Step 1, since $t\in G$ we can write $t=kh$ for some $k\in K$ and some $h\in H$.  Thus $h=k^{-1}t$, and so $h$ is trivial on K-theory.  But by Lemma~\ref{Pic action K} the only line bundle twist that satisfies this is the identity.  Hence $h=1$, so $k=t$ and thus $t\in K=\Br\scrD$.

\step{3}  $K\unlhd\, G$.  This follows immediately from Step 2, since being the identity on K-theory is clearly closed under conjugation.

\step{4} $K\cap H=\{1_G\}$.  Again, this holds by Lemma~\ref{Pic action K}, since the only line bundle twist that is the identity on K-theory is the identity.
\smallskip

Combining Steps 1, 3 and 4 we see that $G\cong K\rtimes H$, as required.
%\marginpar{This is the wikipedia definition: first equivalent bullet point}
\end{proof}

In particular, as is standard for semidirect products, there is an induced exact sequence
\begin{equation}
1\to\Br\scrD\to\cAut{}\scrD\to\Pic X\to 1.\label{split ses groups}
\end{equation}
This sequence is the generalisation of \cite[5.4(ii)]{T08} to higher length flops.

\subsection{Application: Stringy K\"ahler Moduli Spaces}
In this subsection we compute the stringy K\"ahler moduli space $\cStab{n}\scrD/\cAut{}\scrD$ for all smooth single curve flops.  In this case, Theorem~\ref{affine summary}  reduces to the statement that the mutation class containing $N$ is in bijection with the chambers of an infinite hyperplane arrangement in $\mathbb{R}^1$,  which we draw as
\[
\begin{tikzpicture}[xscale=0.85]
\draw[densely dotted] (-3,0)--(9.5,0);
\node (A) at (-2.2,0) [cvertex] {};
\node (B) at (0.8,0) [cvertex] {};
\node (C) at (3.5,0) [cvertex] {};
\node (D) at (6,0) [cvertex] {};
\node (E) at (8.7,0) [cvertex] {};
\end{tikzpicture}\qquad
\]
extended in both directions to infinity.  The walls are labelled by the indecomposable $R$-modules that are summands of elements in the mutation class of $N$, and the chambers are labelled by their direct sums. Wall crossing corresponds to mutation.  Since $N=R\oplus N_1$ from \eqref{ass modi} must appear in its mutation class, the centre of the hyperplane arrangement has the following form
\begin{equation}
\begin{tikzpicture}
\draw[densely dotted] (-2.5,0)--(9,0);
\node (A) at (-2.2,0) [cvertex] {};
\node (B) at (0.8,0) [cvertex] {};
\node (C) at (3.5,0) [cvertex] {};
\node (D) at (6,0) [cvertex] {};
\node (E) at (8.7,0) [cvertex] {};
\node at (-2.2,0.3) {$\scriptstyle $};
\node at (0.8,0.3) {$\scriptstyle N_{1}$};
\node at (3.5,0.3) {$\scriptstyle R$};
\node at (6,0.3) {$\scriptstyle N_{1}^*$};
\node at (8.7,0.3) {$\scriptstyle $};
%\draw[draw=none] (A) -- node[gap] {$$}(B);
\draw[draw=none] (B) -- node[gap] {$R\oplus N_1$}(C);
\draw[draw=none] (C) -- node[gap] {$R\oplus N_1^*$}(D);
%\draw[draw=none] (D) -- node[gap] {$$}(E);
\end{tikzpicture}
\quad
\label{arrangement labelled one curve}
\end{equation}
Under this convention, with $R\oplus N_1$ on the left and $R\oplus N_1^*$ on the right, $\mathsf{L}=f_*\scrL=f_*\scrO(1)$ generates a subgroup of  $\Cl(R)$ which acts on \eqref{arrangement labelled one curve},  taking the wall labelled $R$ to the next wall to the right for which the $R$-module labelling it has rank one.

The following computes the numerics of the hyperplane arrangements that can arise.
\begin{prop}\label{affine hyper calc}
Suppose that $X\to\Spec R$ is an irreducible length $\ell$ flop, where $X$ is smooth.  Then the corresponding affine hyperplane arrangement, together with the ranks of the modules labelling each wall, are, for $\ell=1,\hdots 6$ respectively:
\def\HypScalex{8}
\[
\begin{array}{c}
%%%%%%%%%%%%%%%%%%%%%l=1
\begin{tikzpicture}[scale=0.9]
\draw[densely dotted] (-1,0) -- ($(\HypScalex,0)+(1,0)$);
{\foreach \i in {0,2}
\filldraw[fill=white,draw=black] (\HypScalex*1/2*\i,0) circle (2pt);
}
\node at (0,-0.25) {$\scriptstyle 1$};
\node at (\HypScalex,-0.25) {$\scriptstyle 1$};
\end{tikzpicture}\\
%%%%%%%%%%%%%%%%%%%%%l=2
\begin{tikzpicture}[scale=0.9]
\draw[densely dotted] (-1,0) -- ($(\HypScalex,0)+(1,0)$);
{\foreach \i in {0,1,2}
\filldraw[fill=white,draw=black] (\HypScalex*1/2*\i,0) circle (2pt);
}
\node at (0,-0.25) {$\scriptstyle 1$};
\node at (\HypScalex,-0.25) {$\scriptstyle 1$};
\node at (1/2*\HypScalex,-0.25) {$\scriptstyle 2$};

\end{tikzpicture}\\
%%%%%%%%%%%%%%%%%%%%%l=3
\begin{tikzpicture}[scale=0.9]
\draw[densely dotted] (-1,0) -- ($(\HypScalex,0)+(1,0)$);
{\foreach \i in {0,1,2}
\filldraw[fill=white,draw=black] (\HypScalex*1/2*\i,0) circle (2pt);
}
{\foreach \i in {1,2}
\filldraw[fill=white,draw=black] (\HypScalex*1/3*\i,0) circle (2pt);
}
\node at (0,-0.25) {$\scriptstyle 1$};
\node at (\HypScalex,-0.25) {$\scriptstyle 1$};
\node at (1/2*\HypScalex,-0.25) {$\scriptstyle 2$};
\node at (1/3*\HypScalex,-0.25) {$\scriptstyle 3$};
\node at (2/3*\HypScalex,-0.25) {$\scriptstyle 3$};
\end{tikzpicture}\\
%%%%%%%%%%%%%%%%%%%%%l=4
\begin{tikzpicture}[scale=0.9]
\draw[densely dotted] (-1,0) -- ($(\HypScalex,0)+(1,0)$);
{\foreach \i in {0,1,2}
\filldraw[fill=white,draw=black] (\HypScalex*1/2*\i,0) circle (2pt);
}
{\foreach \i in {1,2}
\filldraw[fill=white,draw=black] (\HypScalex*1/3*\i,0) circle (2pt);
}
{\foreach \i in {1,2,3}
\filldraw[fill=white,draw=black] (\HypScalex*1/4*\i,0) circle (2pt);
}
\node at (0,-0.25) {$\scriptstyle 1$};
\node at (\HypScalex,-0.25) {$\scriptstyle 1$};
\node at (1/2*\HypScalex,-0.25) {$\scriptstyle 2$};
\node at (1/3*\HypScalex,-0.25) {$\scriptstyle 3$};
\node at (2/3*\HypScalex,-0.25) {$\scriptstyle 3$};
\node at (1/4*\HypScalex,-0.25) {$\scriptstyle 4$};
\node at (3/4*\HypScalex,-0.25) {$\scriptstyle 4$};
\end{tikzpicture}\\
%%%%%%%%%%%%%%%%%%%%%l=5
\begin{tikzpicture}[scale=0.9]
\draw[densely dotted] (-1,0) -- ($(\HypScalex,0)+(1,0)$);
{\foreach \i in {0,1,2}
\filldraw[fill=white,draw=black] (\HypScalex*1/2*\i,0) circle (2pt);
}
{\foreach \i in {1,2}
\filldraw[fill=white,draw=black] (\HypScalex*1/3*\i,0) circle (2pt);
}
{\foreach \i in {1,2,3}
\filldraw[fill=white,draw=black] (\HypScalex*1/4*\i,0) circle (2pt);
}
{\foreach \i in {1,2,3,4}
\filldraw[fill=white,draw=black] (\HypScalex*1/5*\i,0) circle (2pt);
}
\node at (0,-0.25) {$\scriptstyle 1$};
\node at (\HypScalex,-0.25) {$\scriptstyle 1$};
\node at (1/2*\HypScalex,-0.25) {$\scriptstyle 2$};
\node at (1/3*\HypScalex,-0.25) {$\scriptstyle 3$};
\node at (2/3*\HypScalex,-0.25) {$\scriptstyle 3$};
\node at (1/4*\HypScalex,-0.25) {$\scriptstyle 4$};
\node at (3/4*\HypScalex,-0.25) {$\scriptstyle 4$};
\node at (1/5*\HypScalex,-0.25) {$\scriptstyle 5$};
\node at (2/5*\HypScalex,-0.25) {$\scriptstyle 5$};
\node at (3/5*\HypScalex,-0.25) {$\scriptstyle 5$};
\node at (4/5*\HypScalex,-0.25) {$\scriptstyle 5$};
\end{tikzpicture}\\
%%%%%%%%%%%%%%%%%%%%%l=6
\begin{tikzpicture}[scale=0.9]
\draw[densely dotted] (-1,0) -- ($(\HypScalex,0)+(1,0)$);
{\foreach \i in {0,1,2}
\filldraw[fill=white,draw=black] (\HypScalex*1/2*\i,0) circle (2pt);
}
{\foreach \i in {1,2}
\filldraw[fill=white,draw=black] (\HypScalex*1/3*\i,0) circle (2pt);
}
{\foreach \i in {1,2,3}
\filldraw[fill=white,draw=black] (\HypScalex*1/4*\i,0) circle (2pt);
}
{\foreach \i in {1,2,3,4}
\filldraw[fill=white,draw=black] (\HypScalex*1/5*\i,0) circle (2pt);
}
{\foreach \i in {1,2,3,4,5}
\filldraw[fill=white,draw=black] (\HypScalex*1/6*\i,0) circle (2pt);
}
\node at (0,-0.25) {$\scriptstyle 1$};
\node at (\HypScalex,-0.25) {$\scriptstyle 1$};
\node at (1/2*\HypScalex,-0.25) {$\scriptstyle 2$};
\node at (1/3*\HypScalex,-0.25) {$\scriptstyle 3$};
\node at (2/3*\HypScalex,-0.25) {$\scriptstyle 3$};
\node at (1/4*\HypScalex,-0.25) {$\scriptstyle 4$};
\node at (3/4*\HypScalex,-0.25) {$\scriptstyle 4$};
\node at (1/5*\HypScalex,-0.25) {$\scriptstyle 5$};
\node at (2/5*\HypScalex,-0.25) {$\scriptstyle 5$};
\node at (3/5*\HypScalex,-0.25) {$\scriptstyle 5$};
\node at (4/5*\HypScalex,-0.25) {$\scriptstyle 5$};
\node at (1/6*\HypScalex,-0.25) {$\scriptstyle 6$};
\node at (5/6*\HypScalex,-0.25) {$\scriptstyle 6$};
\end{tikzpicture}
\end{array}
\]
In each case the hyperplane arrangement is infinite, and the labels repeat.
\end{prop}
\begin{proof}
By Katz--Morrison \cite{KM,Kawa} it is known that for a smooth single-curve flop, the $\scrJ\subset\Delt$ is one of the following cases:
\[
\begin{array}{cccccc}
\begin{array}{c}
\begin{tikzpicture}[scale=0.25]
\node at (1,0) [DW] {};
\end{tikzpicture}
\end{array}
&
\begin{array}{c}
\begin{tikzpicture}[scale=0.25]
\node at (0,0) [DB] {};
\node at (1,0) [DW] {};
\node at (2,0) [DB] {};
\node at (1,1) [DB] {};
\end{tikzpicture}
\end{array}
&
\begin{array}{c}
\begin{tikzpicture}[scale=0.21]
\node at (0,0) [DB] {};
\node at (1,0) [DB] {};
\node at (2,0) [DW] {};
\node at (2,1) [DB] {};
\node at (3,0) [DB] {};
\node at (4,0) [DB] {};
\end{tikzpicture}
\end{array}
&
\begin{array}{c}
\begin{tikzpicture}[scale=0.21]
\node at (0,0) [DB] {};
\node at (1,0) [DB] {};
\node at (2,0) [DW] {};
\node at (2,1) [DB] {};
\node at (5,0) [DB] {};
\node at (3,0) [DB] {};
\node at (4,0) [DB] {};
\end{tikzpicture}
\end{array}
&
\begin{array}{c}
\begin{tikzpicture}[scale=0.21]
\node at (0,0) [DB] {};
\node at (1,0) [DB] {};
\node at (2,0) [DB] {};
\node at (2,1) [DB] {};
\node at (3,0) [DW] {};
\node at (4,0) [DB] {};
\node at (5,0) [DB] {};
\node at (6,0) [DB] {};
\end{tikzpicture}
\end{array}
&
\begin{array}{c}
\begin{tikzpicture}[scale=0.21]
\node at (0,0) [DB] {};
\node at (1,0) [DB] {};
\node at (2,0) [DW] {};
\node at (2,1) [DB] {};
\node at (3,0) [DB] {};
\node at (4,0) [DB] {};
\node at (5,0) [DB] {};
\node at (6,0) [DB] {};
\end{tikzpicture}
\end{array}\\
A_1&D_4&E_6&E_7&E_8(5)&E_8(6)
\end{array}
\]
We analyse each individually.  In each case, by \cite[\S1]{IW9} the $\scrJ$-affine arrangement $\Cone{\scrJ_{\aff}}$ can be calculated by using local wall crossing rules.  Combinatorially, this is very elementary, and is explained in detail in \cite[1.1]{Kinosaki}.  
We sketch the  $D_4$ case here.

As in \eqref{arrangement labelled one curve}, consider the chamber 
\[
\begin{tikzpicture}
\draw[densely dotted] (-5,0) -- (5,0);
%\node at (5.25,0) {$\bR$};
{\foreach \i in {-2,-1,0,1,2}
\filldraw[fill=white,draw=black] (-2*\i,0) circle (2pt);
}
\node at (-4,-0.25) {$\scriptstyle R$};
\node at (-2,-0.25) {$\scriptstyle N_1^*$};
\end{tikzpicture}
\]
We first replace the modules by their ranks, and we label the chamber via McKay correspondence.  The fact we will repeatedly use below is that the rank of a summand equals the number $\updelta_i$, where $i$ is the vertex in the Dynkin diagram that the summand corresponds to, and $\updelta=(\updelta_i)_{i\in\Updelta_{\aff}}$ is the null root \cite[9.3]{IW9}.  Doing this, we obtain

\newcommand{\Dbase}[1][]{%
\begin{tikzpicture}[scale=0.25]
\node at (0,0) [DB] {};
\node at (1,0) [DW] {};
\node at (2,0) [DB] {};
\node at (1,1) [DB] {};
\node at (1,-1) [DW] {};
\end{tikzpicture}
}
\[
\begin{tikzpicture}
\draw[densely dotted] (-5,0) -- (5,0);
%\node at (5.25,0) {$\bR$};
{\foreach \i in {-2,-1,0,1,2}
\filldraw[fill=white,draw=black] (-2*\i,0) circle (2pt);
}
{\foreach \i in {1}
\node at (-2*\i-1,0.5) {\Dbase};
}
\node at (-4,-0.25) {$\scriptstyle 1$};
\node at (-2,-0.25) {$\scriptstyle 2$};
\end{tikzpicture}
\]
To obtain wall crossing over the wall labelled $2$, temporarily delete the vertex corresponding to $2$, apply the Dynkin involution to the \emph{remainder} (which is trivial for $A_1\times A_1\times A_1\times A_1$), then insert back in the vertex labelled $2$. 
\[
\begin{tikzpicture}[scale=0.25]
\node at (-8,0) [DB] {};
\node at (-7,0) [DW] {};
\node at (-6,0) [DB] {};
\node at (-7,1) [DB] {};
\node at (-7,-1) [DW] {};
\draw[->] (-4,0) --node[above]{\scriptsize delete}(-2,0);
\node at (0,0) [DB] {};
\node[densely dotted] at (1,0) [DB] {};
\node at (2,0) [DB] {};
\node at (1,1) [DB] {};
\node at (1,-1) [DW] {};
\draw[->] (4,0) --node[above]{\scriptsize involution}(8,0);
\node at (10,0) [DB] {};
\node[densely dotted] at (11,0) [DB] {};
\node at (12,0) [DB] {};
\node at (11,1) [DB] {};
\node at (11,-1) [DW] {};
\draw[->] (14,0) --node[above]{\scriptsize insert}(16,0);
\node at (18,0) [DB] {};
\node at (19,0) [DW] {};
\node at (20,0) [DB] {};
\node at (19,1) [DB] {};
\node at (19,-1) [DW] {};
\end{tikzpicture}
\]
The wall crossing is thus described by 
\[
\begin{tikzpicture}
\draw[densely dotted] (-5,0) -- (5,0);
{\foreach \i in {-2,-1,0,1,2}
\filldraw[fill=white,draw=black] (-2*\i,0) circle (2pt);
}
{\foreach \i in {1,0}
\node at (-2*\i-1,0.5) {\Dbase};
}
\node at (-4,-0.25) {$\scriptstyle 1$};
\node at (-2,-0.25) {$\scriptstyle 2$};
\node at (0,-0.25) {$\scriptstyle 1$};
\end{tikzpicture}
\]
Applying the same local rule but instead at the other vertex, and repeating, gives
\[
\begin{tikzpicture}
\draw[densely dotted,->] (-5,0) -- (5,0);
\node at (5.25,0) {$\bR$};
{\foreach \i in {-2,-1,0,1,2}
\filldraw[fill=white,draw=black] (-2*\i,0) circle (2pt);
}
{\foreach \i in {-2,-1,0,1,1}
\node at (-2*\i-1,0.5) {\Dbase};
}
\node at (-4,-0.25) {$\scriptstyle 1$};
\node at (-2,-0.25) {$\scriptstyle 2$};
\node at (0,-0.25) {$\scriptstyle 1$};
\node at (2,-0.25) {$\scriptstyle 2$};
\node at (4,-0.25) {$\scriptstyle 1$};
\end{tikzpicture}
\]
The $E_6$ case is explained in detail in \cite[1.1]{Kinosaki}, and is summarised by the following.  The shaded region can be ignored for now, but will be used later in Theorem~\ref{SKMS text}.
\newcommand{\EbaseA}[1][]{%
\begin{tikzpicture}[scale=0.21]
\node at (0,0) [DB] {};
\node at (1,0) [DB] {};
\node at (2,0) [DW] {};
\node at (2,1) [DB] {};
\node at (2,2) [DW] {};
\node at (3,0) [DB] {};
\node at (4,0) [DB] {};
\end{tikzpicture}
}
\newcommand{\EbaseB}[1][]{%
\begin{tikzpicture}[scale=0.21]
\node at (0,0) [DB] {};
\node at (1,0) [DB] {};
\node at (2,0) [DW] {};
\node at (2,1) [DW] {};
\node at (2,2) [DB] {};
\node at (3,0) [DB] {};
\node at (4,0) [DB] {};
\end{tikzpicture}
}
\[
\begin{tikzpicture}
\filldraw[gray!10!white] (-2.5,-0.5) -- (-2.5,1) -- (5.5,1)--(5.5,-0.5) --cycle;
\draw[densely dotted] (-4.5,0) -- (6.5,0);
{\foreach \i in {-3,-2,-1,0,1,2}
\filldraw[fill=white,draw=black] (-2*\i,0) circle (2pt);
}
{\foreach \i in {2,1}
\node at (-2*\i+1,0.5) {\EbaseA};
}
{\foreach \i in {0,-1}
\node at (-2*\i+1,0.5) {\EbaseB};
}
\node at (5,0.5) {\EbaseA};
\node at (-4,-0.25) {$\scriptstyle 3$};
\node at (-2,-0.25) {$\scriptstyle 1$};
\node at (0,-0.25) {$\scriptstyle 3$};
\node at (2,-0.25) {$\scriptstyle 2$};
\node at (4,-0.25) {$\scriptstyle 3$};
\node at (6,-0.25) {$\scriptstyle 1$};
\end{tikzpicture}
\]

For $E_7$, $E_8(5)$ and $E_8(6)$, the calculations are, respectively,
\newcommand{\ESbaseA}[1][]{%
\begin{tikzpicture}[scale=0.21,rotate=90]
\node at (0,0) [DB] {};
\node at (1,0) [DB] {};
\node at (2,0) [DW] {};
\node at (2,1) [DB] {};
\node at (-1,0) [DB] {};
\node at (3,0) [DB] {};
\node at (4,0) [DB] {};
\node at (5,0) [DW] {};
\end{tikzpicture}
}
\newcommand{\ESbaseB}[1][]{%
\begin{tikzpicture}[scale=0.21,rotate=90]
\node at (-1,0) [DB] {};
\node at (0,0) [DB] {};
\node at (1,0) [DB] {};
\node at (2,0) [DW] {};
\node at (2,1) [DB] {};
\node at (3,0) [DW] {};
\node at (4,0) [DB] {};
\node at (5,0) [DB] {};
\end{tikzpicture}
}
\newcommand{\ESbaseC}[1][]{%
\begin{tikzpicture}[scale=0.21,rotate=90]
\node at (-1,0) [DB] {};
\node at (0,0) [DW] {};
\node at (1,0) [DB] {};
\node at (2,0) [DB] {};
\node at (2,1) [DB] {};
\node at (3,0) [DW] {};
\node at (4,0) [DB] {};
\node at (5,0) [DB] {};
\end{tikzpicture}
}

\[
\begin{tikzpicture}
\filldraw[gray!10!white] (-3,-0.5) -- (-3,2) -- (6,2)--(6,-0.5) --cycle;
\draw[densely dotted] (-5,0) -- (7.5,0);
%\node at (7.75,0) {$\bR$};
{\foreach \i in {-3,-2,-1,0,1,2,3,4}
\filldraw[fill=white,draw=black] (-1.5*\i+2,0) circle (2pt);
}
\node at (-3.25,1) {\ESbaseA};
\node at (-1.75,1) {\ESbaseA};
\node at (-0.25,1) {\ESbaseB};
\node at (1.25,1) {\ESbaseC};
\node at (2.75,1) {\ESbaseC};
\node at (4.25,1) {\ESbaseB};
\node at (5.75,1) {\ESbaseA};
\node at (-4,-0.25) {$\scriptstyle 4$};
\node at (-2.5,-0.25) {$\scriptstyle 1$};
\node at (-1,-0.25) {$\scriptstyle 4$};
\node at (0.5,-0.25) {$\scriptstyle 3$};
\node at (2,-0.25) {$\scriptstyle 2$};
\node at (3.5,-0.25) {$\scriptstyle 3$};
\node at (5,-0.25) {$\scriptstyle 4$};
\node at (6.5,-0.25) {$\scriptstyle 1$};
\end{tikzpicture}
\]

\newcommand{\ESSbaseA}[1][]{%
\begin{tikzpicture}[scale=0.21,rotate=90]
\node at (0,0) [DB] {};
\node at (1,0) [DB] {};
\node at (2,0) [DB] {};
\node at (2,1) [DB] {};
\node at (3,0) [DW] {};
\node at (4,0) [DB] {};
\node at (5,0) [DB] {};
\node at (6,0) [DB] {};
\node at (7,0) [DW] {};
\end{tikzpicture}
}
\newcommand{\ESSbaseB}[1][]{%
\begin{tikzpicture}[scale=0.21,rotate=90]
\node at (0,0) [DB] {};
\node at (1,0) [DB] {};
\node at (2,0) [DB] {};
\node at (2,1) [DB] {};
\node at (3,0) [DW] {};
\node at (4,0) [DW] {};
\node at (5,0) [DB] {};
\node at (6,0) [DB] {};
\node at (7,0) [DB] {};
\end{tikzpicture}
}
\newcommand{\ESSbaseC}[1][]{%
\begin{tikzpicture}[scale=0.21,rotate=90]
\node at (0,0) [DB] {};
\node at (1,0) [DB] {};
\node at (2,0) [DB] {};
\node at (2,1) [DW] {};
\node at (3,0) [DB] {};
\node at (4,0) [DW] {};
\node at (5,0) [DB] {};
\node at (6,0) [DB] {};
\node at (7,0) [DB] {};
\end{tikzpicture}
}
\newcommand{\ESSbaseD}[1][]{%
\begin{tikzpicture}[scale=0.21,rotate=90]
\node at (0,0) [DB] {};
\node at (1,0) [DB] {};
\node at (2,0) [DB] {};
\node at (2,1) [DW] {};
\node at (3,0) [DW] {};
\node at (4,0) [DB] {};
\node at (5,0) [DB] {};
\node at (6,0) [DB] {};
\node at (7,0) [DB] {};
\end{tikzpicture}
}
\newcommand{\ESSbaseE}[1][]{%
\begin{tikzpicture}[scale=0.21,rotate=90]
\node at (0,0) [DW] {};
\node at (1,0) [DB] {};
\node at (2,0) [DB] {};
\node at (2,1) [DB] {};
\node at (3,0) [DW] {};
\node at (4,0) [DB] {};
\node at (5,0) [DB] {};
\node at (6,0) [DB] {};
\node at (7,0) [DB] {};
\end{tikzpicture}
}
\[
\begin{tikzpicture}
\filldraw[gray!10!white] (-3.25,-0.5) -- (-3.25,2) -- (6.75,2)--(6.75,-0.5) --cycle;
\draw[densely dotted] (-4.5,0) -- (7.5,0);
%\draw[densely dotted,->] (-4.5,0) -- (7.5,0);
%\node at (7.75,0) {$\bR$};
{\foreach \i in {-5,-4,-3,-2,-1,0,1,2,3,4,5,6}
\filldraw[fill=white,draw=black] (-1*\i+2,0) circle (2pt);
}
\node at (-3.5,1) {\ESSbaseA};
\node at (-2.5,1) {\ESSbaseA};
\node at (-1.5,1) {\ESSbaseB};
\node at (-0.5,1) {\ESSbaseC};
\node at (0.5,1) {\ESSbaseD};
\node at (1.5,1) {\ESSbaseE};
\node at (2.5,1) {\ESSbaseE};
\node at (3.5,1) {\ESSbaseD};
\node at (4.5,1) {\ESSbaseC};
\node at (5.5,1) {\ESSbaseB};
\node at (6.5,1) {\ESSbaseA};
\node at (-4,-0.25) {$\scriptstyle 5$};
\node at (-3,-0.25) {$\scriptstyle 1$};
\node at (-2,-0.25) {$\scriptstyle 5$};
\node at (-1,-0.25) {$\scriptstyle 4$};
\node at (0,-0.25) {$\scriptstyle 3$};
\node at (1,-0.25) {$\scriptstyle 5$};
\node at (2,-0.25) {$\scriptstyle 2$};
\node at (3,-0.25) {$\scriptstyle 5$};
\node at (4,-0.25) {$\scriptstyle 3$};
\node at (5,-0.25) {$\scriptstyle 4$};
\node at (6,-0.25) {$\scriptstyle 5$};
\node at (7,-0.25) {$\scriptstyle 1$};
\end{tikzpicture}
\]

\newcommand{\ESSSbaseA}[1][]{%
\begin{tikzpicture}[scale=0.21,rotate=90]
\node at (0,0) [DB] {};
\node at (1,0) [DB] {};
\node at (2,0) [DW] {};
\node at (2,1) [DB] {};
\node at (3,0) [DB] {};
\node at (4,0) [DB] {};
\node at (5,0) [DB] {};
\node at (6,0) [DB] {};
\node at (7,0) [DW] {};
\end{tikzpicture}
}
\newcommand{\ESSSbaseB}[1][]{%
\begin{tikzpicture}[scale=0.21,rotate=90]
\node at (0,0) [DB] {};
\node at (1,0) [DB] {};
\node at (2,0) [DW] {};
\node at (2,1) [DB] {};
\node at (3,0) [DW] {};
\node at (4,0) [DB] {};
\node at (5,0) [DB] {};
\node at (6,0) [DB] {};
\node at (7,0) [DB] {};
\end{tikzpicture}
}
\newcommand{\ESSSbaseC}[1][]{%
\begin{tikzpicture}[scale=0.21,rotate=90]
\node at (0,0) [DB] {};
\node at (1,0) [DW] {};
\node at (2,0) [DB] {};
\node at (2,1) [DB] {};
\node at (3,0) [DW] {};
\node at (4,0) [DB] {};
\node at (5,0) [DB] {};
\node at (6,0) [DB] {};
\node at (7,0) [DB] {};
\end{tikzpicture}
}
\newcommand{\ESSSbaseD}[1][]{%
\begin{tikzpicture}[scale=0.21,rotate=90]
\node at (0,0) [DB] {};
\node at (1,0) [DW] {};
\node at (2,0) [DB] {};
\node at (2,1) [DB] {};
\node at (3,0) [DB] {};
\node at (4,0) [DB] {};
\node at (5,0) [DW] {};
\node at (6,0) [DB] {};
\node at (7,0) [DB] {};
\end{tikzpicture}
}
\newcommand{\ESSSbaseE}[1][]{%
\begin{tikzpicture}[scale=0.21,rotate=90]
\node at (0,0) [DB] {};
\node at (1,0) [DB] {};
\node at (2,0) [DB] {};
\node at (2,1) [DB] {};
\node at (3,0) [DW] {};
\node at (4,0) [DB] {};
\node at (5,0) [DW] {};
\node at (6,0) [DB] {};
\node at (7,0) [DB] {};
\end{tikzpicture}
}
\newcommand{\ESSSbaseF}[1][]{%
\begin{tikzpicture}[scale=0.21,rotate=90]
\node at (0,0) [DB] {};
\node at (1,0) [DB] {};
\node at (2,0) [DB] {};
\node at (2,1) [DB] {};
\node at (3,0) [DW] {};
\node at (4,0) [DB] {};
\node at (5,0) [DB] {};
\node at (6,0) [DW] {};
\node at (7,0) [DB] {};
\end{tikzpicture}
}
\[
\begin{tikzpicture}[scale=0.9]
\filldraw[gray!10!white] (-5.25,-0.5) -- (-5.25,2.25) -- (6.75,2.25)--(6.75,-0.5) --cycle;
\draw[densely dotted] (-6.5,0) -- (7.5,0);
%\node at (7.75,0) {$\bR$};
{\foreach \i in {-5,-4,-3,-2,-1,0,1,2,3,4,5,6,7,8}
\filldraw[fill=white,draw=black] (-1*\i+2,0) circle (2pt);
}
\node at (-5.5,1.1) {\ESSSbaseA};
\node at (-4.5,1.1) {\ESSSbaseA};
\node at (-3.5,1.1) {\ESSSbaseB};
\node at (-2.5,1.1) {\ESSSbaseC};
\node at (-1.5,1.1) {\ESSSbaseD};
\node at (-0.5,1.1) {\ESSSbaseE};
\node at (0.5,1.1) {\ESSSbaseF};
\node at (1.5,1.1) {\ESSSbaseF};
\node at (2.5,1.1) {\ESSSbaseE};
\node at (3.5,1.1) {\ESSSbaseD};
\node at (4.5,1.1) {\ESSSbaseC};
\node at (5.5,1.1) {\ESSSbaseB};
\node at (6.5,1.1) {\ESSSbaseA};
\node at (-6,-0.25) {$\scriptstyle 6$};
\node at (-5,-0.25) {$\scriptstyle 1$};
\node at (-4,-0.25) {$\scriptstyle 6$};
\node at (-3,-0.25) {$\scriptstyle 5$};
\node at (-2,-0.25) {$\scriptstyle 4$};
\node at (-1,-0.25) {$\scriptstyle 3$};
\node at (0,-0.25) {$\scriptstyle 5$};
\node at (1,-0.25) {$\scriptstyle 2$};
\node at (2,-0.25) {$\scriptstyle 5$};
\node at (3,-0.25) {$\scriptstyle 3$};
\node at (4,-0.25) {$\scriptstyle 4$};
\node at (5,-0.25) {$\scriptstyle 5$};
\node at (6,-0.25) {$\scriptstyle 6$};
\node at (7,-0.25) {$\scriptstyle 1$};
\end{tikzpicture}\qedhere
\]
\end{proof}

The following extends the pictures in  \cite[p6169]{T08} and \cite[Figure 1]{Aspinwall}, which describe the case $\ell=1$, to all higher lengths.

\begin{thm}\label{SKMS text}
Suppose that $X\to\Spec R$ is a length $\ell$ irreducible flop, where $X$ is smooth. Then the SKMS is one of the following:
\[
\begin{array}{cccc}
\begin{array}{c}
\begin{tikzpicture}
%equator around the back
\draw[densely dotted,color=gray] (1,0) arc (0:180:1cm and 0.2cm);
%equator around the front
\draw[densely dotted] (1,0) arc (0:-180:1cm and 0.2cm)
coordinate[pos=0.5] (A);
%black circle for the sphere
\draw ([shift=(-84:1cm)]0,0) arc (-84:84:1cm)  
[bend left] to (96:1cm)
arc (96:264:1cm)
[bend left] to cycle;
%dots at top and bottom of sphere
\draw[densely dotted] ([shift=(86.75:1cm)]0,0) arc  (86.75:94:1cm);
\draw[densely dotted] ([shift=(-86.75:1cm)]0,0) arc  (-86.75:-94:1cm);
%false equator of labels of dots
\draw[draw=none] (1,-0.2) arc (0:-180:1cm and 0.2cm)
coordinate[pos=0.5] (a);
%dots
\filldraw[fill=white,draw=black] (A) circle (1.5pt);
%labels
\node at (a) {$\scriptstyle 1$};
\end{tikzpicture}
\end{array}
&
\begin{array}{c}
\begin{tikzpicture}
%equator around the back
\draw[gray,densely dotted] (1,0) arc (0:180:1cm and 0.2cm);
%equator around the front
\draw[densely dotted] (1,0) arc (0:-180:1cm and 0.2cm)
coordinate[pos=0.65] (A) coordinate[pos=0.35] (B);
%black circle for the sphere
\draw ([shift=(-84:1cm)]0,0) arc (-84:84:1cm)  
[bend left] to (96:1cm)
arc (96:264:1cm)
[bend left] to cycle;
%dots at top and bottom of sphere
\draw[densely dotted] ([shift=(86.75:1cm)]0,0) arc  (86.75:94:1cm);
\draw[densely dotted] ([shift=(-86.75:1cm)]0,0) arc  (-86.75:-94:1cm);
%false equator of labels of dots
\draw[draw=none] (1,-0.2) arc (0:-180:1cm and 0.2cm)
coordinate[pos=0.65] (a) coordinate[pos=0.35] (b);
%dots
\filldraw[fill=white,draw=black] (A) circle (1.5pt);
\filldraw[fill=white,draw=black] (B) circle (1.5pt);
%labels
\node at (a) {$\scriptstyle 1$};
\node at (b) {$\scriptstyle 2$};
\end{tikzpicture}
\end{array}
&
\begin{array}{c}
\begin{tikzpicture}
%equator around the back
\draw[gray,densely dotted] (1,0) arc (0:180:1cm and 0.2cm);
%equator around the front
\draw[densely dotted] (1,0) arc (0:-180:1cm and 0.2cm)
coordinate[pos=0.71] (A) coordinate[pos=0.57] (B) coordinate[pos=0.43] (C) coordinate[pos=0.29] (D);
%black circle for the sphere
\draw ([shift=(-84:1cm)]0,0) arc (-84:84:1cm)  
[bend left] to (96:1cm)
arc (96:264:1cm)
[bend left] to cycle; 
%dots at top and bottom of sphere
\draw[densely dotted] ([shift=(86.75:1cm)]0,0) arc  (86.75:94:1cm);
\draw[densely dotted] ([shift=(-86.75:1cm)]0,0) arc  (-86.75:-94:1cm);
%false equator of labels of dots
\draw[draw=none] (1,-0.2) arc (0:-180:1cm and 0.2cm)
coordinate[pos=0.71] (a) coordinate[pos=0.57] (b) coordinate[pos=0.43] (c) coordinate[pos=0.29] (d);
%dots
\filldraw[fill=white,draw=black] (A) circle (1.5pt);
\filldraw[fill=white,draw=black] (B) circle (1.5pt);
\filldraw[fill=white,draw=black] (C) circle (1.5pt);
\filldraw[fill=white,draw=black] (D) circle (1.5pt);
%labels
\node at (a) {$\scriptstyle 1$};
\node at (b) {$\scriptstyle 3$};
\node at (c) {$\scriptstyle 2$};
\node at (d) {$\scriptstyle 3$};
\end{tikzpicture}
\end{array}
&
\begin{array}{c}
\begin{tikzpicture}
%equator around the back
\draw[gray,densely dotted] (1,0) arc (0:180:1cm and 0.2cm);
%equator around the front
\draw[densely dotted] (1,0) arc (0:-180:1cm and 0.2cm)
coordinate[pos=0.8] (A) coordinate[pos=0.67] (B) coordinate[pos=0.56] (C) coordinate[pos=0.45] (D) coordinate[pos=0.33] (E) coordinate[pos=0.2] (F);
%black circle for the sphere
\draw ([shift=(-84:1cm)]0,0) arc (-84:84:1cm)  
[bend left] to (96:1cm)
arc (96:264:1cm)
[bend left] to cycle; 
%dots at top and bottom of sphere
\draw[densely dotted] ([shift=(86.75:1cm)]0,0) arc  (86.75:94:1cm);
\draw[densely dotted] ([shift=(-86.75:1cm)]0,0) arc  (-86.75:-94:1cm);
%false equator of labels of dots
\draw[draw=none] (1,-0.2) arc (0:-180:1cm and 0.2cm)
coordinate[pos=0.8] (a) coordinate[pos=0.67] (b) coordinate[pos=0.56] (c) coordinate[pos=0.45] (d) coordinate[pos=0.33] (e) coordinate[pos=0.2] (f);
%dots
\filldraw[fill=white,draw=black] (A) circle (2pt);
\filldraw[fill=white,draw=black] (B) circle (2pt);
\filldraw[fill=white,draw=black] (C) circle (2pt);
\filldraw[fill=white,draw=black] (D) circle (2pt);
\filldraw[fill=white,draw=black] (E) circle (2pt);
\filldraw[fill=white,draw=black] (F) circle (2pt);
%labels
\node at (a) {$\scriptstyle 1$};
\node at (b) {$\scriptstyle 4$};
\node at (c) {$\scriptstyle 3$};
\node at (d) {$\scriptstyle 2$};
\node at (e) {$\scriptstyle 3$};
\node at (f) {$\scriptstyle 4$};
\end{tikzpicture}
\end{array}\\
\ell=1&\ell=2&\ell=3&\ell=4
\end{array}
\]
The cases $\ell=5,6$ behave slightly differently; respectively they are:
\[
\begin{array}{cc}
\begin{array}{c}
\begin{tikzpicture}[scale=1.5]
%equator around the back
\draw[gray,densely dotted] (1,0) arc (0:180:1cm and 0.2cm);
%equator around the front
\draw[densely dotted] (1,0) arc (0:-180:1cm and 0.2cm)
coordinate[pos=0.85] (A) coordinate[pos=0.75] (B) coordinate[pos=0.67] (C) coordinate[pos=0.6] (D) coordinate[pos=0.53] (E) coordinate[pos=0.465] (F)
coordinate[pos=0.4] (G) coordinate[pos=0.33] (H) coordinate[pos=0.25] (I) coordinate[pos=0.15] (J);
%black circle for the sphere
\draw ([shift=(-84:1cm)]0,0) arc (-84:84:1cm)  
[bend left] to (96:1cm)
arc (96:264:1cm)
[bend left] to cycle; 
%dots at top and bottom of sphere
\draw[densely dotted] ([shift=(86.75:1cm)]0,0) arc  (86.75:94:1cm);
\draw[densely dotted] ([shift=(-86.75:1cm)]0,0) arc  (-86.75:-94:1cm);
%false equator of labels of dots
\draw[draw=none] (1,-0.15) arc (0:-180:1cm and 0.2cm)
coordinate[pos=0.85] (a) coordinate[pos=0.75] (b) coordinate[pos=0.67] (c) coordinate[pos=0.6] (d) coordinate[pos=0.53] (e) coordinate[pos=0.465] (f)
coordinate[pos=0.4] (g) coordinate[pos=0.33] (h) coordinate[pos=0.25] (i) coordinate[pos=0.15] (j);
%dots
\filldraw[fill=white,draw=black] (A) circle (1.5pt);
\filldraw[fill=white,draw=black] (B) circle (1.5pt);
\filldraw[fill=white,draw=black] (C) circle (1.5pt);
\filldraw[fill=white,draw=black] (D) circle (1.5pt);
\filldraw[fill=white,draw=black] (E) circle (1.5pt);
\filldraw[fill=white,draw=black] (F) circle (1.5pt);
\filldraw[fill=white,draw=black] (G) circle (1.5pt);
\filldraw[fill=white,draw=black] (H) circle (1.5pt);
\filldraw[fill=white,draw=black] (I) circle (1.5pt);
\filldraw[fill=white,draw=black] (J) circle (1.5pt);
%labels
\node at (a) {$\scriptstyle 1$};
\node at (b) {$\scriptstyle 5$};
\node at (c) {$\scriptstyle 4$};
\node at (d) {$\scriptstyle 3$};
\node at (e) {$\scriptstyle 5$};
\node at (f) {$\scriptstyle 2$};
\node at (g) {$\scriptstyle 5$};
\node at (h) {$\scriptstyle 3$};
\node at (i) {$\scriptstyle 4$};
\node at (j) {$\scriptstyle 5$};
\end{tikzpicture}
\end{array}
&
\begin{array}{c}
\begin{tikzpicture}[scale=1.5]
%equator around the back
\draw[gray,densely dotted] (1,0) arc (0:180:1cm and 0.2cm);
%equator around the front
\draw[densely dotted] (1,0) arc (0:-180:1cm and 0.2cm)
coordinate[pos=0.85] (A) coordinate[pos=0.76] (B) coordinate[pos=0.69] (C) coordinate[pos=0.63] (D) coordinate[pos=0.575] (E) coordinate[pos=0.525] (F)
coordinate[pos=0.475] (G) coordinate[pos=0.425] (H) coordinate[pos=0.37] (I) coordinate[pos=0.31] (J)
coordinate[pos=0.24] (K) coordinate[pos=0.15] (L);
%black circle for the sphere
\draw ([shift=(-84:1cm)]0,0) arc (-84:84:1cm)  
[bend left] to (96:1cm)
arc (96:264:1cm)
[bend left] to cycle; 
%dots at top and bottom of sphere
\draw[densely dotted] ([shift=(86.75:1cm)]0,0) arc  (86.75:94:1cm);
\draw[densely dotted] ([shift=(-86.75:1cm)]0,0) arc  (-86.75:-94:1cm);
%false equator of labels of dots
\draw[draw=none] (1,-0.15) arc (0:-180:1cm and 0.2cm)
coordinate[pos=0.85] (a) coordinate[pos=0.76] (b) coordinate[pos=0.69] (c) coordinate[pos=0.63] (d) coordinate[pos=0.575] (e) coordinate[pos=0.525] (f)
coordinate[pos=0.475] (g) coordinate[pos=0.425] (h) coordinate[pos=0.37] (i) coordinate[pos=0.31] (j)
coordinate[pos=0.24] (k) coordinate[pos=0.15] (l);
%dots
\filldraw[fill=white,draw=black] (A) circle (1.5pt);
\filldraw[fill=white,draw=black] (B) circle (1.5pt);
\filldraw[fill=white,draw=black] (C) circle (1.5pt);
\filldraw[fill=white,draw=black] (D) circle (1.5pt);
\filldraw[fill=white,draw=black] (E) circle (1.5pt);
\filldraw[fill=white,draw=black] (F) circle (1.5pt);
\filldraw[fill=white,draw=black] (G) circle (1.5pt);
\filldraw[fill=white,draw=black] (H) circle (1.5pt);
\filldraw[fill=white,draw=black] (I) circle (1.5pt);
\filldraw[fill=white,draw=black] (J) circle (1.5pt);
\filldraw[fill=white,draw=black] (K) circle (1.5pt);
\filldraw[fill=white,draw=black] (L) circle (1.5pt);
%%labels
\node at (a) {$\scriptstyle 1$};
\node at (b) {$\scriptstyle 6$};
\node at (c) {$\scriptstyle 5$};
\node at (d) {$\scriptstyle 4$};
\node at (e) {$\scriptstyle 3$};
\node at (f) {$\scriptstyle 5$};
\node at (g) {$\scriptstyle 2$};
\node at (h) {$\scriptstyle 5$};
\node at (i) {$\scriptstyle 3$};
\node at (j) {$\scriptstyle 4$};
\node at (k) {$\scriptstyle 5$};
\node at (l) {$\scriptstyle 6$};
\end{tikzpicture}
\end{array}\\
\ell=5&\ell=6
\end{array}
\]
\end{thm}
\begin{proof}
By Theorem~\ref{semidirect main} and the resulting exact sequence \eqref{split ses groups}, we first quotient by $\Br\scrD$, then quotient by $\Pic X$.  By Theorem~\ref{regular covering}, it suffices to identify $(\mathbb{C}^n\backslash \scrH^{\aff}_{\bC})/\Pic X$.

But this is \S\ref{Pic action}, see \eqref{Pic for ell 3}.  Indeed, the action of $\scrO(1)$ moves chambers to the left, by either $1,2,4,6,10,$ or $12$ steps, depending on the length of the flopping curve.  Thus, for example in the case $\ell=2$,  the generator $\scrO(-1)$ of $\Pic X$ acts via
\[
\begin{tikzpicture}
\filldraw[gray!10!white] (-0.5,1) -- (1.5,1) -- (1.5,-1)--(-0.5,-1) --cycle;
\draw[blue,densely dotted,->] (0,-1) -- (0,1);
\draw[blue,densely dotted,<-] (-2.5,0) -- (2.5,0);
{\foreach \i in {-2,-1,0,1,2}
\filldraw[fill=white,draw=black] (-1*\i,0) circle (2pt);
}
\draw[thick,->,yshift=1em] (-0.5,0) -- (1.5,0);
\draw[thick,->,yshift=-1em] (-0.5,0) -- (1.5,0);
\draw (-0.5,-1) -- (-0.5,1);
\draw[dotted] (1.5,-1) -- (1.5,1);
\end{tikzpicture}
\]
where we have shaded the fundamental domain. Thus, identifying edges to form a cylinder, the quotient space is
\[
\begin{array}{c}
\begin{tikzpicture}[scale=0.85]
%\draw (0,0) circle (2cm);
\draw (1,-1) -- (1,1);
\draw (-1,-1) -- (-1,1);
\draw (0,1) ellipse (1cm and 0.2cm);
%\draw (0,-1) ellipse (1cm and 0.2cm);
\draw[densely dotted] (1,-1) arc (0:180:1cm and 0.2cm);
\draw (1,-1) arc (0:-180:1cm and 0.2cm);
\filldraw[fill=white,draw=black] (-0.5,0) circle (2pt);
\filldraw[fill=white,draw=black] (0.5,0) circle (2pt);
\node at (-0.5,-0.25) {$\scriptstyle 1$};
\node at (0.5,-0.25) {$\scriptstyle 2$};
\end{tikzpicture}
\end{array}
\sim
\begin{array}{c}
\begin{tikzpicture}
%\draw (0,0) circle (1cm);
\draw ([shift=(-84:1cm)]0,0) arc (-84:84:1cm)  
[bend left] to (96:1cm)
arc (96:264:1cm)
[bend left] to cycle;
%equator around the back
\draw[densely dotted] (1,0) arc (0:180:1cm and 0.2cm);
%equator around the front
\draw[densely dotted] (1,0) arc (0:-180:1cm and 0.2cm)
coordinate[pos=0.65] (A) coordinate[pos=0.35] (B);
%dots at top and bottom of sphere
\draw[densely dotted] ([shift=(86.75:1cm)]0,0) arc  (86.75:94:1cm);
\draw[densely dotted] ([shift=(-86.75:1cm)]0,0) arc  (-86.75:-94:1cm);
%false equator of labels of dots
\draw[draw=none] (1,-0.2) arc (0:-180:1cm and 0.2cm)
coordinate[pos=0.65] (a) coordinate[pos=0.35] (b);
%dots
\filldraw[fill=white,draw=black] (A) circle (1.5pt);
\filldraw[fill=white,draw=black] (B) circle (1.5pt);
%labels
\node at (a) {$\scriptstyle 1$};
\node at (b) {$\scriptstyle 2$};
\end{tikzpicture}
\end{array}
\]
All other cases are similar, by identifying the left and right hand sides of the fundamental regions shaded in the proof of Proposition~\ref{affine hyper calc}.
\end{proof}

\appendix
\section{Combinatorial Tracking Results}\label{comb appendix}

In this appendix, which is independent of the rest of the paper, we give the proof of Propositions~\ref{complexified tracking} and \ref{complexified tracking 2}, and Lemma~\ref{path connected}.  The Propositions are proved in \S\ref{proof of Props section}, whilst Lemma~\ref{path connected} appears as Lemma~\ref{path connected app}. Throughout, we use the notation from Section~\ref{hyperplane section}. 

\subsection{Preliminary Results}
First, for $L\in\Mut_0(N)$ write $z\in(\Uptheta_L)_{\mathbb{C}}$ as
\[
z=(x_1,\hdots,x_n) + (y_1,\hdots,y_n)\ii \in\bR^n_x + \bR^n_y\,\ii.
\]
The action of $\upvarphi_L$ on K-theory is induced from $\bZ$, so it independently acts on both factors $\bR^n_x$ and $\bR^n_y$.  We will write $\scrH_{x}$ for the hyperplanes $\scrH$ viewed in $\bR^n_x$, and $\scrH_{y}$ for the hyperplanes $\scrH$ viewed in $\bR^n_y$.  Likewise for $H\in\scrH$, we will write $H_x\in\scrH_x$ and $H_y\in\scrH_y$ accordingly. Since $\scrH_{\bC}\colonequals \{ H\times H\mid H\in \scrH\}$, by definition
\begin{equation}
z=x+y\ii\in\scrH_{\bC}\iff  \exists\, H\in\scrH\mbox{ such that }x\in H_{x}\mbox{ and }y\in H_{y}.\label{both factors}
\end{equation}
On the other hand, for $L\in\Mut(N)$ write $z\in (\scrK_L)_{\mathbb{C}}$ as
\[
z=(x_0,x_1,\hdots,x_n) + (y_0,y_1,\hdots,y_n)\ii \in\bR^{n+1}_x + \bR^{n+1}_y\,\ii.
\]
The action of $\upphi_L$ again acts independently on both factors.  As in \S\ref{complexified actions subsection}, consider $\scrW$, the set of full hyperplanes in $\scrK_L\otimes{\bR}$ that separate the open chambers of $\Cone{\scrJ_{\aff}}$.

\begin{exa}\label{affine hypes picture}
In Example~\ref{cones and level example}, $\scrW$ is the following infinite collection of hyperplanes in $\mathbb{R}^2$
\[
\begin{tikzpicture}[very thin,scale=1.5,>=stealth]
\draw (-3,0) -- (3,0);
\draw (0,-1) -- (0,1);
\draw (0,0) -- (0,1);
\draw (0,0) -- (3,0);
{\foreach \i [evaluate=\i as \j using \i+1] in {1,2,3,4,5,6,7,8,9,10,11,12,13,14,15,16,17,18,19,20}
\draw (3*\i/\j,-1) -- (-3*\i/\j,1);
}
{\foreach \i [evaluate=\i as \j using \i+1] in {1,2,3,4,5,6,7,8,9,10,11,12,13,14,15,16,17,18,19,20}
\draw (-3,\i/\j) -- (3,-\i/\j);
}
\end{tikzpicture}
\] 
The hyperplanes converge on the line $\upvartheta_0+3\upvartheta_1=0$, but $\scrW$ does not contain this line.
\end{exa}

Write $\scrW_x$ for the hyperplanes $\scrW$ viewed in $\bR^{n+1}_x$, and $\scrW_{y}$ for the hyperplanes $\scrW$ viewed in $\bR^{n+1}_y$.  Mirroring the above notation, for $W\in\scrW$, similarly consider $W_x\in\scrW_x$ and $W_y\in\scrW_y$.  Again, by definition
\begin{equation}
z=x+y\ii\in\scrW_{\bC}\iff  \exists\, W\in\scrW\mbox{ such that }x\in W_{x}\mbox{ and }y\in W_{y}.\label{both factors 2}
\end{equation}

The following is clear.
\begin{lem}\label{hyper form}
With notation as above, the following statements hold.
\begin{enumerate}
\item\label{hyper form 1} The hyperplanes in $\scrH$ contain the coordinate axes, and are all of the form $\uplambda_1x_{i_1}+\hdots+\uplambda_sx_{i_s}=0$ where each $i_j\in\{1,\hdots,n\}$ and each $\uplambda_j>0$.
\item\label{hyper form 2} The hyperplanes in $\scrW$ contain the coordinate axes, and are all of the form $\uplambda_1x_{i_1}+\hdots+\uplambda_sx_{i_s}=0$ where each $i_j\in\{0,1,\hdots,n\}$ and each $\uplambda_j>0$.
\end{enumerate}
\end{lem}
\begin{proof}
By definition, in both cases $C_+$ is a chamber. Since the coordinate axes bound this, both first statements follow. The second statements follow from the fact that $C_+$ is a chamber, together with the observation that if some $\uplambda_j<0$, then the hyperplane would pass through $C_+$.
\end{proof}

The following is then immediate, and establishes $\supseteq$ in Proposition~\ref{complexified tracking}.
\begin{cor}\label{subset 1}
$\bigcup\upvarphi_L(\bH_+)\subseteq \Uptheta_\bC\backslash \scrH_\bC$
\end{cor}
\begin{proof}
By \eqref{both factors} it is clear that $\upvarphi_L$ restricts to a bijection between $(\scrH_L)_\bC$ and $\scrH_\bC$, since we already know that it does this on both factors.  Hence it suffices to show that $\bH_+\subseteq \bC^n\backslash \scrH_\bC$.

For this, consider $z=x+y\ii\in\bH_+$.  If all $y_i>0$, then by Lemma~\ref{hyper form}\eqref{hyper form 1}, all $\uplambda_1y_{i_1}+\hdots+\uplambda_sy_{i_s}>0$, so $y\notin\scrH_{y}$, and hence $z\notin\scrH_\bC$ by \eqref{both factors}.  By permuting the numbering if necessary, we can thus assume that $y_1=\hdots=y_t=0$ for some $1\leq t\leq n$, and that $y_{t+1},\hdots,y_n>0$.  In this case, by the positivity of $y_{t+1},\hdots,y_n$, and the fact the rest are zero, again using Lemma~\ref{hyper form}\eqref{hyper form 1}  it follows that $y$ avoids all members of $\scrH_y$ except those hyperplanes of the form
\[
\uplambda_1y_{i_1}+\hdots+\uplambda_sy_{i_s}=0
\]
where $i_1,\hdots,i_s\in\{1,\hdots,t\}$.  But since $z\in\bH_+$, the fact that $y_1=\hdots=y_t=0$ forces $x_1,\hdots,x_t<0$.  Hence $x$ avoids all the corresponding members
\[
\uplambda_1x_{i_1}+\hdots+\uplambda_sx_{i_s}=0
\]
of $\scrH_x$. Thus, overall, $y$ avoids some hyperplanes, and the hyperplanes that it does not avoid are avoided by $x$.  Again by \eqref{both factors} it follows that $z\notin\scrH_\bC$.
 \end{proof}

For the affine version of the above, recall that 
\[
\bE_+\colonequals \biggr\{ z\in \bH_+'  \Bigm| \sum_{j=0}^n(\rk_RL_j) z_j=\ii \biggr\}
=(\Level_L)_\mathbb{C}\cap\bH'_+.
\]
The following is elementary. 

\begin{lem}[\ref{path connected}]\label{path connected app}
The subspace $\bE_+\subset (\scrK_L)_{\bC}$ is path connected.
\end{lem}
\begin{proof}
Set $\uplambda_i\colonequals \rk_RL_i$, and consider
\[
\bE_+^\circ\colonequals \biggr\{z\in(\scrK_L)_\mathbb{C} \Bigm| \sum_{j=0}^n\uplambda_jz_j=\ii\mbox{ and }\mathrm{Im}(z_j)>0\mbox{ for all }j \biggr\}=\bH_+'^{\circ}\cap(\Level_L)_{\mathbb{C}}.
\]
This is visibly path connected.  To prove the statement, it suffices to show that for every boundary point $z\in\bE_+\backslash \bE_+^\circ$, there is a path in $\bE_+$ from $z$ to a point $w$ in $\bE_+^\circ$.

Write $z=x+y\ii $, then since  $z\in\bE_+$, not all $y_j$ can be zero.  Further, since $z\in\bE_+\backslash \bE_+^\circ$, after reordering if necessary we can write
\[
y=(\underbrace{0,\hdots,0}_{k}, \underbrace{y_{k},\hdots, y_n}_{>0})\in\bH_+^\circ\cap(\Level_L)_{\mathbb{C}}.
\]
for some $k$ such that $0<k< n$.  Set $\upgamma\colonequals\sum_{i=0}^k\uplambda_i$  and fix $s$ such that  
$0<s<\frac{\uplambda_{n}y_{n}}{\upgamma}$. Then for any $t\in [0,s]$, consider the point 
\[
y(t)\colonequals (\underbrace{t,\hdots,t}_{k}, y_{k},\hdots, y_{n-1}, y_{n}-\tfrac{\upgamma t}{\uplambda_{n}})\in\scrK_L.
\]
This satisfies $\sum_{i=0}^n\uplambda_iy(t)_i=\sum_{i=k}^n\uplambda_iy_i$, which equals $1$ since $z\in\bE_+$.  Setting $z(t)=x+ y(t)\ii$, it is then clear that $z(t)\in(\Level_L)_\mathbb{C}$ for all $t\in[0,s]$.  On the other hand, by inspection $z(t)\in\bH'_+$ for all $t\in[0,s]$. It follows that $z(t)\in\bE_+$ for all $t\in[0,s]$, so setting $w\colonequals x+ y(s)\ii\in \bE_+^\circ$, the path $p\colon [0,s]\to \bE_+$ sending $t\mapsto z(t)$ connects $z$ and $w$.
\end{proof}

Recall that $\scrH^{\aff}_{\mathbb{C}}= 
\scrW_{\mathbb{C}}\cap\Level_{\bC}$. The following establishes $\supseteq$ in Proposition~\ref{complexified tracking 2}.

\begin{cor}\label{subset 22}
$\bigcup\upphi_L(\bE_+)\subseteq \Level_{\bC}\!\backslash\scrH_{\bC}^{\aff}$
\end{cor}
\begin{proof}
By \eqref{both factors 2}  $\upphi_L$ restricts to a bijection between $(\scrW_L)_\bC$ and $\scrW_\bC$. Since mutation functors also preserve the level, and $\scrH^{\aff}_{\mathbb{C}}= 
\scrW_{\bC}\cap\Level_{\bC}$, it suffices to show that $\bE_+\subseteq \Level_{\bC}\!\backslash\scrH_{\bC}^{\aff}$.

For this, consider $z=x+y\ii\in\bE_+$. The key is to view this in $\bH'_+$, then follow the proof of Corollary~\ref{subset 1}.  We appeal to Lemma~\ref{hyper form}\eqref{hyper form 2} instead of Lemma~\ref{hyper form}\eqref{hyper form 1}, and \eqref{both factors 2} instead of \eqref{both factors}, then it follows that $z\notin\scrW_\bC$.  Hence  $z\in\Level_{\bC}\!\backslash\scrH_{\bC}^{\aff}$.
 \end{proof}

To obtain the converse direction in both Propositions~\ref{complexified tracking} and \ref{complexified tracking 2}  is slightly more tricky.  As preparation, recall that if $\cH$ is a real hyperplane arrangement, then the intersection poset $\scrL(\cH)$ of $\cH$ is the set of all possible intersections of subsets of hyperplanes from $\cH$. For $X,Y\in\scrL(\cH)$, consider
\begin{align*}
\cH_X&\colonequals\{ H\mid X\subseteq H\},\\
\cH^Y&\colonequals\{ H\cap Y\mid  Y\nsubseteq H\},
\end{align*}
called the   localization and restriction arrangements respectively.  The localization $\cH_X$ is a subarrangement of $\cH$, whilst $\cH^Y$ is an arrangement in $Y$. The following is elementary.

\begin{lem}\label{int then loc}
If $X\subseteq Y$, then $(\cH_X)^Y=(\cH^Y)_{X\cap Y}$.
\end{lem}
\begin{proof}
On one hand $(\cH_X)^Y=\{H\cap Y\mid X\subseteq H\mbox{ and } Y\nsubseteq H\}$.  On the other hand
\[
(\cH^Y)_{X\cap Y}=\{ H\cap Y\mid Y\nsubseteq H\mbox{ and }X\cap Y\subseteq H\cap Y\}.
\]
These two sets are clearly the same, using the assumption that $X\subseteq Y$.
\end{proof}

Returning to our flops setting, recall from Lemma~\ref{hyper form} that the coordinate axes belong to $\scrH$ and $\scrW$.   As notation, for a subset $\scrI\subseteq \{1,\hdots,n\}$, consider 
\[
\mathrm{B}=\{\upvartheta_{j}=0 \mbox{ for all }j\in \scrI\}=\bigcap_{j\in \scrI}\{\upvartheta_i=0\}\in\scrL(\scrH).
\] 
Similarly, for $\scrI'\subseteq \{0,1,\hdots,n\}$, consider $\mathds{B}=\{x_{j}=0 \mbox{ for all }j\in \scrI'\}\in\scrL(\scrW)$. 
In the affine setting, the following result will be crucial.  

\begin{lem}\label{removing gives finite}
If $\scrI'\subsetneq \{0,\hdots,n\}$, then $\scrW_{\mathds{B}}$ is finite
\end{lem}
\begin{proof}
Say the Tits cone associated to the extended ADE graph $\Updelta_{\aff}$ lives inside the vector space $V\cong \bR^{|\Updelta_{\aff}|}$, with coordinates $(x_0,\hdots,x_m)$.   Write $\scrT$ for the set of full hyperplanes in $V$ that separate the chambers in the Tits cone.   The arrangement $\scrW$ is the set of full hyperplanes that separate chambers in $\Cone{\scrJ_{\aff}}$, so by Definition~\ref{T Cone def} we can write $\scrW=\scrT^Y$ for some $Y$ obtained by intersecting coordinate axes.  By the commutative diagram in Lemma~\ref{int then loc}, to show that $\scrW_\mathds{B}=(\scrT^Y)_{\mathds{B}}$ is finite, it suffices to show that $\scrW=\scrT_X$ is finite for any $X=\bigcap_{k\in \scrK}\{x_k=0\}$ with $\scrK\subsetneq \{0,\hdots,m\}$.  In turn, it suffices to show that the quotient  
\[
\scrT/X\colonequals \{ H/X\mid H\in\scrT_X\},
\]
is a finite arrangement.  This lives in $V/X$, which has lower dimension.

For $k\in \scrK$, applying the Coxeter element $s_k$ to the basis $\{x_j+X\mid j\in\scrK\}$ of $V/X$ negates the $x_k$ entry, and adds some multiple of $x_k$ to its neighbours in $\scrK$, via the standard Coxeter rule.  By inspection, this is the same as the Coxeter rule for $\Upgamma$, where $\Upgamma$ is obtained from $\Updelta_{\aff}$ by deleting the vertices  that are not in $\scrK$.  It follows that $\scrT/X$ is the Tits cone associated to the diagram $\Gamma$.  But deleting a non-empty set of vertices in an extended ADE Dynkin diagram gives a finite ADE Dynkin diagram, or a disjoint union thereof.  Hence the Tits cone for  $\Upgamma$ has finitely many hyperplanes, hence so too does $\scrT/X$, and thus $\scrT_X$.
\end{proof}

Given $\scrI\subseteq\{1,\hdots,n\}$, $\scrI'\subsetneq\{0,1,\hdots,n\}$, define
\begin{align*}
\mathrm{D}_-&\colonequals \{ \upvartheta\in\Uptheta_\bR\mid \upvartheta_j<0\mbox{ for all }j\in \scrI\}\\
\mathds{D}_-&\colonequals \{ x\in\scrK_\bR\mid x_j<0\mbox{ for all }j\in {\scrI'}\},
\end{align*}
and let $\MutTo_{\scrI}(N)$, respectively $\MutTo_{\scrI'}(N)$, be the set of all mutations $L\to\hdots\to N$ whose constituent length one paths all have labels in the set $\scrI$, respectively $\scrI'$. The following is one of the main technical results of this Appendix.

\begin{prop}\label{E lemma}
Consider subsets $\scrI\subseteq\{1,\hdots,n\}$, $\scrI'\subsetneq\{0,1,\hdots,n\}$, with associated $\mathrm{B}$ and $\mathds{B}$. Then the following statements hold.
\begin{enumerate}
\item\label{E lemma 1} The hyperplanes in $\scrH_\mathrm{B}$ are precisely those hyperplanes $\uplambda_1\upvartheta_{i_1}+\hdots+\uplambda_s\upvartheta_{i_s}=0$ from $\scrH$ such that every $i_j\in\scrI$.  Necessarily each $\uplambda_j>0$.
\item\label{E lemma 1b} The hyperplanes in $\scrW_\mathds{B}$ are precisely those hyperplanes $\uplambda_1x_{i_1}+\hdots+\uplambda_sx_{i_s}=0$ from $\scrW$ such that every $i_j\in\scrI'$.  Necessarily each $\uplambda_j>0$.
\item\label{E lemma 2} There are decompositions
\[
\bigcup_{\upalpha\in\MutTo_{\scrI}(N)}\kern -10pt\upvarphi_\upalpha(\mathrm{D}_-)=\Uptheta_\bR\backslash \scrH_\mathrm{B}
\quad\mbox{and}\quad
\bigcup_{\upbeta\in\MutTo_{\scrI'}(N)}\kern -10pt\upphi_\upalpha(\mathds{D}_-)=\scrK_\bR\backslash \scrW_\mathds{B}.
\]
\end{enumerate}
\end{prop}
\begin{proof}
(1) The first statement is elementary. The second is immediate from Lemma~\ref{hyper form}\eqref{hyper form 1}. Part (2) is identical, using instead Lemma~\ref{hyper form}\eqref{hyper form 2}.\\
(3)  Tracking across the mutation $\upnu_iN\to N$ for $i\in \scrI$, by \eqref{k-correspond} $\mathrm{D}_-$ gets sent to the region
\begin{equation}
\upvartheta_i>0, \quad\upvartheta_j+b_{ij}\upvartheta_i<0\mbox{ for }j\in \scrI-\{i\}\label{Eplus track 1}
\end{equation}
of $\Uptheta_\bR=\bR^n$, where $b_{ij}\in \bZ_{\geq0}$.  In contrast, $C_-$ gets sent to the region
\begin{equation}
\upvartheta_i>0, \quad\upvartheta_j+b_{ij}\upvartheta_i<0\mbox{ for }j\in\{1,\hdots,n\}-\{i\}.\label{Cplus track 1}
\end{equation}
Thus, by part \eqref{E lemma 1} we see that the walls of \eqref{Eplus track 1} are precisely those hyperplanes in $\scrH$ that belong to $\scrH_\mathrm{B}$ and also bound the walls of \eqref{Cplus track 1}. Hence the region \eqref{Eplus track 1} is bounded by elements of $\scrH_\mathrm{B}$.  There are no further walls inside this region, since otherwise there would be further walls within \eqref{Cplus track 1}, which is not the case, using the $C_-$ version of Theorem~\ref{HomMMP finite summary}.

Repeating the above argument, all chambers adjacent to $\mathrm{D}_-$ in $\scrH_\mathrm{B}$ can be obtained by tracking $\mathrm{D}_-$ through some mutation $\upnu_iN\to N$.   The proof  then just proceeds by induction.  Consider $\upnu_j\upnu_iN\to\upnu_iN\to N$, track $\mathrm{D}_-$ through both mutations, and just appeal to the $C_-$ version of Theorem~\ref{HomMMP finite summary}. This process finishes since the numbers of chambers in $\scrH_\mathrm{B}$ is finite, since $\scrH$ is, and at each stage mutation at the labels in $\scrI$ describes the $|\scrI|$ possible wall-crossings in each chamber.

The last statement is similar, replacing $\Updelta$ by $\Updelta_{\aff}$, and using the $C_-$ version of Theorem~\ref{affine summary} in place of Theorem~\ref{HomMMP finite summary}. The key point is that, since $\scrI'$ is a proper subset, by Lemma~\ref{removing gives finite} the hyperplane arrangement $\scrW_\mathds{B}$ is still finite, and so the conclusion of the last sentence in the above paragraph still holds.
\end{proof}

\begin{cor}\label{subset 2}
There is an inclusion $\bC^n\backslash \scrH_\bC\subseteq\bigcup_{L\in\Mut_0(N)}\upvarphi_L(\bH_+)$, and furthermore $\Level_{\bC}\!\backslash\scrH_{\bC}^{\aff}\subseteq\bigcup_{L\in\Mut(N)}\upphi_L(\bE_+)$.
\end{cor}
\begin{proof}
Pick $z\in \bC^n\backslash \scrH_\bC$, and write $z=x+y\ii$.  By Theorem~\ref{HomMMP finite summary} applied to $\scrH_y$, we can find some $L\in\Mut_0(N)$ such that $\upvarphi_L^{-1}(y)$ has all coordinates $\geq 0$. If all coordinates are positive, then $\upvarphi_L^{-1}(z)\in\bH_+$, thus $z\in\upvarphi_L(\bH_+)$, as required.

Hence we can assume that some coordinate of $y'\colonequals\upvarphi_L^{-1}(y)\in\Uptheta_L$ is zero.  Thus there exists some non-empty subset $\scrI\subseteq\{1,\hdots,n\}$ such that $y'_i=0$ for all $i\in\scrI$, and $y'_j>0$ for $j\notin\scrI$.  Note that $\scrI=\{1,\hdots,n\}$ is possible, in which case all $y'_i=0$.  

Recall that $\upvarphi_L\colonequals \upvarphi_\upalpha$, where $\upalpha\colon L\to\hdots\to N$.  Write
\[
\upvarphi_\upalpha^{-1}(z)=\upvarphi_L^{-1}(z)=x'+y'\ii.
\]
Since $x'$ belongs to the closure of some chamber in the $x$-version of $\Uptheta_L$,  applying Proposition~\ref{E lemma}\eqref{E lemma 2} to $(\scrH_L)_x\subset\Uptheta_L$ we can find a sequence of mutations
\[
M\xrightarrow{\upnu_{j_1}}\hdots\xrightarrow{\upnu_{j_s}}L
\]
with each $j_k\in\scrI$, such that $\upvarphi_{j_1}^{-1}\hdots\upvarphi_{j_s}^{-1}(x)\in\overline{D}_-$.  Since each $j_k\in\scrI$, $y'_{j_k}=0$, and so this path has no effect on $y'$.  In particular, 
\begin{equation}
\upvarphi_{j_1}^{-1}\hdots\upvarphi_{j_s}^{-1}\upvarphi_\upalpha^{-1}(z)=x''+y'\ii,\label{show in Hplus}
\end{equation}
where $x''_i\leq 0$ provided that $i\in\scrI$. Since $z\notin\scrH_{\mathbb{C}}$, we must have $x_i''\neq 0$ for all $i\in\scrI$, else $\upvarphi_{j_1}^{-1}\hdots\upvarphi_{j_s}^{-1}\upvarphi_\upalpha^{-1}(z)$ and thus $z$ belongs to a complexified hyperplane.  Hence \eqref{show in Hplus} belongs to $\bH_+$,  and so $z\in\upvarphi_\upalpha\upvarphi_{j_s}\hdots\upvarphi_{j_1}(\bH_+)$.  This is the tracking of $\bH_+$ through the chain
\[
M\xrightarrow{\upnu_{j_1}}\hdots\xrightarrow{\upnu_{j_s}}L\overbrace{\to\hdots\to}^{\upalpha} N.
\] 
Appealing to  Proposition~\ref{theta trivial}, $\upvarphi_\upalpha\upvarphi_{j_s}\hdots\upvarphi_{j_1}=\upvarphi_M$, and so $z\in\upvarphi_M(\bH_+)$.  

For the second statement, let $z\in \Level_{\bC}\!\backslash\scrH_{\bC}^{\aff}$.  The proof proceeds as above, replacing $\upvarphi$ by $\upphi$ at all stages.  We obtain a subset $\scrI'\subseteq\{0,1,\hdots,n\}$, but crucially now $\scrI'\neq\{0,1,\hdots,n\}$ since $\upvarphi_L^{-1}(z)\in(\Level_L)_{\mathbb{C}}$.  This is due to the fact that mutation functors preserve the level, and so multiples of the $y$ coordinates must sum to one, hence not all the $y$ can be zero.  Hence $\scrI'$ is a proper subset, which still allows us to still appeal to Proposition~\ref{E lemma}\eqref{E lemma 2}.  We thus still deduce that $\upphi_M^{-1}(z)\in\bH'_+$.  Since mutation functors preserve the level, automatically it follows that $\upphi_M^{-1}(z)\in\bE_+$, and so $z\in\upphi_M(\bE_+)$.
\end{proof}

We lastly show that the unions are disjoint, which requires the following.

\begin{lem}\label{key comb g-vect}
Let $L,M\in\Mut_0(N)$, respectively $\Mut(N)$, and let $\upalpha\colon L\to M$ be a minimal path. Suppose that $\upvarphi_\upalpha(a'_1,\hdots,a'_n)=(a_1,\hdots,a_n)$ in $(\Uptheta_M)_\bR$, respectively $\upphi_\upalpha(a'_0,\hdots,a'_n)=(a_0,\hdots,a_n)$ in $(\scrK_M)_\bR$, with all $a_i,a_i'\in\mathbb{R}_{\geq 0}$, and write $\scrI=\{i\mid a'_i=0\}$. 
\begin{enumerate}
\item\label{key comb g-vect 11} If $\scrI=\emptyset$, then $L\cong M$.
\item\label{key comb g-vect 22} If $\scrI\neq\emptyset$, then the following statements hold.
\begin{enumerate}
\item\label{key comb g-vect 0} The simple mutations giving $\upalpha$ all have labels in the set $\scrI$.
\item\label{key comb g-vect 1} $a_i'=a_i$ for all $i$. 
\item\label{key comb g-vect 2} $\upvarphi_\upalpha[\scrP_i']=[\scrP_i]$, respectively $\upphi_\upalpha[\scrP_i']=[\scrP_i]$, for all $i\notin\scrI$.
\end{enumerate} 
\end{enumerate} 

\end{lem}
\begin{proof}
We prove the affine case $\scrK_M$, with the case $\Uptheta_M$ being similar.

Set $p=\sum_{i=0}^na'_i[\scrP_i]$.  For a general subset $J$ of $\{0,1,\hdots,n\}$ and $ N'\in \Mut(N)$, set
\[
C_J\colonequals \left\{ \sum_{i=0}^nc_i[\scrP_i]\in(\scrK_{N'})_\bR
\left|
\begin{array}{cl}
c_i=0 &\mbox{if }i\in J\\
c_i>0 &\mbox{if }i\notin J
\end{array}
\right.
\right\}. 
\]
As calibration, note that $C_+=C_{\emptyset}$.

\noindent
(1) If all $a_i>0$, then since chambers map to chambers,  $\upphi_\upalpha(p)\in C_+$.  Thus $C_+\cap\upphi_\upalpha(C_+)\neq\emptyset$. By Theorem~\ref{affine summary} (respectively \ref{HomMMP finite summary} for $\Uptheta$), $L\cong M$. 

\noindent
(2) By assumption $\scrI\neq\emptyset$, in which case its complement $\scrI^c$ in $\{0,1,\hdots,n\}$ is a proper subset. Since $p\in C_\scrI$, which is a codimension $|\scrI^c|$ wall of $C_+$, and $\upphi_\upalpha$ maps walls to walls (maintaining codimension), it follows that $\upphi_\upalpha(p)$ lies in a codimension $|\scrI^c|$ wall of $C_+$.  These all have the form  $C_{\scrI'}$ for some $|\scrI'|=|\scrI|$, and so  it follows that $\upphi_\upalpha(p)\in C_{\scrI'}$ for some such $\scrI'$ with $|\scrI'|=|\scrI|$.  

We first argue that (a) implies  (b) and (c). Indeed, since $\upphi_j$, with $j \in \scrI$, negates the entry $j$ (which is zero) and adds zero to the neighbours, evidently this has no effect on elements in $C_{\scrI}$, and so $(a_0,\hdots,a_n)=\upphi_\upalpha(a'_0,\hdots,a'_n)=(a'_0,\hdots,a'_n)$. Consequently, we have $a'_i=a_i$ for all $i$, proving (b).  Furthermore, (c) follows immediately from (a), using Lemma \ref{k-correspond proj}.

Now we prove (a). Write $\upalpha\colon L_0\colonequals L\to L_1\to \hdots\to L_{m}\colonequals M$ for a minimal path from $L$ to $M$. By \cite[1.12]{IW9}, in every $\scrK_{L_i}$, all chambers and all walls of all codimension are labelled by Coxeter information. Namely, by $wC_J$ for some $w$ in the affine Weyl group $W_{\Updelta_{\aff}}$, and some subset $J$ of the vertices of the affine Dynkin diagram $\Updelta_{\aff}$.  The codimension $|\scrI^c|$ walls have the form $w C_{J}$ for certain $J$ and $w\in W_{\Updelta_{\aff}}$, with $|J|=|\scrI^c|$.    Consequently, tracking $C_\scrI$ in $\scrK_L\otimes\mathbb{R}$ through $\upphi_\upalpha=\upphi_{i_n}\hdots\upphi_{i_1}$ gives a sequence of labelled walls
\[
C_\scrI \mapsto w_1C_{\scrI_1} \mapsto w_2 C_{\scrI_2} \mapsto  \hdots  \mapsto w_m C_{\scrI_m} = \upphi_\upalpha(C_\scrI),  
\]
where the last term is in $\scrK_M\otimes\mathbb{R}$.  At each step, since atoms follow the weak order, the length of the smallest coset representative $w_i$ cannot decrease.  This holds just since the statement is true for chambers, 
%\marginpar{Although the $w$ are in `different' Weyl groups, the length difference is always the same}
and the labels on the walls are induced from these.

Since $\upphi_\upalpha(p)\in C_{\scrI'}$, and $\upphi_\upalpha(p)\in \upphi_\upalpha(C_\scrI)= w_m C_{\scrI_m}$, it follows that $C_{\scrI'}\cap w_m C_{\scrI_m}\neq\emptyset$.  As is standard \cite[V.4.6, Proposition 5]{Bourbaki},
%Let $\scrI$ and $\scrI'$ be subsets $\Delta_0$, and $w$ an element of the Coxeter group. If $w(C_\scrI)\cap C_{\scrI'}\neq\emptyset$, then $\scrI=\scrI'$ and $w\in W_\scrI$.
 we deduce that $\scrI_m=\scrI'$, and $w_m\in W_{\scrI'}$, where $W_{\scrI'}$ is the subgroup of $W_{\Updelta_{\aff}}$ generated by $s_i$ with $i\in\scrI'$.   In particular $w_m C_{\scrI_m}=C_{\scrI'}$, and so the above chain is
\[
C_\scrI \mapsto w_1C_{\scrI_1} \mapsto w_2 C_{\scrI_2} \mapsto  \hdots  \mapsto C_{\scrI'}.  
\]
As the minimal length of the coset representative $w_i$ cannot decrease throughout the chain, and the chain starts and finishes with length zero, it follows that at each step $w_i C_{\scrI_i}=C_{\scrI_i}$.  Thus, at step one, $\upphi_{i_1}\colon C_\scrI\mapsto C_{\scrI_1}$.  By inspection, this  occurs if and only if the label $i_1$ is in the set $\scrI$, and $\scrI=\scrI_1$.  Inducting along the chain, every label $i_t$ is in the set $\scrI$, and (a) follows.
\end{proof}

\begin{cor}\label{disjoint 1}
Suppose that $L,M\in\Mut_0(N)$, and let $\upalpha$ be the minimal path from $L$ to $M$.  Then  
$\upvarphi_\upalpha(\mathbb{H}_+)\cap \mathbb{H}_+ \neq\emptyset$ in $\Uptheta_M\iff L\cong M$. The same statement holds for $\Mut(N)$, using instead $\upphi_\upalpha(\bE_+)\cap \bE_+ \neq\emptyset$.
\end{cor}
\begin{proof}
($\Leftarrow$) is clear.  For ($\Rightarrow$),  if the intersection is nonempty, then 
\begin{align*}
 \upvarphi_\upalpha(z'_1,\hdots,z'_n) = (z_1,\hdots,z_n)
% \\
%\sum_{i=1}^n z'_i \upvarphi_\upalpha[\scrP_i'] = \sum_{i=1}^n z_i [\scrP_i],
\end{align*}
where all $z_i, z'_i \in \mathbb{H}$. Splitting into real and imaginary parts,
\begin{equation}
 \upvarphi_\upalpha(x'_1,\hdots,x'_n) = (x_1,\hdots,x_n)
  \quad \text{and} \quad   \upvarphi_\upalpha(y'_1,\hdots,y'_n) = (y_1,\hdots,y_n).\label{two split}
%  
% \sum_{i=1}^n x'_i \upvarphi_\upalpha[\scrP_i'] = \sum_{i=1}^n x_i [\scrP_i] \quad \text{and} \quad  \sum_{i=1}^n y'_i \upvarphi_\upalpha[\scrP_i'] = \sum_{i=1}^n y_i [\scrP_i].\label{two split}
\end{equation} 
Since all $z_i,z_i'\in\mathbb{H}$, necessarily the right hand equation belongs to $\upvarphi_\upalpha(\overline{C}_+)\cap \overline{C}_+$.
As before, write $\scrI$ for the set of $i$ for which $y'_i=0$. 

On one hand, if $\scrI$ is not a proper subset, then   $y_i=y'_i=0$ for all $i$.  Since every $z_i,z'_i\in\bH$, all $x_i,x'_i<0$. Using the $C_-$ version of Proposition~\ref{HomMMP finite summary} applied to the $x$-coordinate, $L\cong M$.  On the other hand, if $\scrI=\emptyset$, then by Lemma~\ref{key comb g-vect}\eqref{key comb g-vect 11} we also have $L\cong M$.  

Hence we can assume that  $\scrI\neq\emptyset$ and $\scrI$ is a proper subset.  Since $\scrI\neq\emptyset$, by Lemma~\ref{key comb g-vect}\eqref{key comb g-vect 22}, $\upalpha$ comprises of mutations with labels only from the set $\scrI$. Further, all $y_i'=y_i$ (and so in particular $y_i=0$ if $i\in\scrI$), and  $\upvarphi_\upalpha[\scrP_i']=[\scrP_i]$ if $i\notin\scrI$. Since $z_i,z_i'\in\mathbb{H}$, we then deduce that $x_i<0$ and $x_i'<0$ for all $i\in \scrI$, and that we can re-write the left hand equation in \eqref{two split} to obtain
\[
\sum_{i\in\scrI} x_i' \upvarphi_\upalpha[\scrP_i']+\sum_{i\notin\scrI} x_i' [\scrP_i]
=
\sum_{i\in\scrI} x_i [\scrP_i]
+
\sum_{i\notin\scrI} x_i [\scrP_i].
\]
Set $I\subset \scrI^c$ to consist of those $i$ such that $x'_i-x_i\geq 0$, and let $I^c=\scrI^c-I$. Re-arranging gives
\begin{align*}
\sum_{i \in I} (x'_i-x_i) [\scrP_i] 
+ \sum_{i\in\scrI}(- x_i) [\scrP_i]
&=  \sum_{i \in I^c} (x_i-x'_i) [\scrP_i] 
- \sum_{i\in\scrI} x_i' \upvarphi_\upalpha[\scrP'_i].\\
&= \upvarphi_\upalpha\left( \sum_{i \in I^c} (x_i-x'_i) [\scrP_i'] 
+ \sum_{i\in\scrI} (-x_i' )[\scrP_i']\right).
\end{align*}
The left hand side has non-negative coefficients in every entry $\{0,1\hdots,n\}$, and  the right hand side is in $\upvarphi_\upalpha(\overline{C}_+) $.  If $I=\emptyset$, so that $I^c=\scrI^c$, then the right hand side is in $\upvarphi_\upalpha(C_+)$ and so by Lemma~\ref{key comb g-vect}\eqref{key comb g-vect 11}, $L\cong M$.

Hence our final case is when $\scrI\neq\emptyset$, $\scrI$ is proper, and $I\neq\emptyset$.  We will show that this cannot occur, by exhibiting a contradiction.  Again, the above displayed equation lies in $\upvarphi_\upalpha(\overline{C}_+) \cap \overline{C}_+$  and so now  by Lemma~\ref{key comb g-vect}\eqref{key comb g-vect 22} the coefficients on both sides must match.  But coefficients in $I^c$ do not appear on the left hand side, nor do coefficients in $I$ on the right, so we deduce that $x_i-x_i'=0$ for all $i\in\scrI^c$.  But then 
\[
\upvarphi_\upalpha\left(\sum_{i\in\scrI}(- x'_i) [\scrP_i]\right)= \sum_{i\in\scrI} (-x_i) [\scrP_i],
\]
and so by  Lemma~\ref{key comb g-vect}\eqref{key comb g-vect 22} $\upalpha$ comprises mutations only from the set $\scrI^c$.  But as stated above, $\upalpha$ can only comprise labels in the set $\scrI$.  This is a contradiction, which shows that this final case cannot exist.

The proof of the second statement is similar.  The set $\scrI$ must be proper since $z\in\bE_+$, so all $y$ coordinates are nonnegative, and after weighting by $\rk_RM_i$ they sum to one.  Hence not all can be zero.  Given this, the rest of the proof remains the same: since $\bE_+\subset\mathbb{H}'_+$, replacing $\upvarphi$ by $\upphi$ throughout, and starting the indices with $0$, the logic above still holds, as we can still appeal to Lemma~\ref{key comb g-vect} in the affine case.
\end{proof}

\begin{cor}\label{disjoint}
If $L,M\in\Mut_0(N)$,  then $\upvarphi_{L}(\mathbb{H}_+)\cap \upvarphi_{M}(\mathbb{H}_+) \neq\emptyset$ in $\Uptheta_N\iff L\cong M$. The same statement holds for $\Mut(N)$, using instead $\upphi_{L}(\bE_+)\cap \upphi_{M}(\bE_+) \neq\emptyset$ in $\scrK_N$.
\end{cor}
\begin{proof}
($\Leftarrow$) is clear.  For ($\Rightarrow$), if  $\upvarphi_{L}(\mathbb{H}_+)\cap \upvarphi_{M}(\mathbb{H}_+) \neq\emptyset$, then $\upvarphi_{M}^{-1}\circ\upvarphi_{L}^{\phantom -}(\mathbb{H}_+)\cap \mathbb{H}_+ \neq\emptyset$.  By Proposition~\ref{theta trivial}\eqref{theta trivial 3}, $\upvarphi_{\upalpha}(\mathbb{H}_+)\cap \mathbb{H}_+ \neq\emptyset$ in $\Uptheta_{M}$, where $\upalpha$ is the minimal path from $L$ to $M$.  By Corollary~\ref{disjoint 1}, the atom $\upalpha$ is the identity, and  the result follows. The second statement is identical, as we can still appeal to Proposition~\ref{theta trivial}\eqref{theta trivial 3} and  Corollary~\ref{disjoint 1}.
\end{proof}

\subsection{Proof of Propositions~\ref{complexified tracking} and \ref{complexified tracking 2}}\label{proof of Props section}

Proposition~\ref{complexified tracking} asserts that there is an equality
\[
\Uptheta_{\bC}\backslash\scrH_{\bC}=\bigcup_{L\in\Mut_0(N)}\upvarphi_L(\bH_+),
\]
where the union on the right hand side is disjoint.  This now follows by combining Corollaries~\ref{subset 1}, \ref{subset 2} and \ref{disjoint}.  On the other hand, Proposition~\ref{complexified tracking 2} asserts that
\[
\Level_{\bC}\!\backslash\scrH_{\bC}^{\aff}=\bigcup_{L\in\Mut(N)}\upphi_L(\bE_+),
\]
where the union on the right hand side is disjoint.  This now follows by combining Corollaries~\ref{subset 22}, \ref{subset 2} and \ref{disjoint}.

\section{List of Notation}\label{notation appendix}
\noindent
\begin{longtable}[l]{p{1.8cm} p{8.25cm} p{1cm}}
%\opt{12pt}{\begin{longtable}{p{3cm} p{8.5cm} p{0.5cm}}}
%{\bf Section 2}\\
$\Curve$ & $\colonequals f^{-1}(\m)$ with reduced scheme structure&p\pageref{sec:FlopsviaNC}\\
$n$&number of irreducible components of $\Curve$&p\pageref{sec:FlopsviaNC}\\
$\Curve_1,\hdots,\Curve_n$ & irreducible components of $\Curve$&p\pageref{sec:FlopsviaNC}\\
$\scrL_i$ &line bundle on $X$ such that $\scrL_i\cdot \Curve_j=\updelta_{i,j}$&\S\ref{sec:tilt and modify}\\
$\scrV_i$&universal extension of $\scrL_i$&\eqref{extension}\\
$\scrV_X$& tilting bundle $\oplus_{i=0}^n\scrV_i^*$  on $X$&\eqref{tilt equiv}\\
$N_i$, $N$&$f_*(\scrV_i^*)$ and $f_*(\scrV_X)$ respectively&\eqref{ass modi}\\
$\Lambda$&endomorphism algebra $\End_X(\scrV_X)\cong \End_R(N)$&\S\ref{sec:tilt and modify}\\
$\refl R$ &category of reflexive $R$-modules&\S\ref{sec:tilt and modify}\\
$\CM R$ & category of maximal CM $R$-modules&\S\ref{sec:tilt and modify} \\
$\add M$ & summands of finite sums of module $M$ &\S\ref{sec:mut and equiv} \\
$\upnu_iL$&  mutation of $L$ at $i$-th summand&\S\ref{sec:mut and equiv}\\
$\upnu_i\Lambda$ &endomorphism algebra $\End_R(\upnu_iN)$&\S\ref{sec:mut and equiv}\\
$\Mut(N)$ &iterated mutations of $N$&\ref{exchange def}\\
$\Mut_0(N)$ &subset mutating at only $(i\neq0)$-th summands&\ref{exchange def}\\
$\EG(N)$&exchange graph of $N$&\ref{exchange def}\\
$\EG_0(N)$&subgraph with vertices CM modules&\ref{exchange def}\\
$\Upphi_i$&mutation functor at the $i$-th summand&\eqref{mutation eqv}\\
${\sf Flop}_i$&quasi-inverse of the flop functor&\ref{flopmut}\\
%\\
%{\bf Section 3}\\
$\Updelta$, $\Updelta_{\aff}$&ADE Dynkin diagram, and extended version&\S\ref{elephant subsection}\\
$\star$&extending vertex in $\Updelta_{\aff}$&\S\ref{elephant subsection}\\
$\scrJ$&subset of vertices of $\Updelta$&\S\ref{elephant subsection}\\
$\scrJ_{\aff}$&$\scrJ$ considered as a subset of $\Updelta_{\aff}$&\S\ref{elephant subsection}\\
$\Cone{\scrJ}$&$\scrJ$-finite hyperplane arrangement&\S\ref{elephant subsection}\\
$\Cone{\scrJ_{\aff}}$&$\scrJ$-affine arrangement&\S\ref{elephant subsection}\\
$\scrP_i$&projective module of the $i$-th summand&\S\ref{affine notation subsection}\\
$\Lambda_L$&endomorphism algebra $\End_R(L)$&\S\ref{affine notation subsection}\\
$\scrK_L$&$K_0$-group of  perfect complexes over $\Lambda_L$&\S\ref{affine notation subsection}\\
$\upphi_i$&isomorphism of $K_0$-groups associated to $\Upphi_i$&\S\ref{affine notation subsection}\\
$\Upphi_L$&compositions of $\Upphi_i$ along a shortest path&\eqref{eqn:defn PhiL}\\
$\upphi_L$&isomorphism of $K_0$-groups associated to $\Upphi_L$&\S\ref{affine notation subsection}\\
$C_+$&positive cone in a real vector space&\S\ref{affine notation subsection}\\
$\Uptheta_L$&$\colonequals\scrK_L/[\scrP_0]$, set of stability parameters&\S\ref{finite notation subsection}\\
$\upvarphi_L$&isomorphism on $\Uptheta$ induced from $\upphi_L$&\S\ref{finite notation subsection}\\
$\rk_RM$ &rank of $R$-module $M$&\ref{def: level}\\
$\Level_L$&$\colonequals\{z\in\scrK_L\otimes\bR\mid \sum(\rk_RL_j)z_j=1\}$&\ref{def: level}\\
$\Alcove_L$&$\colonequals\upphi_L(C_+)\cap\,\Level_N$&\S\ref{affine notation subsection}\\
$\scrH^{\aff}$&complement of the union of all $\Alcove_L$&\eqref{affine H}\\
$\scrH$&complement of the union of all $\upvarphi_L(C_+)$&\eqref{finite H}\\
$\bH$&the semi-closed upper half plane in $\bC$&\S\ref{complexified actions subsection}\\
$\bH_+$& subset of $(\Uptheta_L)_{\bC}\cong \bC^n$ corresponding to $\bH^n$&\S\ref{complexified actions subsection}\\
$\bH_+'$&subset of $(\scrK_L)_{\bC}\cong \bC^{n+1}$ corresponding to $\bH^{n+1}$&\S\ref{complexified actions subsection}\\
$\ii$&imaginary number $\sqrt{-1}$&\S\ref{complexified actions subsection}\\
$(\Level_L)_{\bC}$&$\colonequals\{z\in \scrK_L\otimes \bC\mid \sum(\rk_RL_j)z_j=\ii\}$&\S\ref{complexified actions subsection}\\
$\bE_+$&$\colonequals\bH_+'\cap(\Level_L)_{\bC}$&\S\ref{complexified actions subsection}\\
$\scrW$&hyperplanes separating chambers of $\Cone{\scrJ_{\aff}}$&p\pageref{cpx of Haff}\\
$W_{\bC}$&$W\oplus \ii W$&p\pageref{cpx of Haff}\\
$\scrW_{\bC}$&union of all $W_{\bC}$&p\pageref{cpx of Haff}\\
$\scrH_{\bC}^{\aff}$&$\scrW_{\bC}\cap(\Level_N)_{\bC}$&\eqref{cpx of Haff}\\
%\\
%{\bf Section 4}\\
$\Gamma_{\cH}$&graph of a hyperplane arrangement $\cH$&\ref{def:Gamma graph}\\
$\cG_{\cH}^+$&category of positive paths in $\Gamma_{\cH}$&\ref{arrangement prelim subsection}\\
$\cG_{\cH}$&groupoid completion of $\cG_{\cH}^+$&\ref{def: Deligne groupoid completion}\\
$\mathds{G}^{\aff}$, $\mathds{G}$&$\cG_{\scrH^{\aff}}$ and $\cG_{\scrH}$ respectively&\ref{not: 4.4}\\
$\Upphi_{\upalpha}$&compositions of $\Upphi_i$ along a path $\upalpha$&\S\ref{first results}\\
%\\
%{\bf Section 5}\\
$\Stab\scrT$&stability conditions on triangulated category $\scrT$&\S\ref{sec: stab generalities}\\
$\Auteq\scrT$&group of autoequivalences of $\scrT$&\S\ref{sec: stab generalities}\\
$\Phi_*$&map on stability conditions induced by  $\Phi$&\S\ref{sec: stab generalities}\\
$\scrS_i$&simple module corresponding to projective $\scrP_i$&\S\ref{sec:stab normal and mut}\\
$\scrB_L$&category of finite length $\Lambda_L$-modules&\S\ref{sec:stab normal and mut}\\
$\scrA_L$&subcategory of $\scrB_L$ without $\scrS_0$&\S\ref{sec:stab normal and mut}\\
$\scrC_L$, $\scrD_L$ & subcategory of $\Db(\fmod\Lambda_L)$ of complexes with cohomology in $\scrA_L$, respectively $\scrB_L$&\S\ref{sec:stab normal and mut}\\
$\scrA,\scrB,\scrC,\scrD$&corresponding categories when $L=N$&\S\ref{sec:stab normal and mut}\\
$\Stab\cA$&stability on heart $\cA$ satisfying HN property&\S\ref{sec:stab normal and mut}\\
$\nStab{n}\scrD_L$&space of normalised stability conditions on $\scrD_L$&\S\ref{sec:stab normal and mut}\\
$\nStab{n}\scrB_L$&$\colonequals\Stab\scrB_L\cap\nStab{n}\scrD_L$&\S\ref{sec:stab normal and mut}\\
%\\
%{\bf Section 6}\\
$\cStab{}\scrC$&component of $\Stab\scrC$ containing $\Stab\scrA$&p\pageref{stab on C and D section}\\
$\cStab{n}\scrD$&component of $\nStab{n}\scrD$ containing $\Stab\scrB$&p\pageref{stab on C and D section}\\
$\Targ_0(C_+)$&morphisms in $\dsG$ which terminate at $C_+$&\S\ref{sec: chamber decomp}\\
$\Targ(C_+)$&morphisms in $\dsG^{\aff}$ which terminate at $C_+$&\S\ref{sec: chamber decomp}\\
$\mathrm{U}_L$&certain open subset of $\Stab\scrA_L$&\ref{chamber notation}\\
$\bU_L$&certain open subset of $\Stab\scrB_L$&\ref{chamber notation}\\
$\dsN_L$&$\colonequals \bU_L\cap \nStab{n}\scrB_L $&\ref{chamber notation}\\
$\Stab\scrA_{\upalpha}$&$\colonequals(\Upphi_{\upalpha})_*(\Stab\scrA_{s(\upalpha)})$&\ref{chamber notation}\\
$\Stab\scrB_{\upbeta}$&$\colonequals(\Upphi_{\upbeta})_*(\Stab\scrB_{s(\upbeta)})$&\ref{chamber notation}\\
$\mathrm{U}_{\upalpha}$&$\colonequals(\Upphi_{\upalpha})_*(\mathrm{U}_{s(\upalpha)})$&\ref{chamber notation}\\
$\bU_{\upbeta}$&$\colonequals(\Upphi_{\upbeta})_*(\bU_{s(\upbeta)})$&\ref{chamber notation}\\
$\nStab{n}\scrB_{\upbeta}$&$\colonequals (\Upphi_{\upbeta})_*(\nStab{n}\scrB_{s(\upbeta)})$&\ref{chamber notation}\\
$\dsN_\upbeta$&$\colonequals (\Upphi_{\upbeta})_*(\dsN_{s(\upbeta)})$&\ref{chamber notation}\\
$\Br\scrC$&$\colonequals \{\Upphi_{\upalpha}|_{\scrC}\mid \upalpha\in\End_{\mathds G}(C_+)\}$&\ref{not: pure braid images}\\
$\Br\scrD$&$\colonequals \{\Upphi_{\upbeta}|_{\scrD}\mid \upbeta\in\End_{\dsG^{\aff}}(C_+)\}$&\ref{not: pure braid images}\\
%\\
%{\bf Section 7}\\
$\cAut{}\scrC$&certain subgroup of $\Auteq\scrC$&\S\ref{sec: auto of C}\\
$\mathsf{L}_i$&$\colonequals f_*\scrL_i$&\S\ref{section line bundles}\\
$\mathsf{L}\cdot M$&$\colonequals (\mathsf{L}\otimes M)^{**}$&\S\ref{section line bundles}\\
$\upvarepsilon$&natural isomorphism from $\Lambda_M$ to $\Lambda_{\mathsf{L}\cdot M}$&\S\ref{section line bundles}\\
$\cAut{}\scrD$&certain subgroup of $\Auteq\scrD$&\S\ref{sec: auto of D}
\end{longtable}


\begin{thebibliography}{BM2}
 
\bibitem[A]{Aspinwall}
P.~S.~Aspinwall, \emph{A point's point of view of stringy geometry}, J.\ High Energy Phys.\ 2003, no.~1, 002.
 
\bibitem[B1]{Bayer}
A.~Bayer, \emph{A tour to stability conditions on derived categories}, \href{https://www.maths.ed.ac.uk/~abayer/dc-lecture-notes.pdf}{link}.

  \bibitem[BM]{BMacri} A.~Bayer and E.~Macr\`\i, {\it The space of stability conditions on the local projective plane}, Duke Math.\ J.\ \textbf{160} (2011), no.~2, 263--322.
   
 \bibitem[B2]{Bourbaki}
 N.~Bourbaki, \emph{Lie groups and Lie algebras}.  Chapters 4--6. Elements of Mathematics (Berlin). Springer--Verlag, Berlin, 2002. xii+300pp. 
 
 \bibitem[B3]{B02} T.~Bridgeland, {\it Flops and derived categories}, Invent.\ Math.\ {\bf 147} (2002), no.~3, 613--632. 
 
 \bibitem[B4]{B07} T.~Bridgeland, {\it Stability conditions on triangulated categories},  Ann. of Math. (2) \textbf{166} (2007), no.~2, 317--345.
 
 \bibitem[B5]{B09} T.~Bridgeland, {\it Spaces of stability conditions}, Algebraic Geometry-Seattle 2005. Part 1, 1--21, Proc. Sympos. Pure Math., {\bf 80}, Part 1, Amer. Math. Soc., Providence, RI, (2009).
 
 \bibitem[B6]{B3} T.~Bridgeland, {\it Stability conditions and Kleinian singularities},  Int.\ Math.\ Res.\ Not.\ IMRN (2009), no.~21, 4142--4157.
 
\bibitem[BM2]{BM} T.~Bridgeland and A.~Maciocia, \emph{Fourier-Mukai transforms for quotient varieties}, 
J.\ Geom.\ Phys.\ \textbf{122} (2017), 119--127. 

 \bibitem[BS]{BS} T.~Bridgeland and I.~Smith, {\it Quadratic differentials as stability conditions}, Publ. Math. Inst. Hautes \'Etudes Sci. {\bf 121} (2015), 155--278. 
 
\bibitem[C]{Chen}
J.-C.~Chen, \emph{Flops and equivalences of derived categories for threefolds with only terminal Gorenstein singularities}, J. Differential Geom. \textbf{61} (2002), no. 2, 227--261.
 
 
\bibitem[D1]{Deligne}
P. Deligne, \emph{Les immeubles des groupes de tresses g\'en\'eralis\'es}, Invent.\ Math.\ \textbf{17} (1972), 273--302.

\bibitem[D2]{Delucchi}
E.~Delucchi, \emph{Combinatorics of covers of complexified hyperplane arrangements}, Arrangements, local systems and singularities, 1--38, 
Progr.\ Math., \textbf{283}, Birkh\"auser Verlag, Basel, 2010. 

\bibitem[DW]{DW1}
W.~Donovan and M.~Wemyss, \emph{Noncommutative deformations and flops}, Duke Math.\ J.\ \textbf{165} (2016), no.~8, 1397--1474. 

\bibitem[H1]{Hat} A.~Hatcher, {\it Algebraic topology}, Cambridge University Press, Cambridge (2002), xii+544pp.

 \bibitem[HW]{HW} Y.~Hirano and M.~Wemyss, {\it Faithful actions from hyperplane arrangements}, Geom.\  Topol.\ \textbf{22} (2018), no.~6, 3395--3433. 
 
 \bibitem[H2]{HuybrechtsFM}
D.~Huybrechts, \emph{Fourier-Mukai transforms in algebraic geometry}, Oxford Mathematical Monographs. The Clarendon Press, Oxford University Press, Oxford, 2006. viii+307 pp.

\bibitem[IM]{IM} N.~Iwahori, and H.~Matsumoto, \emph{On some Bruhat decomposition and the structure of the Hecke rings of $\mathfrak{p}$-adic Chevalley groups}, Inst.\ Hautes \'{E}tudes Sci.\ Publ.\ Math.\ no.~25 (1965), 5--48.

 
 \bibitem[IW1]{IW1}  O.~Iyama and M.~Wemyss, {\it Maximal modifications and Auslander-Reiten duality for non-isolated singularities}. Invent.\ Math.\ \textbf{197} (2014), no.~3, 521--586.
 
 \bibitem[IW2]{IW9}
O.~Iyama and M.~Wemyss, \emph{Tits cones intersections and applications}, \href{https://www.maths.gla.ac.uk/~mwemyss/MainFile_for_web.pdf}{\sf preprint}. 

\bibitem[KM]{KM}
S.~Katz and D.~R.~Morrison, \emph{Gorenstein threefold singularities with small resolutions via invariant theory for Weyl groups}, J.\ Algebraic Geom.\ \textbf{1} (1992) 449--530.

\bibitem[K1]{Karmazyn}
J.~Karmazyn, \emph{Quiver GIT for varieties with tilting bundles}, Manuscripta Math.\ \textbf{154} (2017), no.~1--2, 91--128.

\bibitem[K2]{Kawa}
Y.~Kawamata, \emph{General hyperplane sections of nonsingular flops in dimension 3}, Math.\ 
Res.\ Lett.\ \textbf{1} (1994) 49--52.

\bibitem[P1]{Paris}
L.~Paris, \emph{Universal cover of Salvetti's complex and topology of simplicial arrangements of hyperplanes}, Trans.\ Amer.\ Math.\ Soc.\ \textbf{340} (1993), no.~1, 149--178.

\bibitem[P2]{Paris3}
L.~Paris, \emph{The covers of a complexified real arrangement of hyperplanes and their fundamental groups},  Topology Appl. \textbf{53} (1993), no.~1, 75--103.

\bibitem[R]{Pagoda}
M.~Reid, \emph{Minimal models of canonical 3-folds}, Algebraic varieties and analytic varieties (Tokyo, 1981), 131--180, Adv. Stud. Pure Math., 1, North-Holland, Amsterdam, 1983.

\bibitem[S]{Salvetti}
M.~Salvetti, \emph{Topology of the complement of real hyperplanes in $\mathbb{C}^N$}, Invent.\ Math.\ \textbf{88} (1987), no.~3, 603--618.


\bibitem[SY]{SY}
Y.~Sekiya and K.~Yamaura, \emph{Tilting theoretical approach to moduli spaces over preprojective algebras}, Algebr.\ Represent.\ Theory \textbf{16} (2013), no.~6, 1733--1786.

\bibitem[T]{T08}
Y.~Toda, \emph{Stability conditions and crepant small resolutions}, Trans.\ Amer.\ Math.\ Soc.\ \textbf{360} (2008), no.~11, 6149--6178.

 \bibitem[VdB]{VdB}
M.~Van den Bergh, \emph{Three-dimensional flops and noncommutative rings}, 
Duke Math.\ J.\ \textbf{122}  (2004), no.~3, 423--455.


\bibitem[W1]{HomMMP}
M.~Wemyss, \emph{Flops and Clusters in the Homological Minimal Model Program}, Invent.\ Math.\ \textbf{211} (2018), no.~2, 435--521.

\bibitem[W2]{Kinosaki}
M.~Wemyss, \emph{Autoequivalences for $3$-fold flops: an overview}, Proceedings of Kinosaki Symposium on Algebraic Geometry (2018), 105--112.

\end{thebibliography}
\end{document}